\documentclass[11pt,oneside]{amsart}

\usepackage{graphicx, amsmath, amssymb, amscd, mathtools, hyperref, amsthm, euscript, 
amsfonts,bm,color}  
\DeclareMathAlphabet{\mathpzc}{OT1}{pzc}{m}{it}

\newtheorem{thm}[equation]{Theorem}
\newtheorem{theorem}[equation]{Theorem}

\newtheorem{Rmk}[equation]{Remark}\newtheorem{remark}[equation]{Remark}

\newtheorem{Prop}[equation]{Proposition}
\newtheorem{prop}[equation]{Proposition}\newtheorem{proposition}[equation]{Proposition}

\newtheorem{cor}[equation]{Corollary}
\newtheorem{Coro}[equation]{Corollary}
\newtheorem{lem}[equation]{Lemma}\newtheorem{Lem}[equation]{Lemma}
\newtheorem{lemma}[equation]{Lemma}
\newtheorem{Def}[equation]{Definition}
\newtheorem{definition}[equation]{Definition}

\numberwithin{equation}{section}

\numberwithin{equation}{section}

\newcommand{\be}{begin{equation}}

\newcommand{\e}{\epsilon}
\newcommand{\z}{\mathbb{Z}}

\newcommand{\N}{\mathbb{N}}

\newcommand{\br}{\mathbb{R}}

\newcommand{{\grinv}}{{\Cal G}^{-r}}

\newcommand{\ba}{\backslash}

\newcommand{\G}{\Gamma}

\newcommand{\Cal}{\mathcal}

\newcommand{\la}{\lambda}
\newcommand{\ra}{\rangle}

\newcommand{\bp}{\begin{pmatrix}}
\newcommand{\ep}{\end{pmatrix}}
\renewcommand{\be}{\begin{equation}}
\newcommand{\ee}{\end{equation}}

\renewcommand{\bp}{{\rm bp}}

\renewcommand{\O}{\operatorname{O}}

\renewcommand{\L}{\Cal L}

\newcommand{\PS}{\op{PS}}

\newcommand{\norm}[1]{\lVert #1 \rVert}
\newcommand{\abs}[1]{\lvert #1 \rvert}

\newcommand{\op}{\operatorname}\newcommand{\supp}{\operatorname{supp}}

\newcommand{\BR}{\operatorname{BR}}
\newcommand{\BMS}{\operatorname{BMS}}

\newcommand{\Om}{\Omega}

\newcommand{\ga}{\gamma}
\newcommand{\F}{\mathcal F}
\renewcommand{\i}{\op{i}}

\def\scrC{{\mathcal C}}

\def\e{\mathrm{e}}
\def\i{\mathrm{i}}

\def\O{\operatorname{O{}}}

\def\supp{\operatorname{supp}}

\newcommand{\La}{\Lambda}

\newcommand{\fg}{\mathfrak{g}}
\newcommand{\fp}{\mathfrak{p}}
\newcommand{\fa}{\mathfrak{a}}

\newcommand{\fk}{\mathfrak{k}}

\newcommand{\Ga}{\Gamma}\newcommand{\bb}{\mathbb}
\newcommand{\cal}{\mathcal}
\renewcommand{\e}{\varepsilon}
\renewcommand{\epsilon}{\e}

\newcommand{\dg}{D_\Gamma^{\star}}
\newcommand{\inte}{\op{int}}

\begin{document}

\title[Invariant measures for horospherical actions]{
Invariant measures for horospherical actions and Anosov groups.}

\author{Minju Lee}
\address{Mathematics department, Yale university, New Haven, CT 06520, Current address: Institute for Advanced Study, Princeton, NJ 08540}
\email{minju@ias.edu}
\author{Hee Oh}
\address{Mathematics department, Yale university, New Haven, CT 06520 and Korea Institute for Advanced Study}
\email{hee.oh@yale.edu}
\begin{abstract} Let $\Gamma$ be a Zariski dense Anosov subgroup of a connected semisimple real algebraic group  $G$.
For a maximal horospherical subgroup $N$ of $G$,
we show that the space of all non-trivial $NM$-invariant ergodic and $A$-quasi-invariant
Radon measures  on $\Ga\ba G$, up to proportionality,
 is  homeomorphic to $\br^{\text{rank}\,G-1}$, where
  $A$ is a maximal real split torus and $M$ is a maximal compact subgroup that normalizes $N$.
 One of the main ingredients is to establish the $NM$-ergodicity of all Burger-Roblin measures.
\end{abstract}

\thanks{Oh is supported in part by NSF grants}
\maketitle

\tableofcontents

\section{Introduction}
Let $G$ be a connected semisimple real algebraic group and $\Gamma<G$ a Zariski dense discrete subgroup. A subgroup $N$ of $G$ is called {\it horospherical} if there exists a diagonalizable element $a\in G$ such that
$$N=\{g\in G: a^k g a^{-k}\to \infty\quad \text{as $k\to +\infty$}\},$$
or equivalently, $N$ is the unipotent radical of a parabolic subgroup of $G$.
We assume that $N$ is a maximal horospherical subgroup, which exists uniquely up to conjugation.
We are interested in the measure rigidity property of the $N$-action on the homogeneous space
$\Gamma\ba G$. When $\Gamma$ is a lattice, i,e., when $\Gamma\ba G$ has finite volume, the well-known measure rigidity theorems of Furstenberg \cite{Fu}, Veech \cite{Vee} and Dani \cite{Da}
give a complete classification of Radon measures (=locally finite Borel measures) invariant by $N$.
 This rigidity phenomenon extends to any unipotent subgroup action by the celebrated theorem of Ratner in \cite{Ra}.

When $G$ has rank one and $\Gamma$ is geometrically finite,
 the horospherical subgroup 
action on $\Gamma\ba G$ 
is known to be essentially uniquely ergodic; there exists a  unique non-trivial
invariant ergodic Radon measure  on $\Gamma\ba G$, called the Burger-Roblin measure (\cite{Bu}, \cite{Ro}, \cite{Win}). When $\Gamma$ is geometrically infinite, there may be a continuous family of horospherically invariant ergodic measures
as first discovered by Babillot and Ledrappier (\cite{Bab}, \cite{BL}). For a certain class of geometrically infinite groups,
a complete classification of horospherically invariant ergodic measures has been obtained; see \cite{Sa1}, \cite{Sa2}, \cite{Led}, \cite{OP}, \cite{LL}, \cite{L}, etc. We refer to a recent article by Landesberg and Lindenstrauss \cite{LL} for a more precise description on the rank one case.

 When $G$ has rank at least $2$ and $\Gamma$ has infinite co-volume in $G$, very little is known
 about invariant measures. In this paper, we focus on a special class of discrete subgroups, called {\it  Anosov} subgroups.
In the rank one case, this class coincides with the class of  convex cocompact subgroups, and hence
the class of Anosov subgroups can be considered
as  a generalization of convex cocompact subgroups of  rank one  groups to higher rank. The works of Burger \cite{Bu2} and Quint \cite{Quint2} on a higher rank version of the Patterson-Sullivan theory supply  a continuous family of maximal horospherically 
invariant Burger-Roblin measures, as was introduced in \cite{ELO}.
We show that all of these Burger-Roblin measures are ergodic for maximal horospherical foliations, and classify all ergodic {\it non-trivial} Radon measures for maximal horospherical foliations, which are also quasi-invariant under Weyl chamber flow. In particular, we establish a homeomorphism
 between the space of
these measures and the interior of the projective limit cone of $\Gamma$, which is again homeomorphic to
$\br^{\text{rank }G-1}$.

In order to formulate our main result precisely, we begin with the definition of an Anosov subgroup of $G$.
 Let $P$ be the normalizer of $N$, i.e.,
 a minimal parabolic subgroup of $G$ and $\cal F:=G/P$ the Furstenberg boundary.
We denote by $\cal F^{(2)}$ the unique open $G$-orbit in $\cal F\times \cal F$. 
A Zariski dense discrete subgroup $\Ga<G$ is called an {\it Anosov subgroup} (with respect to $P$) if it is a  finitely generated word hyperbolic group which admits a $\Ga$-equivariant embedding $\zeta$ of the Gromov boundary $\partial \Ga$  into $\cal F$ such that $(\zeta(x),\zeta(y))\in\cal F^{(2)}$ for all $x\ne y$ in $\partial\Ga$. 

First introduced by Labourie \cite{La} as the images of
 Hitchin representations of surface groups (\cite{Hi}, \cite{FG}), this definition  is due to Guichard and Wienhard \cite{GW}, who showed that
 Anosov subgroups (more precisely, Anosov representations) form an open subset in the representation variety $\op{Hom}(\Gamma, G)$. The class of Anosov groups include Schottky subgroups \cite{Q4} and hence any
 Zariski dense discrete subgroup of $G$ contains an Anosov subgroup (\cite{Ben}, \cite{Q5}).
We also refer to the work of Kapovich, Leeb and Porti \cite{KLP} for other equivalent characterizations of Anosov groups, as well as to
the
excellent survey articles by Kassel \cite{Kas} and Wienhard \cite{Wie} on higher Teichm\"uller theory.

We let $P=NMA$ be the Langlands decomposition of $P$, so that $N$ is the unipotent radical of $P$, 
$A$ is a maximal real split torus of $G$, and $M$
is a compact subgroup that commutes with $A$. Note that any maximal horospherical subgroup arises in this way, i.e., as the unipotent radical of a minimal parabolic subgroup.

 The limit set $\La$ of $\Gamma$ is the unique minimal $\Gamma$-invariant closed subset of $\cal F$.
Hence the following set
$$\mathcal E:=\{[g] \in \Gamma\ba G:  gP \in \Lambda \}$$
is the unique minimal $P$-invariant closed subset of $\Ga\ba G$.
We call a $P$-quasi-invariant measure on $\Ga\ba G$ {\it non-trivial} if its support is contained in $\mathcal E$.
\begin{thm} \label{mmm} For any Anosov subgroup $\Gamma<G$,
the space $\cal Q_\Gamma$ of all non-trivial $NM$-invariant ergodic and $A$-quasi-invariant
 Radon measures on $\Ga\ba G$, up to constant multiples, is homeomorphic to
 $ \br^{\text{rank}\,G-1}$. 
\end{thm}

In order to describe the explicit homeomorphism, we  need to define Burger-Roblin measures on $\cal E$.
Denote by $\fa$ the Lie algebra of $A$ and fix a positive Weyl chamber $\fa^+\subset\fa$
so that $\log N$ is the sum of positive root subspaces.
Fix a maximal compact subgroup $K$ of $G$ so that the Cartan decomposition $G=K(\exp \fa^+) K$ holds. 
Let $\mu:G\to \fa^+$ denote the Cartan projection map (Def. \ref{Cartan}).
We denote  by $\cal L_\Gamma\subset \fa^+$ the limit cone of $\Gamma$, that is the asymptotic cone of 
$\mu(\Gamma)$ (Def. \ref{lc}). 
Let $\psi_\Ga : \fa\to\bb R\cup\{-\infty\}$ denote the growth indicator function 
of $\Gamma$ (Def. \ref{def.GI}). 

For Anosov subgroups, the following two spaces are homeomorphic to each other:
$$   D_\Ga^\star :=
\{\psi\in \fa^*: \psi \ge \psi_\Gamma,  \psi(v)=\psi_\Gamma(v) \text{ for some
$v\in \inte \L_\Ga$}\}\simeq  \inte (\mathbb P \L_\Ga) $$
where $ \inte(\mathbb P\L_\Ga)$ denotes the interior of the projective limit cone
$\bb P\cal L_\Ga$ (Proposition \ref{homeo}). Since $\inte ( \L_\Ga)$ is a non-empty open convex cone
of $\fa^+$ \cite[Thm. 1.2]{Ben}, it follows that $\dg $ is homeomorphic to
$  \br^{\text{rank}\,G-1}$.
We remark that $D_\Ga^\star$ is in fact a closed analytic submanifold of $\mathfrak a^*$ \cite[Prop. 4.11]{PS}.

 For a linear form $\psi\in \fa^*$, a Borel probability measure $\nu$ on the limit set
 $\La$ is called a $(\Ga, \psi)$-Patterson Sullivan measure if 
 for all $\ga\in \Ga$ and $\xi\in \cal F$,
 \be\label{gc000} \frac{d\ga_* \nu}{d\nu}(\xi) =e^{\psi (\beta_\xi (o, \gamma o))}\ee
 where   $o=[K]\in G/K$ and
$\beta:\cal F\times G/K\times G/K\to\fa$ denotes the $\fa$-valued Busemann function 
(Def. \ref{Bu}).  Quint constructed a $(\Gamma, \psi)$-Patterson-Sullivan measure 
for each $\psi\in D_\Ga^\star$ \cite{Quint2}; for $\Gamma$ Anosov, 
this measure exists uniquely (hence $\Gamma$-ergodic), which we denote by $\nu_\psi$ (see Theorem \ref{pop} and references therein).

In the rest of the introduction, we let $\Gamma <G$ be an Anosov subgroup. 
By a Patterson-Sullivan measure on $\La$, we mean a $(\Gamma, \psi)$-Patterson-Sullivan measure on $\La$ for some
$\psi\in \fa^*$.
We show: \begin{thm}\label{con-class}
The map $\psi\mapsto \nu_\psi$ is a homeomorphism between $\dg$ and the space of all Patterson-Sullivan measures on $\La$. Moreover, Patterson-Sullivan measures are pairwise mutually singular.
\end{thm}
When $\G$ is the fundamental group of a closed negatively curved manifold, the above theorem also follows from
\cite{Led2}.
  
 We also denote by $ \nu_\psi$ the $M$-invariant lift of $\nu_\psi$ on $\cal F\simeq
 K/M$ to $K$ by abuse of notation.
The Burger-Roblin measure $m^{\BR}_\psi$ on $\Gamma\ba G$ is induced from the following $\Gamma$-invariant measure
$\tilde m^{\BR}_\psi$ on $G$: for $g=k(\exp b) n\in KAN$,
\be\label{introbr} d\tilde m^{\BR}_\psi (g)=e^{\psi (b)} dn \,db\, d\nu_\psi (k) \ee
where $dn$ and $db$ are Lebesgue measures on $N$ and $\fa$ respectively.

\medskip

The following is a more elaborate version of Theorem \ref{mmm}:
\begin{thm}[Classification] \label{main2} 
The map 
$\psi\mapsto [m^{\BR}_\psi]$ defines a homeomorphism  between $\dg $ and $ \cal Q_\Ga$. 
\end{thm}

While the $P$-ergodicity of $m_\psi^{\BR}$ follows from the $\Ga$-ergodicity of $\nu_\psi$, 
establishing the ergodicity of $m_\psi^{\BR}$, and hence  
the well-definedness of the above map is the most significant part of Theorem \ref{main2}:
\begin{thm}[Ergodicity] \label{main1}
 For each $\psi\in \dg$,  $m_{\psi}^{\BR}$ is $NM$-ergodic.
\end{thm}

Since $\cal E$ is a second countable topological space, Theorem \ref{main1} implies:
\begin{cor}
For $m_\psi^{\BR}$ almost all $x\in \cal E$,
$$\overline{xNM}=\cal E.$$
\end{cor}
A Radon measure $m$ on $\Ga\ba G$ is called $P$-semi-invariant if there exists a character $\chi : P\to\bb R^*$ such that $p_*m=\chi(p)m $ for all $p\in P$.
Note that any $P$-semi-invariant Radon measure is necessarily $NM$-invariant since $NM$ is unimodular.
We show that any $P$-semi-invariant Radon measure on $\cal E$ is of the form $m_\psi^{\BR}$ for some $\psi\in\dg$ (Proposition \ref{prop.NMA}).
Hence Theorem \ref{main1} implies:
\begin{cor}
The space of all $P$-semi-invariant Radon measures on $\cal E$ coincides with $\cal Q_\Ga$, up to constant multiples.\end{cor}

\subsection*{ Discussion on the proof of Theorem \ref{main1}.} Fix $\psi\in \dg$. Defining a $\Ga$-invariant Radon measure $\widehat\nu_\psi$ on $\cal H:=G/NM\simeq \cal F\times\mathfrak a$ by
$$
d\widehat\nu_\psi (gP, b)=e^{ \psi(b)} d\nu_\psi(gP)\,db,
$$
the standard duality theorem implies that
the $NM$-ergodicity of $m_\psi^{\BR}$ is equivalent to the $\G$-ergodicity of $\widehat \nu_\psi$.

Generalizing the observation of Schmidt \cite{Sch2} (also see \cite{Ro}) to a higher rank situation,
the $\G$-ergodicity of $\widehat \nu_\psi$ follows if 
the closed subgroup, say ${\mathsf E}_{\nu_\psi}={\mathsf E}_{\nu_\psi}(\Ga)$, consisting of all
 $\nu_\psi$-essential values  is equal to $\fa$ (Proposition \ref{prop.erg1}):

\begin{definition}\label{esso} \rm An element $v\in\mathfrak a$ is called   a $(\Ga, \nu_\psi)$-\textit{essential value}, if for any $\e>0$ and
any
Borel set $B\subset \cal F$ with $\nu_\psi
(B)>0$, there exists $\ga\in\Ga$ such that 
$$
B\cap\ga B\cap\{\xi\in \cal F : \norm{\beta_\xi(o,\ga o)-v}<\e\}
$$
has a positive $\nu_\psi$-measure.
\end{definition}

Recalling that the
 Jordan projection $\lambda(\Gamma)$ of $\Gamma$ generates a dense subgroup of $\fa$ \cite{Ben},
the following is the main ingredient of our proof of Theorem \ref{main1}:
\begin{prop}\label{mt}
For each $\psi\in \dg$, there exists a finite subset
$F_\psi\subset \la(\Gamma)$ such that 
$$\lambda(\Ga) -F_\psi\subset {\mathsf E}_{\nu_\psi}(\Ga).$$
In particular, $ {\mathsf E}_{\nu_\psi}(\Ga)=\fa$.
\end{prop}
 See Proposition \ref{posm} for a more general version stated for any Zariski dense normal subgroup of $\Ga$.
 
Among other things, the following three key properties of Anosov groups play important roles in
the proof of Proposition \ref{mt}:
\begin{enumerate}
\item (Antipodality) $ \La \times \La -\{(\xi, \xi)\} \subset \cal F^{(2)}$;
\item (Regularity) If $\gamma_i\to \infty$ in $\Gamma$, then $\alpha (\mu(\gamma_i)) \to \infty$ for each simple root
$\alpha$ of $\op{Lie}(G)$ with respect to $\fa^+$;
\item (Morse property) There exists a constant $D>0$ such that any discrete geodesic ray $[e,x)$ 
in $\Gamma$ tending to $x\in \partial \Gamma$
 is contained in the $D$-neighborhood
of some $gA^+$ in $G$ where $g\in G$ satisfies $gP=\zeta(x)$.
\end{enumerate} (1) is a part of the definition of an Anosov subgroup. (2) follows from the fact that $\cal L_\Gamma\subset \inte \fa^+\cup \{0\}$ (\cite{Q4}, \cite{Samb2}, \cite{BCLS})
in view of Lemma \ref{st}.
(3) is proved in \cite{KLP} (see also Proposition \ref{Morse}).

 Many aspects of our proof
of  Proposition \ref{mt} can be simplified  for a special class of $\psi\in \dg$ with certain strong positivity property (cf. Lemma \ref {lem.SP}); however as our eventual goal is the classification theorem as stated in Theorem \ref{mmm}, we need to address all $\psi\in \dg$ which makes the proof much more intricate and requires the full force of the Anosov property of $\Gamma$.
 
  Fix $\ga_0\in \Ga$.  We aim to show that
 $\la(\ga_0)\in \mathsf E_{\nu_\psi}(\Ga)$. More precisely,
 for any $\e>0$ and any Borel subset $B\subset \cal F$ with $\nu_\psi(B)>0$, there exists $\gamma\in \Gamma$ such that
\be\label{ob}
\nu_\psi( B\cap\ga \ga_0\ga^{-1} B\cap\{\xi\in \cal F : \norm{\beta_\xi(o,\ga \ga_0 \ga^{-1}  o)-\la (\ga_0)}<\e\}) >0.
\ee
 For $p\in G/K$, we define  
  $$d_{\psi, p} (\xi_1, \xi_2)=e^{-[\xi_1, \xi_2]_{\psi, p}}$$
for any $\xi_1\ne \xi_2$ in $\La$, where $[\cdot, \cdot]_{\psi, p}$ denotes the $\psi$-Gromov product based at $p$ (Def. \ref{def.GP0}).
Its well-definedness is due to the antipodality (1). In the rank one case,
this is simply the restriction of the classical visual metric  to the limit set $\La$. In general,
it is not even symmetric but
we show that  any sufficiently small  power of $d_{\psi, p}$ is comparable
 to some genuine metric on $\La$:
\begin{theorem}\label{prop.met0}
For all sufficiently small $s>0$, there exist a metric $d_s$ on $\La$ and $C_s>0$ such that
for all $\xi_1\neq\xi_2$ in $\La$,
$$
C_s^{-1} d_s(\xi_1,\xi_2 )\leq  d_{\psi,p}(\xi_1,\xi_2)^s\leq   C_s
d_s(\xi_1,\xi_2).
$$
\end{theorem}

\begin{Rmk}\rm {In the process of proving this theorem, we also show that the Gromov product on $\partial \Gamma$
and the $\psi$-Gromov product $[\cdot, \cdot]_{\psi, p}$ are equivalent to each other (see
Theorem \ref{GG}).}\end{Rmk}

As a consequence  of Theorem \ref{prop.met0}, $d_{\psi, p}$ can be used to define virtual balls 
 with respect to which Vitali type covering lemma can be applied (Lemma \ref{inc}).
Let $\xi_0\in \cal F$ denote the attracting fixed point of $\ga_0$ and
consider
 the family $$D(\ga \xi_0, r):=
 \bb B_p (\ga \xi_0, \frac{1}{3} e^{-\psi (\underline a(\gamma^{-1}p, p) + \op{i} \underline a(\ga^{-1}p, p))} r ),
\, \ga\in \Ga, \, r>0$$ where $\underline a(q,p)$ denotes the $\fa^+$-valued distance from $q$ to $p$ (Def. \ref{under}).
We then show that for all sufficiently small $r>0$, there are infinitely many $D(\ga_i \xi_0, r)$ satisfying \eqref{ob} (Lemma \ref{spn}). The key ingredient in this step is 
the following:
\begin{lem} \label{max00} There exists $C=C( \psi, p)>0$ such that for all $\ga \in \Ga$ and $\xi\in \La$,
$$
-\psi(\underline a (p, \ga p))-C
 \le  \psi(\beta_\xi(\ga p, p))\leq  \psi(\underline a (\ga p,  p))+C.
$$
\end{lem}
In the rank one case,  a stronger statement 
$-d(p,q)\le \beta_{\xi}(q,p) \le d(p,q)$  holds for all $q, p\in G/K$ and $\xi\in \cal F$, which generalizes to
 strongly positive linear forms (Lemma \ref{lem.SP}).
For a general $\psi\in \dg$, our proof of Lemma \ref{max00} is based on the property that
 the orbit map $\gamma\mapsto \gamma (o)$ sends a shadow in the word hyperbolic group $\Ga$ to a shadow in the symmetric space $G/K$ (Proposition \ref{prop.SS}), as well as
 the following lemma, which is of independent interest: we denote by $|\cdot|$ the word length on $\Gamma$ with respect to a fixed finite symmetric generating subset.
 \begin{lem}
There exists $R>0$ such that for any $\gamma_1, \ga_2\in \Ga$ with $|\ga_1 \ga_2|=|\ga_1|+|\ga_2|$,
we have
$$\|\mu(\ga_1 \ga_2) -\mu(\ga_1)-\mu(\ga_2)\| <R.$$
 \end{lem}
 We emphasize that this lemma does not follow from the property of Anosov groups
  that $(\Gamma, |\cdot| )\to G$   
  is a quasi-isometric embedding \cite[Thm. 1.7]{GW}, due to the non-trivial multiplicative constant.
 
To establish \eqref{ob}, we  approximate a general Borel subset
$B\subset \F$ by some $D(\ga \xi_0, r)$ satisfying  \eqref{ob}. 
In this step, we prove the following higher rank generalization of Tukia's theorem \cite[Thm. 4A]{Tuk} (see also \cite{Myr}, \cite{Aga}, \cite{Nak}):
 \begin{thm}\label{Mintro}
For any Patterson-Sullivan measure $\nu$ on $\La$,
the set of Myrberg limit points (Def. \ref{def.MYR}) has full $\nu$-measure.
\end{thm}
 
It follows that for the $AM$-invariant Bowen-Margulis-Sullivan measure $m_\psi^{\BMS}$ on $\Gamma\ba G$, almost all points have dense $A^+M$ orbits (Corollary \ref{cor.AM}). 
Using the property that virtual balls $\bb B_p(\gamma \xi_0, r)$ satisfy a covering lemma (Lemma \ref{inc}) that is a consequence of Theorem \ref{prop.met0}, we show that $\nu_\psi$-almost all Myrberg limit points satisfy the Lebesgue density type statement for the family $\{D(\ga \xi_0, r):\ga \in \Ga, r>0\}$ (Proposition \ref{lem.sh}). By Theorem \ref{Mintro}, this gives a desired approximation of $B$ by some $D(\ga \xi_0, r)$ satisfying  \eqref{ob}. 

We finally remark that in our subsequent work \cite{LO}, we present refined versions of Theorems \ref{mmm} and \ref{main1}
 building on the main results of this paper.
\medskip

\noindent{\bf Organization:}
In section 2, we go over basic definitions and properties of Zariski dense discrete subgroups of $G$.
In section 3, we discuss the notion of $\fa$-valued Gromov product and define the generalized BMS measures for a pair of $(\Ga,\psi)$-conformal densities on $\cal F$. From section 4, we assume that
 $\Gamma$ is Anosov.
In section 4, we observe that the BMS 
measure
$m^{\BMS}_\psi$ is $AM$-ergodic for each $\psi\in \dg$.
 Sections 5 and 6 are devoted to proving Lemma \ref{max00}
and Theorem \ref{prop.met0} respectively. In section 7, we prove that the space of PS-measures on $\La$
is homeomorphic to $\dg$, which is the first part of Theorem \ref{con-class}.  In section 8, we show that the set of Myrberg limit points of $\G$ has full measure
for any $\PS$-measure on $\La$.  In section 9, we discuss the relation between the set of essential values of $\nu_\psi$ and the $NM$-ergodicity of $m^{\BR}_\psi$. In the final section 10, we prove Theorems \ref{main1}, \ref{main2} and the second part of Theorem \ref{con-class}.

\medskip 
{\it{Added to proof}:} It was recently shown that for any $\psi\in \dg$,
any $\psi$-conformal measure is necessarily supported on $\La$, and hence there exists a unique $(\Ga, \psi)$-conformal measure on $\F$ (not only on $\La$),
first in \cite{ELO2} for ranks at most $3$ and in \cite{LO2} for general ranks. As a consequence,  the space of
all $\Ga$-Patterson-Sullivan measures on $\La$ in Theorem \ref{con-class} is equal to
the space of all $(\Ga,\psi)$-conformal measures on $\F$, $\psi\in \dg$.

 \section{Limit set and Limit cone.}\label{ps}
Let $G$ be a connected, semisimple real algebraic group with finite center, and $\Gamma <G$ be a Zariski dense discrete subgroup. We fix, once and for all, a Cartan involution $\theta$ of the Lie algebra $\mathfrak{g}$ of $G$, and decompose $\fg$ as $\mathfrak g=\mathfrak k\oplus\mathfrak{p}$, where $\fk$ and $\fp$ are the $+ 1$ and $-1$ eigenspaces of $\theta$, respectively. 
We denote by $K$ the maximal compact subgroup of $G$ with Lie algebra $\fk$, and by $X=G/K$ the associated symmetric space.
We also choose a maximal abelian subalgebra $\fa$ of $\mathfrak p$.
Choosing a closed positive Weyl chamber $\fa^+$ of $\fa$, let $A:=\exp \mathfrak a$ and $A^+=\exp \mathfrak a^+$. The centralizer of $A$ in $K$ is denoted by $M$, and we set 
$N$ to be the contracting horospherical subgroup: for  $a\in \inte A^+$,
  $N=\{g\in G: a^{-n} g a^n\to e\text{ as $n\to +\infty$}\}$.
  Note that $\log   N $ is the sum of all positive root subspaces for our choice of $A^+$.
  Similarly, we also consider the expanding horospherical subgroup $N^+$:
for  $a\in \inte A^+$,  $N^+:=\{g\in G: a^n g a^{-n}\to e\text{ as $n\to +\infty$}\}$.
We set $$P^+=MAN^+,\quad\text{and} \quad P=P^-=MAN;$$ they are minimal parabolic subgroups of $G$ that are opposite to each other.
The quotient $\F=G/P$ is known as the Furstenberg boundary of $G$, and is isomorphic to $K/M$.

Let $\op N_K(\fa)$  be the normalizer of $\fa$ in $K$.
Let $\cal W:=\op N_K(\fa)/M$ denote the Weyl group. Fixing a left $G$-invariant and right $K$-invariant Riemannian metric on $G$ induces a $\cal W$-invariant inner product on $\mathfrak a$, which we denote by $\langle \cdot,\cdot\rangle$.  The identity coset $[e]$ in $G/K$ is denoted by $o$.

 \medskip 

Denote by $w_0\in \cal W$ the unique element in $\cal W$ such that $\op{Ad}_{w_0}\fa^+= -\fa^+$; it is the longest Weyl element. Note that
$w_0Pw_0^{-1}=P^+$.
\begin{Def}[Visual map]\rm   For each $g\in G$, we define 
   $$g^+:=gP\in G/P\quad\text{and}\quad g^-:=gw_0P\in G/P.$$
For all $g\in G$ and $m\in M$, observe that  $g^{\pm}=(gm)^{\pm}=g(e^{\pm})$.
 Let $\F^{(2)}$ denote the unique open $G$-orbit in $\F\times \F$:
$$\F^{(2)}=G(e^+, e^-)=\{(g^+, g^-)\in \cal F\times \cal F: g\in G\}.$$ 
Note that the stabilizer of $(e^+, e^-)$ is the intersection $P^-\cap P^+=MA$.

\end{Def}

 We say that $\xi,\eta\in\cal F$ are in general position if $(\xi,\eta)\in\cal F^{(2)}$.
 The Bruhat decomposition says that $G$ is the disjoint union $\cup_{w\in \cal W}Nw P^+$, and
$NP^+$ is Zariski open and dense in $G$.
Hence $(\xi, \eta)\notin \cal F^{(2)}$ if and only if $(\xi, \eta)\in G (e^+, we^-)$ for some $w\in \cal W-\{e\}$.

\subsection*{Cartan projection and $\fa^+$-valued distance}

\begin{Def} [Cartan projection]\label{Cartan} \rm
For each $g\in G$, there exists a unique element $\mu(g)\in \mathfrak a^+$, called the Cartan projection of $g$, such that
\begin{equation*}
g\in K\exp(\mu(g))K.
\end{equation*}
\end{Def}

When $\mu(g)\in \inte \fa^+$ and $g=k_1 \exp \mu(g)k_2$, $k_1, k_2$ are determined uniquely up to mod $M$, more precisely,
if $g=k_1'\exp \mu(g)k_2'$, then for some $m\in M$, $k_1=k_1'm$ and $k_2=m^{-1}k_2'$. We write 
$$\kappa_1(g):=[k_1]\in K/M\quad \text{ and } \quad \kappa_2(g):=[k_2]\in M\ba K.$$

\begin{lem} \label{com}\cite[Lem. 4.6]{Ben} For any compact subset $L\subset G$, there exists a compact subset 
$Q=Q(L)\subset \fa$ such that for all $g\in G$,
$$\mu(L g L)\subset \mu(g)+ Q.$$

\end{lem}
 \begin{definition}[$\fa^+$-valued distance]\label{under}\rm
We define $\underline a : X\times X\to\mathfrak a^+$ by 
$$\underline a(p,q):=\mu(g^{-1}h )$$
where $p=g(o)$ and $q=h(o)$.
\end{definition}

\subsection*{Accumulation of points of $X$ on $\cal F$}
Let $\Pi$ denote the set of all simple roots of $\mathfrak g$ with respect to $\fa^+$.

\begin{Def} \rm
 We write that  
\begin{enumerate}
\item
 $v_i\to \infty$ regularly in $ \fa^+$  if  $\alpha(v_i)\to\infty$ as $i\to\infty$ for all $\alpha\in\Pi$;
\item
$a_i\to \infty$ regularly in $ A^+$  if  $\log a_i\to \infty$ regularly in $ \fa^+$;
\item
$g_i\to \infty$ regularly in $G$  if  $\mu(g_i)\to\infty$ regularly in $ \fa^+$.
\end{enumerate} 
 \end{Def}
If $a_i\to \infty$ regularly in $A^+$, then for all $n\in N^+$, 
$$\lim_{i\to \infty} a_i n a_i^{-1}= e$$ uniformly on compact subsets of $N$.

\begin{Def}\label{Gr} \rm
We call $\Gamma$ {\it regular} if  for any sequence $\ga_i\in \Ga$
going to $\infty$ in $G$,
$\ga_i\to \infty$ regularly in $G$.
\end{Def}

\begin{lem}\label{gr}
If the closure of $\{(\xi_i, e^-):i=1,2,\cdots\}$  is contained in $\cal F^{(2)}$, then $a_i \xi_i\to e^+$ for any  sequence $a_i\to \infty$ regularly in $A^+$.
\end{lem}
\begin{proof}
 The hypothesis implies that $\xi_i=n_i e^+$ for a bounded sequence $n_i\in N^+$.
Hence
$a_i\xi_i= a_i n_i e^+ =(a_i n_i a_i^{-1}) e^+\to e^+$
as  $a_i\to \infty$ regularly in $A^+$.
\end{proof}

\begin{Def}\rm
\begin{enumerate}
\item A sequence $g_i\in G$ is said to converge to $\xi\in\cal F$, if $g_i\to\infty$ regularly in $G$ and $\lim\limits_{i\to\infty}\kappa_1(g_i)^+=\xi$.
\item A sequence $p_i=g_i(o) \in X$ is said to converge to $\xi\in \cal F$ if $g_i$ does. 
\end{enumerate}
\end{Def}

\begin{lem}\label{gen} Consider a sequence $g_i=k_ia_i h_i^{-1}$ where $k_i\in K, a_i\in A^+, h_i\in G$
satisfy that $k_i^+\to k_0^+$ in $K$, $h_i\to h_0$ in $G$, and $a_i\to \infty$ regularly in $A^+$. 
Then for any $\xi\in\cal F$ in general position with $h_0^-$, we have
$$\lim\limits_{i\to\infty}g_i\xi= k_0^+.$$\end{lem}
\begin{proof} As $(\xi, h_0^-)\in \cal F^{(2)}$, we have
 $(h_0^{-1}\xi,e^-)\in\cal F^{(2)}$. Since $\cal F^{(2)}$ is open and $h_i^{-1}\xi \to h_0^{-1}\xi$,
we have $(h_i^{-1}\xi ,e^-)\in \cal F^{(2)} $  for all large $i$. 
By Lemma \ref{gr}, $a_i h_i^{-1}\xi\to e^+$ as $i\to \infty$.
Therefore $\lim_{i\to\infty} g_i \xi = \lim_{i\to\infty} k_i (a_i h_i^{-1} \xi)=k_0^+$.
\end{proof}

\begin{lem}  \label{ssame} If $g_i\in G$ converges to $\xi\in \cal F$ and $p_i\in X$ is bounded, then $\lim_{i\to \infty}
g_ip_i= \xi$.
\end{lem}
\begin{proof}
Write $g_i=k_i a_i \ell_i^{-1}\in KA^+K$. The hypothesis implies that $a_i\to \infty$ regularly in $A^+$ and 
$k_i^+\to \xi$ as $i\to \infty$. Let $k_0\in K$ be
such that $k_0^+=\xi$, and $g_i'\in G$ be such that $g_i'(o)=p_i$. Write $g_i g_i'=k_i'a_i' (\ell_i')^{-1}\in KA^+K$.
We need to show that $\lim_{i\to \infty} k_i' = k_0^+$. As $k_i^+\to k_0^+$, it suffices to show that any limit of the sequence
$k_i^{-1}k_i'$ belongs to $M=\op{Stab}_Ke^+$.

Set $q_i:=k_i^{-1}k_i'$.  Let $q$ be a limit of  the sequence $q_i$. By passing to a subsequence, 
we may suppose $q_i \to q\in K$.
Since  $d(o, p_i)=d(g_i o, g_i p_i) = d(a_io,  q_i a_i' o) $,
the sequence  $h_i^{-1}:=a_i^{-1} q_i a_i' $  is bounded.
Passing to a subsequence, assume that $h_i$ converges to some $h_0\in G$ as $i\to\infty$.
Choose $\eta\in \cal F$ that is in general position with both $h_0^-$ and $e^-$. Then
$\lim_{i\to \infty} a_i h_i^{-1}\eta =e^+$ and $\lim_{i\to \infty} q_i a_i' \eta=q^+$ by Lemma \ref{gen}.
Since $
a_ih_i^{-1}\eta=q_ia_i'\eta
$, we get $e^+=q^+=q (e^+)$.
This implies $q\in \op{Stab}_K e^+=M$.
 \end{proof}

\begin{lem}\label{st2}
If  $g_i\to g$ in $G$ and $a_i\to \infty$ regularly in $A^+$, then  for any $p\in X$,
$\lim_{i\to \infty} g_ia_i(p)= g^+$ and
$\lim_{i\to \infty} g_ia_i^{-1}(p)= g^-$.
\end{lem}
\begin{proof} By Lemma \ref{ssame}, it suffices to consider the case when $p=o$.
Write $g_ia_i=k_ib_i\ell_i^{-1}\in KA^+K$. 
As the sequence $g_i$ is bounded,
it follows from Lemma \ref{com}
that  $b_i\to \infty$ regularly in $A^+$.  
In order to show that $g_ia_i(o)\to g^+$, it suffices to show that if $k_i\to k_0$, then $k_0^+=g^+$.
By passing to a subsequence, we may assume that $\ell_i\to \ell_0$ in $K$.
Choose $\xi\in \cal F$ that is in general position with both $\ell_0^{-}$ and $e^-$. Then $g_ia_i \xi \to k_0^+$ by Lemma \ref{gen}.
 On the other hand, as $(\xi, e^-)\in \cal F^{(2)}$,
 $g_i a_i \xi \to g^+$ by Lemma \ref{gr}. Hence $g^+=k_0^+$, proving the first claim.
Now the second claim follows since $g_ia_i^{-1} = g_iw_0 b_i w_0^{-1}$ for some $b_i\in A^+$,
 and $g_iw_0 b_i w_0^{-1} (o)= g_iw_0 b_i (o)\to (gw_0)^+=g^-$.
\end{proof}

\subsection*{Limit set and Limit cone.}
Denote by $m_o$ the $K$-invariant probability measure on $\cal F \simeq K/M$.
\begin{Def}[Limit set] \rm The limit set $\La$ of $\Gamma$ is defined as the set of all points $\xi\in \F$ such that the Dirac measure $\delta_\xi$ is a limit point of $\lbrace \gamma_{\ast} m_o\,:\,\gamma\in\Gamma\rbrace$
in the space of Borel probability measures on $\cal F$.
\end{Def}
Benoist showed that $\Lambda $ is the unique minimal $\Gamma$-invariant closed subset of $\F$. Moreover, $\Lambda $ is Zariski dense in $\F$ (\cite[Section 3.6]{Ben}, see also \cite[Lem. 2.10]{ELO} for a stronger statement). 
\begin{lemma}\label{lem.lim}
For any $p\in X$, we have
$$
\La =\left\{
\begin{array}{c}
\lim\limits_{i\to\infty} \ga_i p \in\cal F:  \ga_i\in\Ga
\end{array}
\right\}.
$$
\end{lemma}
\begin{proof} Let $(\ga_i)_* m_o\to \delta_\xi$, and write $\ga_i=k_i a_i \ell_i^{-1}\in KA^+K$.
Suppose $k_i\to k$. Then $(a_i)_* m_o\to \delta_{k^{-1}\xi}$. It follows that
 $a_i\to\infty$ regularly in $A^+$ and $k^{-1}\xi=e^+$, i.e., $\xi=k^+$.
 Hence $\ga_i\to \xi$ and hence $\ga_i(p)\to \xi$ by Lemma \ref{ssame}. This proves the inclusion $\subset$.
 If $\ga_i p\to \xi$ and $\ga_i=k_ia_i\ell_i^{-1}\in KA^+K$, then $a_i\to \infty$ regularly and $k_i^+\to \xi$.
Since $(a_i)_*m_o$ converges to $\delta_{e^+}$, we have $(\ga_i)_*m_o \to \delta_\xi$. This proves the other inclusion.
\end{proof}

Any element $g\in G$ can be written as the commuting product $g_hg_e g_u$, where $g_h$, $g_e$ and $g_u$
 are unique elements that are conjugate to elements of $A^+$, $K$ and $N$, respectively. 
When $g_h$ is conjugate to an element of $ \op{int}A^+$, $g$ is called
 {\it loxodromic}; in such a case, $g_u=e$.
 If a loxodromic element $g\in G$ satisfies
  $\varphi^{-1}g_h \varphi\in \op{int}A^+$ for $\varphi\in G$, then 
  \be\label{yg} y_g:=\varphi^+ \ee
is  called the attracting  fixed point of $g$. We then have $y_{g^{-1}}=\varphi^-$.

\begin{Lem} \cite[Lem. 3.6]{Ben} \label{dense}
The set $$\{(y_\gamma,y_{\gamma^{-1}})\in \La\times \La:\gamma\text{ is a loxodromic element of } \Gamma\}$$ is dense in $\Lambda \times\Lambda $. 
\end{Lem}
 
The Jordan projection of $g$ is defined
 as $ \lambda (g)\in \fa^+$, where $\exp \lambda (g)$ is the element of $A^+$ conjugate to $g_h$.
 
 \begin{Def}[Limit cone] \label{lc} \rm The \emph{limit cone} $\L_\Gamma\subset\fa^+$ of $\Gamma$ is defined as the smallest closed cone containing the \emph{Jordan projection} $\la(\Gamma)$.
 Alternatively, it can be defined as the asymptotic cone of  $\mu(\Ga)$ 
 \cite[Thm. 1.2]{Ben}.
\end{Def}

The limit cone $\L_\Gamma$ is a convex subset of $\fa^+$ with non-empty interior \cite[Thm. 1.2]{Ben}. 

\begin{Def}[Growth indicator function]\label{def.GI}\rm The growth indicator function $\psi_{\Gamma}\,:\,\fa^+ \rightarrow \br \cup\lbrace- \infty\rbrace$  is defined as a homogeneous function, i.e., $\psi_\Gamma (tu)=t\psi_\Gamma (u)$, such that
  for any unit vector $u\in \fa^+$,
 \begin{equation*}
\psi_{\Gamma}(u):=\inf_{\underset{u \in\scrC}{\mathrm{open\;cones\;}\scrC\subset \fa^+}}\tau_{\cal C}
\end{equation*}
where $\tau_{\cal C}$ is the abscissa of convergence of the series $\sum_{\ga\in\Ga, \mu(\ga)\in\cal C}e^{-t\norm{\mu(\ga)}}$.
\end{Def}
We may consider $\psi_\Ga$ as a function on $\fa$ by setting $\psi_\Ga=-\infty$ outside of $\fa^+$.
Quint showed the following:
\begin{thm} \cite[Thm. IV.2.2]{Quint1} \label{growth} The growth indicator function $\psi_\Gamma$ is concave, upper-semicontinuous, and satisfies
$$\L_\Gamma= \{u\in \fa^+: \psi_\Gamma(u)>-\infty\}.$$
Moreover, $ {\psi_\Gamma}$ is non-negative  on  $\L_\Gamma$ and positive on $\op{int}\L_\Gamma$.
\end{thm}

\section{$\fa$-valued Gromov product and generalized BMS measures}\label{sec.BMS}
\subsection*{Iwasawa cocycle and $\fa$-valued Busemann function}
The Iwasawa decomposition says that the product map $K\times A\times N\to G$
is a diffeomorphism.
 \begin{Def}\rm The Iwasawa cocycle $\sigma: G\times \F \to \mathfrak a$ is defined as follows: for $(g, \xi)\in G\times \F$,
 $ \sigma(g,\xi)\in \fa$ is the unique element satisfying
\be\label{exchange} gk\in K \exp (\sigma(g, \xi)) N\ee
where $k\in K$ is such that $\xi=k^+$.\end{Def}
 It satisfies the cocycle relation
$$\sigma(g_1g_2, \xi)=\sigma(g_1, g_2\xi)+\sigma(g_2, \xi)$$ for all $g_1,g_2
\in G$ and $\xi\in \cal F$.

\begin{Def}\label{Bu} \rm The $\fa$-valued Busemann function $\beta: \F\times X\times X\to\mathfrak a $ is defined as follows: for $\xi\in \F$ and $g(o), h(o)\in X$,
 $$\beta_\xi ( g(o), h(o)):=\sigma (g^{-1}, \xi)-\sigma(h^{-1}, \xi).$$
\end{Def}

Observe that the Busemann function is continuous in all three variables. To ease the notation, we will write  $\beta_\xi ( g, h)=\beta_\xi ( g(o), h(o))$. 
We can check that for all $g, h, q\in G$ and $\xi\in \cal F$,
\begin{gather}\label{eq.basic0}
\begin{aligned}
\beta_\xi(g, h)+\beta_\xi&(h, q)=\beta_\xi (g, q), \\ 
\beta_{g\xi}(gh,gq)&=\beta_\xi(h,q),\text{ and}\\
\beta_\xi(g,e)=&\sigma(g^{-1}, \xi).
\end{aligned}
\end{gather}

Geometrically, if $\xi=k^+\in \F$ for $k\in K$, then for any unit vector $u\in\mathfrak a^+$,
$$\langle \beta_\xi(g, h), u\ra =\lim_{t\to +\infty} d (g(o), \xi_t)- d(h(o), \xi_t)$$
where $\xi_t= k \exp (tu)o\in X$.

\medskip
\begin{lemma}\label{lem.el}
For any loxodromic element $g\in G$ and $p\in X$, 
$$
\beta_{y_g}(p,gp)=\lambda (g) \quad\text{ and }\quad \beta_{y_{g^{-1}}}(p,gp)=- \lambda (g^{-1}).
$$
\end{lemma}
\begin{proof}
Let $\varphi\in G$ be so that $g=\varphi am\varphi^{-1}$ for some $a\in A^+$ and $m\in M$.
If $p=h(o)$ for $h\in G$, then, since $g^{-1}$ fixes $\varphi^+=y_g$,
$$
\beta_{y_g}( p, gp)=\beta_{\varphi^+}( ho, gho)=\sigma(h^{-1},\varphi^+)-\sigma(h^{-1} g^{-1},\varphi^+)= -\sigma(g^{-1},\varphi^+).
$$
Writing $\varphi=k b$ with $k\in K$ and $b\in P$, we have
$$
g^{-1}k=\varphi(am)^{-1}\varphi^{-1}k=kb(am)^{-1}b^{-1}\in Ka^{-1}N.
$$
This gives $\sigma(g^{-1},\varphi^+)=\log a^{-1}=-\lambda(g)$, and hence the first identity.
The second identity follows from the first, by replacing $g$ with $g^{-1}$.
\end{proof}

\subsection*{$\fa$-valued Gromov product}  
\begin{Def}[Opposition involution]\label{op} \rm 
  The involution  $\i:\mathfrak a \to \mathfrak a$ defined by $$\i (u)= -\op{Ad}_{w_0} (u)$$
  is called the opposition involution; it preserves $\fa^+$.
 Note that for all $g\in G$, 
 $$\lambda(g^{-1})=\i(\lambda (g))\quad  \text{ and }\quad  \mu(g^{-1})=\i(\mu(g)).$$ 
 It follows that 
 \be\label{iii} \i  (\L_\Ga)=\L_\Ga\quad\text{ and }\quad \psi_\Ga \circ \i=\psi_\Ga.\ee
\end{Def}
\begin{Def}\label{def.GP} We define the $\fa$-valued Gromov product on $\cal F^{(2)}$ as follows:
for $(\xi,\eta)\in\cal F^{(2)}$,
$$
\cal G(\xi,\eta): = \beta_{g^+}(e,g)+\op i\beta_{g^-}(e,g)
$$
where $g\in G$ satisfies $g^+=\xi$ and $g^-=\eta$. 
\end{Def}

The definition does not depend on the choice of a representative of
$[g]\in G/AM$. For all $h\in G$ and $(x,y)\in \cal F^{(2)}$, we have the following identity:
\begin{equation}\label{eq.id1}
\cal G(hx,hy)-\cal G(x,y)=  \sigma(h,x)+\op i\sigma(h,y) .
\end{equation}
As  $\cal G(y,x)=\op i\cal G(x,y)$,
the Gromov product is not symmetric in general.

\begin{lemma}\cite{Tits}\label{lem.Tits}
There exists  a family of
irreducible representations $(\rho_\alpha, V_\alpha)$, $ \alpha\in \Pi$, of $G$ so that 
\begin{enumerate}
\item the highest weight $\chi_\alpha$ of $\rho_\alpha$
is a positive integral multiple of the fundamental weight $\varpi_\alpha$ corresponding to $\alpha$;
\item the highest weight space of $\rho_\alpha$
 is one dimensional.
 \end{enumerate}
\end{lemma}

For $\alpha\in\Pi$, denote by $V_{\alpha}^+$ the highest weight space of $\rho_\alpha$, and by $V_{\alpha}^<$ its unique complementary $A$-invariant subspace in $V_{\alpha}$.
We have  $\rho_\alpha(P)V_\alpha^+=V_\alpha^+$, and hence
the map $g\mapsto (\rho_\alpha (g) V_\alpha^+)_{\alpha\in \Pi}$ factors through a proper immersion
 $$\mathcal F=G/P\to \prod_{\alpha\in \Pi} \mathbb P(V_\alpha).$$
Let $\langle \cdot, \cdot \rangle_\alpha$ be a $K$-invariant
 inner product on $V_\alpha$ with respect to which $A$ is symmetric;
 then $V_\alpha^+$ and $V_\alpha^{<}$ are orthogonal to each other.
We denote by $\norm{\cdot}_\alpha$ the norm on $V_\alpha$ induced by $\langle \cdot,\cdot\rangle_\alpha$.
For $\varphi\in V_\alpha^*$, $\|\varphi\|_\alpha$ means the operator norm of $\varphi$.
We also use the notation $\|\cdot\|_\alpha$ for a bi-$\rho_\alpha(K)$-invariant norm on $\op{GL}(V_\alpha)$.
\begin{lem}\label{lem.def2} 
For all $\alpha\in\Pi$ and $g\in G$,
\be\label{g2}
\chi_\alpha(\cal G(g^+,g^-))=-\log\frac{|\varphi(v)|}{\norm{\varphi}_\alpha\norm{v}_\alpha}
\ee
where $v\in gV_\alpha^+$ and $\varphi\in V_\alpha^*$ is such that $\op{\ker}\varphi= gV_\alpha^<$.
\end{lem}
\begin{proof} If we define $\cal G'(g^+,g^-)$ to be the unique element of $\fa$ satisfying \eqref{g2},
it is shown in \cite[Lem 4.12]{Samb1} that
$\cal G'$ satisfies \eqref{eq.id1}. Hence for all $h\in G$,
$$\cal G'(h^+, h^-)-\cal G'(e^+, e^-) =\cal G(h^+, h^-) -\cal G(e^+, e^-).$$
We claim
that $\cal G'(e^+, e^-)=0$;
to check this, 
take $\varphi$ to be the projection $V\to V_\alpha^+$ parallel to $V_\alpha^<$.
Since $V_\alpha^+$ and $V_\alpha^<$ are orthogonal, it follows that $\norm{\varphi}_\alpha=1$.
Now for $v\in V_\alpha^+$, we have
$$
\frac{\abs{\varphi(v)}}{\norm{\varphi}_\alpha\norm{v}_\alpha}=\frac{\norm{v}_\alpha}{\norm{v}_\alpha}=1.
$$
Since $\cal G (e^+,e^-)=0$, we conclude $\cal G=\cal G'$ on $\cal F^{(2)}$.
\end{proof}
\begin{remark}\rm
In view of this lemma,
our definition of Gromov product differs by $-\op i$ from the one given in \cite{Samb1}.
\end{remark}

\subsection*{Patterson-Sullivan measures on $\La $}
\begin{Def} [Conformal measures] \rm
Given a closed subgroup $\Gamma <G$ and $\psi\in \mathfrak a^*$,
a Borel probability measure $\nu$ on $\F$ is called a $(\Gamma, \psi)$-conformal measure if, for any $\ga\in \G$ and $\xi\in \F$,
 \be\label{gc0} \frac{d\ga_* \nu}{d\nu}(\xi) =e^{\psi (\beta_\xi (e, \gamma))}\ee
 where $\ga_* \nu(Q)=\nu(\gamma^{-1} Q)$ for any Borel subset $Q\subset \F$. 
\end{Def}
If $2\rho$ denotes the sum of all positive roots of $G$ with respect to $\fa^+$, then it is a standard fact that
a $(G, 2\rho)$-conformal measure is precisely the
unique $K$-invariant probability measure $m_o$ on $\cal F$ (cf. \cite[Prop. 3.3]{Q7}).

Fix a Zariski dense discrete subgroup $\G<G$ in the rest of this section.
\begin{Def}[Patterson-Sullivan measures] \rm{} For 
 $\psi\in\fa^*$, a $(\Gamma,\psi)$-conformal measure supported on the limit set $\Lambda $ will be called a $(\Gamma, \psi)$-PS measure.
By a $\PS$ measure on $\La$, we mean a $(\Ga,\psi)$-$\PS$ measure for some $\psi\in\fa^*$.
\end{Def}

Set $$D_\Gamma:=\{\psi\in \fa^*: \psi \ge \psi_\Gamma  \}.$$

 The following collection of linear forms is of particular importance:
\be \label{ddgg} D_\Gamma^\star:=\{\psi\in D_\Gamma: \psi(u)= \psi_\Ga(u)\text{
for some  $u\in \mathcal L_\Ga\cap  \op{int} \fa^+ $}\}.\ee
By \eqref{iii}, $\psi\circ \i \in \dg$ for all $\psi\in \dg$.
The concavity of $\psi_\Ga$ and the non-emptiness of $\inte \L_\Ga$ imply
that $D_\Ga^\star$ is non-empty by the Hahn-Banach theorem.
When $\psi(u)=\psi_\Ga(u)$, we say $\psi$ is tangent to $\psi_\Ga$ at $u$.

Generalizing the work of Patterson and Sullivan (\cite{Pa}, \cite{Su}),
Quint \cite{Quint2} constructed  a $(\Gamma, \psi)$-PS measure for every $\psi\in \dg$.

   \subsection*{Generalized BMS-measure $m_{\nu_1,\nu_2}$.}
 Given a pair of $\Gamma$-conformal measures on $\F$, we now define an $MA$-semi invariant measure on $\Gamma\ba G$, 
 which we call a generalized BMS-measure.  
\medskip

\begin{Def}[Hopf parametrization] \rm
The map $$gM\to (g^+, g^-, b=\beta_{g^+}(e, g))$$ gives a homeomorphism
between $G/M$ and $ \F^{(2)}\times \fa $,  which is called the Hopf parametrization of $G/M$. 

\end{Def}

 Fixing a pair of $\Gamma$-conformal measures $\nu_{\psi_1}, \nu_{\psi_2}$ on $\F$ for  
 a pair of linear forms $\psi_1, \psi_2\in \mathfrak a^*$, we define a Radon measure $\tilde m_{\nu_{\psi_1}, \nu_{\psi_2}}$ on $G/M$ 
  as follows: for $g=(g^+, g^-, b)\in \F^{(2)}\times \mathfrak a$,
\begin{equation}\label{eq.BMS0}
d\tilde m_{\nu_{\psi_1}, \nu_{\psi_2}} (g)=e^{\psi_1 (\beta_{g^+}(e, g))+\psi_2( \beta_{g^-} (e, g )) } \;  d\nu_{\psi_1} (g^+) d\nu_{\psi_2}(g^-) db,
\end{equation}
  where $db=d\ell (b) $ is the Lebesgue measure on $\mathfrak a$.
\noindent
This measure is left $\Gamma$-invariant, and hence induces a measure on $\Gamma\ba G/M$.
 We denote by $m_{\nu_{\psi_1}, \nu_{\psi_2}}$ its $M$-invariant lift to $\Gamma\ba G$. 
 It is $A$-semi-invariant as \begin{equation}\label{eq.SI}
 a_*m_{\nu_{\psi_1},\nu_{\psi_2}}=e^{(-\psi_1+\psi_2\circ \i)(\log a)}m_{\nu_{\psi_1},\nu_{\psi_2}}
\end{equation}
for all $a\in A$ \cite[Lem. 3.6]{ELO}.

\noindent
 \textbf{BMS-measures: $m_{\nu_\psi,\nu_{\psi\circ\i}}^{\BMS}$.} Let $\psi\in\fa^*$  and let $\nu_{\psi}$ and $\nu_{\psi \circ \i}$ be  respectively $(\Gamma, \psi)$ and $(\Gamma, \psi\circ \i)$-PS measures.
We set 
\begin{equation}\label{def.BMS}
m^{\BMS}_{\nu_\psi, \nu_{\psi\circ \i}}:=m_{\nu_\psi, \nu_{\psi\circ \i}}
\end{equation}
and call it the Bowen-Margulis-Sullivan measure associated to $(\nu_\psi, \nu_{\psi\circ \i})$.
 It is  right $MA$-invariant and its support is given by 
 $$\Omega:=\{x \in\Ga\ba G : x^\pm\in\Lambda\};$$
 since $\Lambda$ is $\Gamma$-invariant, the condition $x^{\pm}\in \Lambda$ is well-defined.
 
Note that for $[g]\in G/M$,
\be\label{bmss}
m^{\BMS}_{\nu_\psi, \nu_{\psi\circ \i}}[g]=e^{\psi(\cal G(g^+,g^-))}d\nu_\psi(g^+)d\nu_{\psi\circ\op i}(g^-)db.
\ee

\medskip 

 \noindent
 \textbf{$N$-invariant BR-measures: $m^{\BR}_{\nu_\psi}$.}
 We set 
 \begin{equation}\label{def.BR}
m^{\BR}_{\nu_\psi}:=m_{ \nu_\psi, m_o}
\end{equation} 
 and call it the $N$-invariant Burger-Roblin measure associated to $\nu_\psi$.
 See \cite[Section 3]{ELO} for the equivalence of this definition with the one given in  \eqref{introbr}.
 The support of $m^{\BR}_{\nu_\psi}$
 is given by 
$$ \cal E:=\{x \in \Gamma\ba G: x^+\in \Lambda\}.$$

\section{Anosov groups and $AM$-ergodicity of BMS measures}\label{pingpong}\label{sec.Anosov}
Let $\Ga$ be a Zariski dense discrete subgroup of $G$, and set $\La ^{(2)}:= (\La \times\La) \cap\cal F^{(2)}$.
\begin{Def}
\rm We say that $\Ga<G$ is Anosov, if it is a finitely generated word hyperbolic group admitting a
$\Ga$-equivariant homeomorphism $\zeta:\partial \Ga\to
\La $ such that $(\zeta(x), \zeta(y))\in \La ^{(2)}$ for all $x\ne  y\in \partial \Ga$, where $\partial \Ga$ denotes the Gromov boundary of $\Ga$.
\end{Def}

Such $\zeta$ is H\"older continuous and exists uniquely (\cite[Prop. 3.2]{La} and \cite[Lem. 2.5]{BCLS}). 
We call it the limit map of $\Ga$. We note that the antipodal property of $\La$
follows directly:
\be\label{anti} \La\times \La -\{(\xi, \xi)\}=\La^{(2)}.\ee
In the literature, this definition is referred to as $P$-Anosov  for a minimal parabolic subgroup $P$ of $G$.
See \cite{GW}, \cite{GGKW} and \cite{KLP} for equivalent characterizations of Anosov subgroups.

 \medskip
 
 In the rest of this section,
let $\Gamma$ be an Anosov subgroup of $G$.

 The following theorem was proved by Quint \cite[Prop. 3.2 and Thm. 4.7]{Q4} for Schottky groups
and by Sambarino \cite[Coro. 3.12, 3.13 and 4.9]{Samb2} and
 by \cite[Thm. 7.9]{ELO} for general Anosov subgroups in view of the results in \cite{CS} (see also \cite[Remark 7.10]{ELO}):
\begin{thm}\label{pop}
\begin{enumerate}
\item  $\L_\Gamma\subset \op{int}  \mathfrak a^+\cup\{0\}$
and every non-trivial element of
$\Ga$ is loxodromic (\cite{La}, \cite{GW}).
 \item $\psi_\Gamma$ is strictly concave and  analytic on $\op{int}\L_\Gamma$.
 \item   $\dg=\{\psi\in D_\Gamma: \psi (u)=\psi_\Gamma(u) \text{ for some $u\in \inte \L_\Ga$}\}.$
 \item For any $\psi\in \dg$, $\psi>0$ on $\cal L_\Gamma -\{0\}$.
\item For any $\psi\in \dg$, there exists a unique $(\Gamma, \psi)$-$\PS$ measure, say $\nu_\psi$, on $\F$. In particular, $\nu_\psi$ is $\Gamma$-ergodic.
\end{enumerate}
 \end{thm}
(1) and (3) imply that if $\psi\in D_\Ga$ is tangent to $\psi_\Ga$ at some $u\in \L_\G-\{0\}$, then $u\in \inte \L_\Ga$.
Note also that (1) implies that any Anosov subgroup is regular as defined in Def. \ref{Gr}.
 For $u\in \L_\Gamma$, we denote by $D_u\psi_\Ga$ the directional derivative of $\psi_\Ga$ at $u$, whenever it exists.
 \begin{proposition}\label{homeo} For each unit vector $u\in \inte \L_\Ga$, $\psi_u:=D_u\psi_\Ga\in \dg$ and
 $D_u\psi_\Ga(u)=\psi_\Ga(u)$. Moreover,
the map $u\mapsto \psi_u$ induces a homeomorphism between the set of unit vectors of $\op{int}\cal L_\Ga$ ($\simeq \inte \bb P\cal L_\Ga$) and $\dg$. Hence $\dg\simeq \br^{\text{rank } G-1}$.
 \end{proposition}
 \begin{proof}  See (\cite[Thm. A]{Samb2}, \cite[Lem. 2.23]{ELO}) for the first claim. The 
 fact that $\psi_u\in D_\Ga^\star$ (and hence the
 well-definedness)
 and surjectivity of the map $u\mapsto D_u\psi_\Ga$
 follows from it, and the injectivity  follows from the strict concavity of $\psi_\Gamma$ as in Theorem \ref{pop}(2).
Continuity follows from the analyticity of $\psi_\Ga$ on $\inte \L_\G$. We claim that if
$D_{u_i}\psi_\Ga\to D_u\psi_\G$ for some unit vectors $u_i, u\in \inte \L_\Ga$, then $u_i\to u$.
Let $v\in \L_\Ga$ be a limit of the sequence $u_i$. By passing to a subsequence, assume $u_i\to v$.
By the upper-semi continuity of $\psi_\Ga$ (Theorem \ref{growth}), we have $$\psi_\Ga(v)\ge \limsup_{i\to \infty} \psi_\Ga (u_i).$$
Since $\psi_\Ga(u_i)=D_{u_i}\psi_\Ga (u_i)$ and $D_{u_i}\psi_\Ga\to D_u\psi_\Ga$,
we get
$\psi_\Ga(v)\ge  D_u\psi_\Ga(v)$. Since $D_u \psi_\Ga\in D_\Ga^\star$,  we have
$\psi_\Ga(v)=  D_u\psi_\Ga(v)$. It follows from Theorem \ref{pop}(1) and (3) that $v\in \inte \L_\Ga$. Since $\psi_\Ga(u)=D_u\psi_\Ga(u)$, the strict concavity of $\psi_\Ga$ on $\inte\L_\Ga$ implies
that $u=v$, establishing the homeomorphism. Since $\inte ( \L_\Ga)$ is a non-empty open convex cone
of $\fa^+$, $\inte(\mathbb P\L_\Ga)\simeq \mathbb P\inte(\L_\Ga)$ is homeomorphic to
$  \br^{\text{rank}\,G-1}$.
\end{proof}

We denote by $\nabla \psi_\Ga$ the gradient of $\psi_\Ga$ so that
$D_u\psi_\Ga(v)=\langle \nabla \psi_\Ga (u), v\rangle$ for $u\in \inte \L_\Ga$ and $v\in \fa$.
Set \be\label{ooo} \mathsf O_\Ga :=\{\nabla \psi_\Gamma (u)\in \fa: u\in \inte\L_\Ga\},\ee
which is an open convex  cone of $\fa-\{0\}$.
By Proposition \ref{homeo},  the map $w\mapsto \langle w, \cdot \rangle$
gives a homeomorphism between $\{\nabla \psi_\Gamma (u)\in \fa: u\in \inte\L_\Ga, \|u\|=1\}$ and $\dg$, and hence
a homeomorphism 
$$ \mathsf O_\Ga\simeq \br_+\dg.$$ Quint showed that there exists a unique unit vector, say $u_\Gamma\in \inte \fa^+$, such that
$\psi_\Gamma(u_\Gamma)=\max_{\|u\|=1} \psi_\Gamma (u)$. The vector $u_\Ga$ is called the direction of maximal growth
of $\Gamma$. If we set $\delta_\Ga:=\psi_\Gamma (u_\Gamma)$,
 then  $\nabla \psi_\Gamma (u_\Gamma)=\delta_\Gamma u_\Gamma$ and
 $$\delta_\G=\limsup_{T\to \infty}\frac{\log \#\{\ga\in \Ga: \|\mu(\ga)\| <T\}}{T}.$$
 
Consider the following dual cone to $\L_\Ga$: 
$$\L_\Ga^*:= \{w\in \fa: \langle w, v\rangle \ge  0\text{ for all $v\in \L_\Ga$}\}.$$
Note that $\inte \L_\Ga^*= \{w\in \fa: \langle w, v\rangle >  0\text{ for all non-zero $v\in \L_\Ga$}\}$.
\begin{lemma} \label{oo}
We have $\mathsf O_\Ga=\inte \L_\Ga^*$.
In particular, 
$$\mathfrak a^+-\{0\}\subset \mathsf O_\Ga\subset \fa-\{0\}.$$
\end{lemma}

\begin{proof}
If $w\in \O_\Ga$, then $\langle w, \cdot \rangle \in \br_+\dg$, and hence by Theorem \ref{pop} (3),
$\langle w, v\rangle >0$ for all $v\in \L_\Ga-\{0\}$. Hence $\mathsf O_\Ga\subset \inte \L_\Ga^*$.
Now suppose $w\in \inte \L_\Ga^*$. Setting $\psi(v):=\langle w, v\rangle$, we claim that $\psi\in \br_+ \dg$; this implies
 $ \inte \L_\Ga^*\subset \mathsf O_\Ga$.
Since $w\in \inte \L_\Ga^*$, $\psi >0$ on $\L_\Ga -\{0\}$ and hence $c:=\max_{\|v\|=1, v\in \L_\Ga} \psi (v)>0$.
Since $\delta_\Ga c^{-1} \psi \ge \psi_\Ga$ on $\L_\Ga$, and hence on $\fa$, 
it follows that for some $\epsilon>0$,  $\epsilon \psi \in \dg$, i.e., $\psi\in \br_+ \dg$.

Since $\L_\Ga-\{0\}\subset \inte \fa^+$ and
 the angle between any two walls of $\mathfrak a^+$ is at most $\pi/2$,  the second claim follows.
\end{proof}

\subsection*{$AM$-ergodicity of $m_\psi^{\BMS}$}
We fix 
$\psi\in\dg$ and set 
\begin{equation}\label{def.BMS2}
\nu:=\nu_\psi\quad\text{ and }\quad m^{\BMS}_\psi:=m^{\BMS}_{\nu_\psi, \nu_{\psi\circ \i}} 
\end{equation}
The composition $c:=\psi\circ\sigma : \Ga\times \Lambda\to \br$ is a H\"older cocycle satisfying
$c(\gamma, y_\gamma)=\psi(\la(\gamma)) >0$ for all non-trivial $\gamma\in \Gamma$.

Consider the action of $\Gamma$ on $\Lambda^{(2)} \times\bb R$ given as follows: for $\ga\in\Gamma$ and $(\xi,\eta,t)\in\Lambda^{(2)}\times\bb R$,
\begin{equation*}
\ga.(\xi,\eta,t)=(\ga \xi,\ga \eta, t+c(\ga,\xi)).
\end{equation*}

The $\bb R$-action on $\Lambda^{(2)} \times\bb R$ defined by
$$
\tau_s(\xi,\eta,t)=(\xi,\eta,t+s)
$$
will be called translation flow. 

The following is proved in \cite[Thm. 3.2]{Samb3} when $\Gamma$ is the fundamental group
of a closed negatively curved manifold, and can be extended for general Anosov groups, using ingredients from \cite{BCLS}.
The sketch of the proof can be found in \cite[Appendix A]{CAR}.
\begin{theorem}\label{thm.reparam}
The action of $\Gamma$ on $\Lambda^{(2)} \times\bb R$ is proper and cocompact, and the measure
$d\tilde{\mathsf m}_\psi(\xi, \eta, t) =e^{ \psi (\cal G(\xi,\eta)) }d\nu_{\psi}(\xi)
\otimes d\nu_{\psi\circ\op i}(\eta) \otimes dt$
induces the measure of maximal entropy, say $\mathsf m_\psi $, for $\{\tau_s:s\in \br\}$ on $\Gamma\ba \La^{(2)}\times \br$. In particular, $\mathsf m_\psi $ is $\{\tau_s:s\in \br\}$-ergodic.
\end{theorem}
In terms of the Hopf parametrization, $\Ga$ acts
 on $ \Lambda^{(2)}\times\mathfrak a=\supp \tilde m_{\psi}^{\BMS}$ as follows: for $\ga\in\Ga$ and $(\xi,\eta,v)\in\Lambda^{(2)}\times\mathfrak a$,
$$
\ga.(\xi,\eta,v)=(\ga \xi,\ga \eta, v+\sigma(\ga,\xi)).
$$

\begin{Coro}\label{AMerg}
For any 
$\psi\in \dg$, the $AM$-action on $(\Gamma\ba G, m_{\psi}^{\BMS})$ is ergodic
and if $\text{rank} \,G\ge 2$, $|m_\psi^{\BMS}|=\infty$.
\end{Coro}
\begin{proof} The  $\{\tau_s:s\in \br\}$-ergodicity of $\mathsf m_\psi$ is equivalent to ergodicity of
$( \Lambda^{(2)}, \Ga, \nu_{\psi}\otimes\nu_{\psi\circ\op i}|_{\Lambda^{(2)}})$,
which is again equivalent to the $AM$-ergodicity  of $m_{\psi}^{\BMS}$.
Consider the projection map $
\pi : \Ga\,\ba\, \Lambda^{(2)}\times\mathfrak a \to\Ga\,\ba\,\Lambda^{(2)}\times\bb R
$ induced by the
$\Ga$-equivariant map $
\Lambda^{(2)}\times\mathfrak a\to\Lambda^{(2)}\times\bb R$ given by
$(\xi,\eta,v)\mapsto (\xi,\eta,\psi(v)).$
Then  $\pi$ is a principal  $\ker \psi$-bundle, which is trivial as  $\ker \psi$ is a vector group. It follows that
there exists a $\ker \psi$-equivariant homeomorphism between $\Ga\,\ba\, \Lambda^{(2)}\times\mathfrak a $ and
$\left(\Ga\,\ba\,\Lambda^{(2)}\times\bb R\right) \times \ker \psi$.
Therefore
$m_\psi^{\BMS}$ disintegrates over the measure
${\mathsf m}_\psi $ with conditional measure being the Lebesque measure on $\ker\psi\simeq \br^{\text{rank} G-1}$
so that $m_\psi^{\BMS}\simeq {\mathsf m}_\psi \otimes \op{Leb}_{\op{ker} \psi}$ (cf. \cite[Prop. 3.5]{Samb1}).
This gives the infinitude of $|m_\psi^{\BMS}|$ when $G$ has rank at least $2$.
\end{proof}
\section{Comparing $\fa$-valued Busemann functions and distances via $\psi$}

When $G$ has rank one, for any $p, q\in X$, the maximum and minimum
 of Busemann function $\beta_\xi(p,q), \xi\in\cal F$ are
  always achieved  as $\pm d(p,q)$. A higher rank generalization of this fact can be stated as follows.

\begin{Lem}\label{lem.SP}
Let $\psi\in\mathfrak a^*$ be strongly positive, in the sense that $\psi$ is a non-negative linear combination of fundamental weights $\varpi_\alpha$, $\alpha\in \Pi$. 
Then for any $p,q\in X$ and $\xi \in \cal F$, we have
 \be\label{ggg}  - \psi (\underline a(q, p))\le
\psi (\beta_\xi(p, q))\le \psi (\underline a(p, q)).
 \ee
\end{Lem}
\begin{proof}
We use notations introduced in Lemma \ref{lem.Tits}. Since $\varpi_\alpha$ is a positive multiple of $\chi_\alpha$,  it suffices to prove the claim
 when $\psi=\chi_\alpha$ for $\alpha\in \Pi$. 

Write $q=go$, and $p=hq$ for some $g,h\in G$.
Note that $$\chi_\alpha(\underline a(p,q))=\chi_\alpha(\mu(g^{-1}h^{-1} g))=\log \norm{\rho_\alpha(g^{-1}h^{-1} g)}_\alpha.$$
Write $g^{-1}\xi= k^+$ for some $ k\in K$ and 
 $g^{-1}h^{-1}g k = k'an\in KAN$.
 Then $$\beta_\xi(p,q)=\sigma(g^{-1} h^{-1} g, k^+)=\log a.$$
Hence for a unit vector $v\in V_\alpha$,
$$
\chi_\alpha(\beta_\xi(p,q))
=\log {\norm{\rho_\alpha(g^{-1}h^{-1} g)\rho_\alpha( k)v}}\le \log {\norm{\rho_\alpha(g^{-1}h^{-1} g)}}_\alpha=
\chi_\alpha(\underline a(p,q)) .
$$

Since 
$ \|\rho_\alpha(g^{-1})\|^{-1}\le  {\norm{\rho_\alpha( g)v}} $ and 
$\chi_\alpha(\underline a(q,p))=\log \norm{\rho_\alpha(g^{-1}hg)}_\alpha$,
we also get
$$
\chi_\alpha(\beta_\xi(p,q))\ge \log {\norm{\rho_\alpha(g^{-1}h g)}}^{-1}_\alpha=
-\chi_\alpha(\underline a(q,p)) .$$
\end{proof}

There are $\psi\in\dg$ that are not strongly positive (see Lemma \ref{oo}).
We establish the following modification for Anosov groups, which is the main goal of this section:
\begin{thm} \label{concave}
Let $\Ga<G$ be Anosov.  
For any $\psi\in\dg$ and $p\in X$, there exists $C=C(\psi,p)>0$ such that for all $\ga\in\Ga$ and $\xi \in \La$,
$$
-\psi(\underline a (p, \ga p))-C
 \le  \psi(\beta_\xi(\ga p, p))\leq  \psi(\underline a (\ga p,  p))+C.
$$
\end{thm}

We begin by noting that $\psi(\underline a(\ga p,p))$ is always positive possibly except for finitely many $\ga$'s:
\begin{Lem} \label{concave0} 
Let $\psi\in\dg$ and $p\in X$.
For any sequence $\ga_i\to\infty$ in $\Ga$,
 $\psi(\underline a(\ga_i p,p))\to +\infty$.
\end{Lem}

\begin{proof} By Lemma \ref{com}, it suffices to check that $\psi(\mu(\ga_i))\to+\infty$
 as $i\to \infty$.  Setting $t_i:=\|\mu(\ga_i)\|^{-1}$, passing to a subsequence,
 we may assume that  $t_i\mu(\ga_i)$ converges to some unit vector $u\in \fa$. 
Since $\L_\Ga$ is the asymptotic cone of $\mu(\Ga)$, we have $u\in\cal L_\Ga$.
Hence we have $\psi(u)>0$ by Lemma \ref{pop}.
Since $\psi(t_i\mu(\ga_i))\to \psi(u)$ and
$\psi(\mu(\ga_i))=t_i^{-1}\psi(t_i\mu(\ga_i))$, we have $\psi(\mu(\ga_i))\to+\infty$.
\end{proof}

The following is the main ingredient of the proof of Theorem \ref{concave}:
\begin{proposition}\label{prop.AC}
For $p\in X$, there exists $C=C(p)>0$  such that for each $(\ga,\xi)\in\Ga\times\La$, we can find $\ga_1=\ga_1(\xi),\ga_2=\ga_2(\xi) \in\Ga$ satisfying
\begin{enumerate}
\item
$\ga=\ga_1\ga_2$ and $|\ga|=|\ga_1|+|\ga_2|$;
\item
$\norm{\beta_\xi(\ga p,p)+\mu(\ga_1)-\mu(\ga_2^{-1})}\leq C$;
\item
$\norm{\underline a(\ga p,p)-\mu(\ga_1^{-1})-\mu(\ga_2^{-1})}\leq C$.
\end{enumerate}
\end{proposition}
\noindent
\textbf{Proof of Theorem \ref{concave} using Proposition \ref{prop.AC}:}
For $\ga\in\Ga$ and $\xi\in\La$, choose $\ga_1,\ga_2\in\Ga$ as in Proposition \ref{prop.AC}.
Then
\begin{align*}
\psi(\beta_\xi(\ga p,p))&\leq \psi(\mu(\ga_2^{-1})-\mu(\ga_1))+C\norm{\psi}\\
&\leq \psi(\mu(\ga_2^{-1})+\mu(\ga_1^{-1}))+C\norm{\psi}\\
&\leq \psi(\underline a(\ga p,p))+2C\norm{\psi},
\end{align*}
where the second inequality is valid because $\psi(\mu(\ga_1^{\pm1}))\geq 0$.
Similarly, we get 
\begin{align*}
\psi(\beta_\xi(\ga p,p))&\geq \psi(\mu(\ga_2^{-1})-\mu(\ga_1))-C\norm{\psi}\\
&\geq -\psi(\mu(\ga_2)+\mu(\ga_1))-C\norm{\psi}. \end{align*}
Since $\op{i}\mu(g^{-1})=\mu(g)$, $\op{i}\underline a(p,q)=\underline a(q,p)$ and the norm
is $\op{i}$-invariant, 
we get 
$\psi(\beta_\xi(\ga p,p))\ge \underline \psi (a(p, \ga p)) -2C \|\psi\|$.
$\qed$

\medskip
The rest of this section is devoted to a proof of Proposition \ref{prop.AC} in which
shadows of $\cal F$ and $\partial\Ga$ as well as their relationship play important roles.
\subsection*{Shadows in $\cal F$}
Let $ q\in X$ and $r>0$.
The shadows of the ball $B(q,r)$ viewed from $p\in X$ and $\xi\in \cal F$ are respectively defined as
 $$O_r(p,q):=\{gk^+\in \cal F: k\in K,\;
  gk\inte A^+o\cap  B(q,r)\ne \emptyset\}$$
 where $g\in G$ satisfies $p=g(o)$, and
$$O_r(\xi,q):=\{h^+\in \cal F: h^-=\xi, ho\in  B(q,r)  \}.$$

The following two lemmas \ref{lem.SC} and \ref{lem.shadow1} are proved for $G=\op{SL}_n(\br)$ in \cite{Thir}. 
For $r>0$, we set $A_r=\{a\in A: \|\log a\|\le r\}$, and $A^+_r=A_r\cap A^+$.
\begin{lem}\label{lem.SC} 
If a sequence $q_i\in X$  converges to $\xi\in \cal F$,
then for any $r>0$, $q\in X$ and $\e>0$, we have, for all sufficiently large $i$,
$$ O_{r-\e} (q_i,q)\subset O_r(\xi,q)\subset O_{r+\e}(q_i,q) .$$ 
\end{lem}
\begin{proof} Since $O_r(\xi,g(o))= gO_r(g^{-1}\xi, o)=kgO_r(e^+, o)$ for $k\in K$ with $k^+=g^{-1}\xi$,
it suffices to consider the case when $q=o$ and $\xi=e^+$.
Since $q_i\to e^+$, we have $q_i=k_ia_i o$ for some $k_i\to e$ in $K$ and $a_i\to\infty$ regularly in $A^+$.
As $k_i\to e$, $O_{r-\e/2}(e^+, o) \subset k_i^{-1}O_r(e^+,o)\subset O_{r+\e/2}(e^+, o)$ for all sufficiently large $i$. Therefore
we may assume without loss of generality that $q_i=a_i$. Let $\xi\in O_r(e^+,o)$, that is, $\xi=h^+$ for some $h\in Pw_0\cap KA_r^+K$.
Note that the sequence $a_i^{-1}hw_0^{-1}a_i$ is bounded as $a_i\in A^+$ and $hw_0^{-1}\in P$.
If we write $a_i^{-1}h=\tilde k_i \tilde a_i \tilde n_i\in KAN$,
then the following gives  the $KAN^+$-decomposition of $a_i^{-1}hw_0^{-1} a_i$:
$$
a_i^{-1}hw_0^{-1} a_i=\tilde k_i \tilde a_i \tilde n_i w_0^{-1}a_i=(\tilde k_iw_0^{-1})(w_0 \tilde a_i w_0^{-1}a_i)(a_i^{-1}w_0\tilde n_i w_0^{-1}a_i)\in KAN^+\!\!.$$
As the product map $K\times A\times N^+\to G$ is a diffeomorphism, there exists $R>1$ such that
$\{w_0 \tilde a_i w_0^{-1}a_i\in A:i\in \N\}\subset  A_R$.
 Moreover as the sequence $a_i^{-1}w_0\tilde n_i w_0^{-1}a_i$ must be bounded while
  $a_i\to\infty$ regularly in $A^+$, it follows that $\tilde n_i\to e$ as $i\to\infty$.
We now claim that for all sufficiently large $i\gg 1$,
$$a_i\tilde k_i\op{int}A^+\cap KA_{r+\e}^+K\ne \emptyset$$ and hence $\xi=h^+= a_i\tilde k_i^+ \in O_{r+\e}(a_io,o)$.
Set $b_i=w_0a_i^{-1}w_0^{-1}$.
Then
$$
a_i\tilde k_ib_i=a_i\tilde k_i(w_0a_i^{-1}w_0^{-1})=h\tilde n_i^{-1}(\tilde a_i^{-1}w_0 a_i^{-1}w_0^{-1}).
$$

Since $\tilde n_i\to e$ as $i\to\infty$ and $\tilde a_i^{-1}w_0 a_i^{-1}w_0^{-1}\in A_R$,
 we can find $\tilde b_i\in b_iA_R$ such that $a_i\tilde k_i\tilde b_i\to h$ as $i\to\infty$.
 On the other hand, by the strong wavefront lemma \cite[Thm. 2.1]{GO}, there exists a neighborhood $\cal O$ of $e$ in $G$ such that
 $KA_r^+K \cal O\subset K A_{r+\e} K$.
Since $h\in KA_r^+K$ and $\tilde b_i\in\op{int}A^+$ for all large $i\gg 1$, 
we obtain that $a_i\tilde k_i\tilde b_i\in K A_{r+\e}K$ for all sufficiently large $i$, proving the inclusion on the right hand side.

Now, in order to show $O_{r-\e}(a_i,o)\subset O_r(e^+,o)$,
let $\eta_i\in O_{r-\e}(a_io,o)$ be arbitrary.
Since $O_{r-\e}(a_io,o)=a_iO_{r-\e}(o,a_i^{-1}o)$, there exists $\tilde k_i\in K$ and $b_i\in \inte A^+$ such that 
$\tilde k_i^+\in O_{r-\e}(o,a_i^{-1}o)$, $\eta_i=a_i\tilde k_i^+$ and $a_i\tilde k_ib_i\in KA_{r-\e}^+K$.
Since $a_i^{-1} o\to e^-$, it follows that $\tilde k_i^+\to e^-$ as $i\to\infty$. Hence $k_i^-$ is in general position with $e^-$ for all large $i$ 
 and hence $(a_i\tilde k_i b_i)^-=a_i\tilde k_i^-\to e^+$ as $i\to\infty$.
For all large $i$, we have  $n_i\in N$ such that $(a_i\tilde k_i b_in_i)^-=e^+$ and $d(n_io,o)<\e$.
Set $h_i:=a_i\tilde k_i b_in_i$.
It follows that $h_i^+=\eta_i$ and $h_i^-=e^+$.
Since
\begin{align*}
d(h_io,o)&=d(a_i\tilde k_i b_in_io,o)\leq d(a_i\tilde k_i b_in_io,a_i\tilde k_i b_io)+d(a_i\tilde k_i b_io,o)\\
&=d(n_io,o)+d(a_i\tilde k_i b_io,o)<\e+(r-\e)=r,
\end{align*}
we have $\eta_i=h_i^+\in O_r(e^+,o)$ as desired.
This finishes the proof.
\end{proof}

The following is an analogue of Sullivan's shadow lemma:
\begin{lemma}\label{lem.shadow1}
There exists $\kappa>0$ such that for any $p,q\in X$ and $r>0$, we have
$$
\sup_{\xi\in O_r(p,q)}\norm{\beta_\xi(p,q)-\underline a(p,q)}\leq \kappa r.
$$
\end{lemma}

We will prove this lemma using the following: \begin{lemma}\label{lem.KR} \label{lem.KR2}
\begin{enumerate}
\item There exists $c_1>1$ such that for all $r\ge 0$,
$$KA_r^+K\subset  KA_{c_1 r}N .$$
\item There exists $c_2>1$ such that for all $g\in G$ and $r\ge 0$, we have 
$$\mu(gKA_r^+K)\subset \mu(g)+\log A_{c_2 r} .$$
\end{enumerate}
\end{lemma}
\begin{proof}
We use notations introduced in Lemma \ref{lem.Tits}.  Since $\chi_\alpha, \alpha\in\Pi$, form a dual basis of $\fa^*$, $\norm{\cdot}_*:=\sum_{\alpha\in\Pi}|\chi_\alpha(\cdot)|$ defines a norm on $\fa$.
Let $k\in K$ and $a\in A_r^+$ be arbitrary.
Write $ak=k'bn\in KAN$. Let $\alpha\in \Pi$.
For $v_\alpha\in V_\alpha^+\ba\{0\}$, we have
$$
\norm{\rho_\alpha(ak)v_\alpha}_\alpha=\norm{\rho_\alpha(k'bn)v_\alpha}_\alpha=\norm{\rho_\alpha(b)v_\alpha}_\alpha=e^{\chi_\alpha(\log b)}\norm{v_\alpha}_\alpha.
$$
On the other hand, we have
$$
e^{-\chi_\alpha(\log a)}\norm{v_\alpha}_\alpha\leq \norm{\rho_\alpha(ak)v_\alpha}_\alpha\leq e^{\chi_\alpha(\log a)}\norm{v_\alpha}_\alpha.
$$

Hence $ |\chi_\alpha(\log b)| \le  |\chi_\alpha(\log a)|$ for all $\alpha\in\Pi$; so
$  \|\log b\|_* \le  \|\log a\|_*.$
Since $\|\cdot\|_*$ and $\|\cdot\|$ are comparable, the first claim  follows.

Note that $e^{\chi_\alpha(\mu(g))}=\norm{\rho_\alpha(g)}_\alpha$ for $g\in G$.
For all $k\in K$ and $a\in A^+$,
$$
\norm{\rho_\alpha(gka)}_\alpha\leq \norm{\rho_\alpha(g)}_\alpha \norm{\rho_\alpha(k)}_\alpha \norm{\rho_\alpha(a)}_\alpha=e^{\chi_\alpha(\log a)}\norm{\rho_\alpha(g)}_\alpha,
$$
and similarly,
$$
\norm{\rho_\alpha(g)}_\alpha\leq \norm{\rho_\alpha(gka)}_\alpha \norm{\rho_\alpha(k^{-1})}_\alpha \norm{\rho_\alpha(a^{-1})}_\alpha=e^{\chi_\alpha(\log a)}\norm{\rho_\alpha(gka)}_\alpha.
$$
It follows that $-\chi_\alpha(\log a)\leq \chi_\alpha(\mu(gka)-\mu(g))\leq\chi_\alpha(\log a)$.
This means that $ \|\mu(gka)-\mu(g)\|_* \le \|\log a\|_*$, finishing the proof. \end{proof}

Since $\i (\mu(g))=\mu(g^{-1})$,  Lemma \ref{lem.KR2}(2) easily implies that
 for all $g\in G$ and $r\ge 0$, we have 
\be \mu(A_r^+K gKA_r^+)\subset \mu(g)+\log A_{2c_2 r} .\ee

\noindent{\bf Proof of Lemma \ref{lem.shadow1}.}
It suffices to prove the claim for $p=o$ and $q=a^{-1}o$ for $a\in A^+$.
Let $\xi=k^+\in O_r(o,a^{-1}o)$ for $k\in K$. Then there exists $b\in\op{int}A^+$ such that  $d(kbo, a^{-1}o)<r$. 
Hence $akb\in KA_r^+K$.
Now note that $ak\in Ke^{-\beta_\xi(o,a^{-1}o)}N$  by the definition of $\beta_\xi(o,a^{-1}o)$ and hence
$$
akb\in Kbe^{-\beta_\xi(o,a^{-1}o)}N\cap KA_r^+K.
$$
By Lemma \ref{lem.KR}(1), $be^{-\beta_\xi(o,a^{-1}o)}\in A_{c_1r}$.
On the other hand, $b\in Ka^{-1}KA_r^+K$, and hence $\log b\in \mu(a^{-1}KA_r^+K)$.
By Lemma \ref{lem.KR2}(2), $b\in a^{-1}A_{c_2r}$.
Since $\underline a(o,a^{-1}o)=\log(a^{-1})$, the lemma is now proved.
\bigskip

Lemma \ref{lem.shadow1} implies Theorem \ref{concave} for those $\xi \in O_r(\ga p,p)$.
In order to control the value of $\beta_\xi(\ga p,p)$ when $\xi\notin O_r(\ga p,p)$,
we use the Anosov property of $\Ga$.
Let us recall some basic terminologies for hyperbolic groups for which
we refer to \cite{BH} and \cite{KB}. 
 
\subsection*{Discrete geodesics}
Let $\Ga$ be a finitely generated word hyperbolic group.
We fix a finite symmetric generating subset $S$ of $\Ga$ once and for all.
 Let $|\cdot| : \Ga\to\bb N\cup \{0\}$ denote the word length associated to $S$. We denote by $d_{\mathsf w}$ the associated left-invariant word metric, that is,
 $d_{\mathsf w} (\ga_1, \ga_2):=|\ga_1^{-1}\ga_2|$ for $\ga_1, \ga_2\in \Ga$.

A finite sequence $(\ga_0,\cdots,\ga_n)$ of elements of $\Ga$ will be called a finite path if $\ga_i^{-1}\ga_{i+1}\in S$ for all $i$.
Such a path will be called a geodesic segment if $|\ga_0^{-1}\ga_n|=n$.
Infinite and bi-infinite paths can be defined analogously.
They will be called  geodesic rays and geodesic lines, respectively, if all of their finite subpaths are geodesic segments.

Let $\partial\Ga$ denote the Gromov boundary of $\Ga$, that is,
 $\partial\Ga$ is the set of equivalence classes of geodesic rays, where two rays are equivalent to each other if and only if their Hausdorff distance is finite.
For a geodesic ray $(\ga_0,\ga_1,\cdots)$, we use the notation
$[\ga_0,\ga_1,\cdots]$ for its equivalence class in $\partial\Ga$. 

Let $(\cdot|\cdot)$ denote the Gromov product in the hyperbolic space $\Ga$ based at $e\in\Ga$:
$(\ga_1| \ga_2):=\frac{1}{2} \left( d_{\mathsf w}(\ga_1, e)+ d_{\mathsf w}(\ga_2, e)- d_{\mathsf w}(\ga_1, \ga_2)\right)$. This extends to $\partial \Ga$: for $x,y\in \partial \Gamma$, $(x|y):=\sup \liminf_{i,j\to \infty}
(\gamma_i| \gamma_j')$ where the supremum is taken over all sequences $\ga_i$ and $\ga_j'$
such that $x=\lim \ga_i$ and $y=\lim \ga_j'$.
The union $\Ga\cup \partial \Ga$ is a compact space with the topology given as follows:
a sequence $\ga_i\in \Ga$ converges to $ x\in \partial\Ga$ if and only if
 $ \lim_{i\to \infty}(\ga_i | v_i)=\infty$ for any geodesic ray
 $(e, v_1, v_2, \cdots)$ representing $x$.

For any $x, y\in\Ga\cup\partial\Ga$, there exists a discrete geodesic starting from $x$ and ending at $y$, which may not be unique.
By $[x,y]$, we mean one of those geodesics and by $[x,y)$ we mean $[x,y]-\{y\}$.

A geodesic triangle is a union of three geodesics, pairwise sharing a common endpoint in $\Ga\cup\partial\Ga$.
Since $\Ga$ is hyperbolic, there exists $\delta=\delta(\Ga,S)>0$ such that for any geodesic triangle $\Delta$, we can find a point on each edge of $\Delta$ so that the set of these triples has diameter less than $\delta$.

\subsection*{Shadows in $\partial\Ga$}
For $R>0$ and $\ga_1,\ga_2 \in \Ga$,
the shadow of the ball $B_R(\ga_2)$ viewed from $\ga_1$
is given by
 $$O_R(\ga_1, \ga_2 )=
 \{x\in \partial\G: \text{
$[\ga_1,x]\cap B_R(\ga_2)\neq\emptyset$ for some geodesic ray $[\ga_1,x]$}\}.$$
Shadows satisfy the equivariance property: for any $\ga,\ga_1,\ga_2\in\Ga$ and $R>0$,
\begin{equation}\label{eq.EV}
\ga O_R(\ga_1,\ga_2)=O_R(\ga\ga_1,\ga\ga_2).
\end{equation}

\begin{lemma}\label{lem.R2}
There exist $R_0>1$ and $N_0>0$ such that the following holds: if $\ga_1,\ga_2\in\Ga$ with $|\ga_1|,|\ga_2|\ge N_0$ satisfies $|\ga_1\ga_2|=|\ga_1|+|\ga_2|$, then for all $R\ge R_0$,
$$
O_{R}(\ga_1\ga_2,e)\cap O_{R}(\ga_1\ga_2,\ga_1)\cap O_{R}(\ga_1,e)\neq\emptyset.
$$
\end{lemma}
\begin{proof}
Since $|\ga_1\ga_2|=|\ga_1|+|\ga_2|$, there exists a geodesic segment $[\ga_1\ga_2,e]$ passing through $\ga_1$, say $\alpha=(\ga_1\ga_2,\cdots,\ga_1,\cdots, e)$.
Since $\Ga$ is word hyperbolic, 
there exists $C>0$ such that $\alpha$ lies in the $C$-neighborhood of some geodesic line, say $(\cdots,u_{-1},u_0,u_1,\cdots)$.
 Set $N_0:=4C$. 
Choose $u_m, u_n,$ and $u_\ell$ to be elements closest to $\ga_1\ga_2$, $\ga_1$, and $e$, respectively.
 
We claim that $|m-\ell |\geq\op{max}(|m-n|,|n-\ell|)$.
By the triangle inequality,
\begin{align*}
|n-\ell|&=d_{\mathsf w}(u_n,u_\ell)\leq d_{\mathsf w}(\ga_1 ,e)+2C=|\ga_1|+2C;\\
|m-n|&=d_{\mathsf w}(u_m,u_n)\leq d_{\mathsf w}(\ga_1\ga_2 ,\ga_1)+2C=| \ga_2|+2C.
\end{align*}
Since $|\ga_1\ga_2|=|\ga_1|+|\ga_2|$ and
$
|\ga_1\ga_2| \leq d_{\mathsf w}(u_m,u_\ell)+2C=|m-\ell|+2C,
$
it follows that 
\begin{align*}|\ga_2|-2C &\geq  \max(|\ga_1|,|\ga_2|)-2 C +N_0
\\ &=\max(|\ga_1|,|\ga_2|)+2 C\\& \geq \max(|n-\ell|, |m-n|).
\end{align*}
This proves the claim.

Now possibly after flipping the geodesic, we may assume that $m\le \ell$. Then the claim implies
that $\ell -m=|m-n|+ |n-\ell |$ and hence $m\le n\le \ell$.
Set $x:=[u_0,u_1,u_2,\cdots]\in\partial\Ga$. Choose geodesic rays $[\ga_1\ga_2,x)$ and $[\ga_1,x)$.
Since the Hausdorff distance between $[\ga_1\ga_2,x)$ and the ray $(u_m,u_{m+1}, \cdots)$ is at most $d_{\mathsf w}(\ga_1\ga_2,u_m)+\delta\leq C+\delta$, it follows that there exist $v_1,v_2\in\Ga$ lying on $[\ga_1\ga_2,x]$ such that $d_{\mathsf w}(u_n,v_1)<C+\delta$ and $d_{\mathsf w}(u_\ell,v_2)<C+\delta$.
Since the Hausdorff distance between $[\ga_1,x)$ and the ray $(u_n,u_{n+1},\cdots)$ is at most $d_{\mathsf w}(\ga_1,u_n)+\delta<C+\delta$, there exists $v_3\in\Ga$ lying on $[\ga_1,x)$ such that $d_{\mathsf w}(u_\ell,v_3)<C+\delta$.
These altogether imply that
$$
x\in O_{2C+\delta}(\ga_1\ga_2,e)\cap O_{2C+\delta}(\ga_1\ga_2,\ga_1)\cap O_{2C+\delta}(\ga_1,e).
$$
\end{proof}
In the rest of this section, we assume that $\Ga$ is an Anosov subgroup of $G$.
The following Morse property of Kapovich-Leeb-Porti \cite[Prop. 5.16]{KLP}  says that a discrete geodesic line (resp. ray) of $\Ga$ is contained in a uniform neighborhood of some $A$-orbit (resp. $A^+$-orbit) in $X$.
\begin{prop}[Morse property]\label{prop.KLP}\label{Morse}
For any Anosov subgroup $\Ga<G$,
there exists $R_1>0$ such that
\begin{enumerate}
\item
If $(\cdots,\ga_{-1},\ga_0,\ga_1,\cdots)$ is a geodesic line in $(\G, d_{\mathsf w})$, then
$$
\sup_{k\in\bb Z} d(\ga_k o,gAo)\leq R_1
$$
 for any $g\in G$ such that $g^+=\zeta([\ga_0,\ga_1,\cdots])$, $g^-=\zeta([\ga_0,\ga_{-1},\cdots])$.
\item
If $(\ga_0,\ga_1,\cdots)$ is a geodesic ray  in $(\G, d_{\mathsf w})$,  then 
$$
\sup_{k\in\bb N}d(\ga_k o,gA^+o)\leq R_1
$$
where $g\in \ga_0 K$ is the unique element satisfying $g^+=\zeta([\ga_0,\ga_1,\cdots])$.
\end{enumerate}
\end{prop}
Using this proposition, we will show that shadows in the Gromov boundary $\partial\Ga$ are mapped to shadows in the Furstenberg boundary $\cal F$ by the limit map $\zeta : \partial\Ga\to\La$ (Proposition \ref{prop.SS}).
We will need the following lemma:
\begin{lemma}\label{lem.WL}
There exists $C>0$ such that for all $\ga\in\Ga$, $\norm{\mu(\ga)}\leq C |\ga|$. In particular,
$d(o, \ga o)\le C d_{\mathsf w} (e, \ga)$.
\end{lemma}
\begin{proof}
We use notations from Lemma \ref{lem.Tits}. Consider the norm
$\norm{\cdot}_*:=\sum_{\alpha\in\Pi}|\chi_\alpha(\cdot)|$ on $\fa$. Let $\ga\in\Ga$ be arbitrary, and write $\ga=s_1\cdots s_\ell$ with $s_i\in S$ and $\ell=|\ga|$.
Since ${\chi_\alpha(\mu(g))}=
\log \norm{\rho_\alpha(g)}_\alpha$ for all $g\in G$ and $\norm{\rho_\alpha(s_1\cdots s_\ell)}_\alpha\leq\norm{\rho_\alpha(s_1)}_\alpha\cdots\norm{\rho_\alpha(s_\ell)}_\alpha$, it follows that 
for each $\alpha\in\Pi$,
$$
\chi_\alpha(\mu(s_1\cdots s_\ell))\leq \chi_\alpha(\mu(s_1))+\cdots+\chi_\alpha(\mu(s_\ell)).
$$

Noting that $\chi_\alpha$ is positive on $\fa^+$, we have
\begin{align*}
&\norm{\mu(\ga)}_*= \sum_{\alpha\in\Pi}|\chi_\alpha(\mu(\ga))|= \sum_{\alpha\in\Pi}\chi_\alpha(\mu(\ga))\\
&\leq  \sum_{\alpha\in\Pi}\big(\chi_\alpha  (\mu(s_1))+\cdots+ \chi_\alpha(\mu(s_\ell))\big)\le C |\ga|
\end{align*}
where $C:= \max\left\{\sum_{\alpha\in\Pi}\chi_\alpha(\mu(s)):s\in S\right\}$. Since
$\norm{\cdot}$ and $\norm{\cdot}_*$ are comparable, this finishes the proof.
\end{proof}

\begin{prop}[Shadows go to shadows]\label{prop.SS}
There exists $c>0$ such that for all $R>1$ and $\ga,\ga'\in\Ga$,
$$\zeta (O_R(\ga', \ga ))\subset O_{c R} (\ga' o, \ga o).$$
\end{prop}
\begin{proof} 
By \eqref{eq.EV}, it suffices to consider the case $\ga'=e$.
Let $x\in O_R(e,\ga)$.
By the definition of $O_R(e,\ga)$, there exists a geodesic ray $(\ga_0'=e,\ga_1',\ga_2',\cdots)$ representing $x$ such that $d_{\mathsf{w}}(\ga_m',\ga)<R$ for some $m\in\bb N$. 
Let $R_1>0$ be the constant from Proposition \ref{prop.KLP}, and $k\in K$ be an element such that $k^+=\zeta([e,\ga_1',\ga_2',\cdots])$.
Then by Proposition \ref{prop.KLP}(2), there exists $a\in A^+$ such that $d(\ga_m' o,kao)\leq R_1$.
By Lemma \ref{lem.WL}, we have
$$
d(\ga o, \ga_m'o)=\norm{\mu(\ga^{-1}\ga_m')}<Cd_{\mathsf w}(\ga,\ga_m')<CR.$$
Therefore
$$
d(\ga o,ka o)\leq d(\ga o,\ga_m' o)+d(\ga_m' o,kao)\leq CR+R_1.
$$
This implies that $\zeta(x)\in O_{CR+R_1}(o,\ga o)$.
Since $R>1$, the conclusion follows by setting $c:=C+R_1$.
\end{proof}

\begin{cor}\label{cor.R2}
There exists $R_2 >0$ such that for all $\ga_1,\ga_2\in\Ga$ with $|\ga_1\ga_2|=|\ga_1|+|\ga_2|$, we have
$$
\norm{\mu(\ga_1\ga_2)-\mu(\ga_1)-\mu(\ga_2)}\leq R_2 .
$$
\end{cor}
\begin{proof}
Let $N_0$ and $R_0$ be given by Lemma \ref{lem.R2}.
If one of $|\ga_1|,|\ga_2|$ is less than $N_0$, then the claim holds by Lemma \ref{com}.
Now assume that $|\ga_1|,|\ga_2|\ge N_0$.
Then by Lemma \ref{lem.R2} and Proposition \ref{prop.SS}, we can choose
$$
\xi\in O_{cR_0}(\ga_1\ga_2o,o)\cap O_{cR_0}(\ga_1\ga_2o,\ga_1o)\cap O_{cR_0}(\ga_1o,o)
$$
where $c$ is as in Proposition \ref{prop.SS}.
By Lemma \ref{lem.shadow1} and the cocycle identity
$$
\beta_{\xi}(\ga_1\ga_2 o,o)=\beta_{\xi}(\ga_1\ga_2 o,\ga_1 o)+\beta_{\xi}(\ga_1 o,o),
$$
we have
$$
\norm{\underline a(\ga_1\ga_2 o,o)-\underline a(\ga_1o,o)-\underline a(\ga_2o,o)}\leq 3\kappa c R_0.
$$
Since $\underline a(go,o)=\op i\mu(g)$ for all $g\in G$ and $\op i$ preserves $\norm{\cdot}$, 
$$
\norm{\mu(\ga_1\ga_2)-\mu(\ga_1)-\mu(\ga_2)}\leq 3\kappa c R_0.
$$
\end{proof}

\medskip
\noindent
\textbf{Proof of Proposition \ref{prop.AC}:}
We may assume that $p=o$ by Lemma \ref{com}.
Let $\ga\in\Ga$ and $\xi\in\La$ be arbitrary. 
If $\ga=\ga_1\ga_2$, we have  
\begin{align*}
\beta_\xi(\ga o,o)&=\beta_\xi(\ga o,\ga_1 o)-\beta_\xi(o,\ga_1 o).
\end{align*}
We claim that we can find $\ga_1,\ga_2\in\Ga$ so that $\ga=\ga_1\ga_2$, $|\ga|=|\ga_1|+|\ga_2|$, and 
\begin{equation}\label{eq.in}
\xi\in O_{c(\delta+1)}(\ga o,\ga_1o)\cap O_{c(\delta+1)}(o,\ga_1o)
\end{equation}
where $c>0$ is as in Proposition \ref{prop.SS}.

If $\xi\in O_{c(\delta+1)}(o,\ga o)$, then we may simply set $\ga_1=\ga$ and $\ga_2=e$.
In general, we find $\ga_1$ as follows. Consider a geodesic triangle $\Delta$ whose vertices are $e, \ga\in\Ga$, and $\zeta^{-1}(\xi)\ \in\partial\Ga$.
Since $\Ga$ is hyperbolic, we can find three points on $\Delta$, one on each edge, whose diameter is less than $\delta$.
Let $\ga_1\in\Ga$ be the point on the geodesic segment joining $e$ and $\ga$, and
set $\ga_2:=\ga_1^{-1}\ga$.
We then have $|\ga|=|\ga_1|+|\ga_2|$, and 
$\zeta^{-1}(\xi) \in O_{\delta}(\ga,\ga_1)\cap O_{\delta}(e,\ga_1)$. Now the claim follows from  Proposition \ref{prop.SS}.

Therefore, by  Lemma \ref{lem.shadow1},
$$\max (\|\beta_{\xi}(\ga o, \ga_1 o)-\mu(\ga_2^{-1})\| ,\|\beta_{\xi} (o, \ga_1 o)-\mu (\ga_1)\| )
 \le \kappa c(\delta+1)$$
and hence
$$\| \beta_{\xi}(\ga o, o) +\mu(\ga_1)-\mu(\ga_2^{-1})\| \le  2\kappa c(\delta+1).$$
Since $|\ga|=|\ga_1|+|\ga_2|$ and $S$ is symmetric, we have $|\ga^{-1}|=|\ga_1^{-1}|+|\ga_2^{-1}|$.
As $\underline a(\ga o,o)=\mu(\ga^{-1})$, we have, by Corollary \ref{cor.R2},
$$
\norm{\underline a(\ga o,o)-\mu(\ga_1^{-1})-\mu(\ga_2^{-1})}\leq R_2.
$$
Hence it suffices to set $C:=\max (2\kappa c(\delta+1), R_2)$. $\qed$ 

\section{Virtual visual metrics via $\psi$-Gromov product}\label{sec.vis}
In this section, we let $\Gamma<G$ be an Anosov subgroup, and fix $\psi\in \dg$.
The main aim here is to show that exponentiating
the following $\psi$-Gromov product defines a virtual visual metric on $\La$
up to a small power.
\begin{definition}\label{def.GP0}
\rm
The $\psi$-Gromov product based at $o$  is a function $\cal F^{(2)}\to \br$ defined as follows:
for any $(\xi_1, \xi_2)\in \cal F^{(2)}$,
$$
[\xi_1,\xi_2]_{\psi, o}:=\psi(\cal G(\xi_1,\xi_2))
$$
where $\cal G$ is the $\fa$-valued Gromov product defined in Definition \ref{def.GP}.
For $p=g(o)\in X$, we set $$
[\xi_1,\xi_2]_{\psi, p}:=[g^{-1} \xi_1,g^{-1} \xi_2]_{\psi, o}.
$$
\end{definition}
For simplicity, we set $[\xi_1,\xi_2]_{ p} :=[\xi_1,\xi_2]_{\psi, p}$.

Define $d_p=d_{\psi, p} : \cal F^{(2)}\to\bb R_{\ge 0}$ by 
\be d_p(\xi_1,\xi_2)= e^{- [\xi_1,\xi_2]_{p}}. \ee
 It follows from \eqref{eq.id1} that for all $g\in G$, $p\in X$, and $(\xi_1,\xi_2)\in \cal F^{(2)}$, we have
\begin{equation}\label{eq.id3}
d_{gp}(\xi_1,\xi_2)=e^{-\psi(\beta_{\xi_1}(gp, p)+\op i\beta_{\xi_2}(gp, p))}d_p(\xi_1,\xi_2)=d_p(g^{-1}\xi_1,g^{-1}\xi_2).
\end{equation}

We set $[\xi, \xi]_p=+\infty$ and $d_p(\xi,\xi)=0$ for all $\xi\in \cal F$.
By the antipodal property \eqref{anti}, $[\cdot, \cdot]_p$ and
$d_p$ are defined on all of $\La\times \La$. The following is the main theorem of this section:
\begin{theorem}\label{prop.met}
Fix $p\in X$.
For all sufficiently small $\e>0$, there exist a metric $d_\e=d_\e(p)
$ on $\La$  and a constant $C_\e=C_\e(p)>0$  such that
for all $\xi_1, \xi_2\in \La$,
$$
{C_\e}^{-1} d_{\psi,p}(\xi_1,\xi_2)^\e\leq  d_\e(\xi_1,\xi_2)\leq C_\e d_{\psi,p}(\xi_1,\xi_2)^\e.
$$
\end{theorem}

This is an analogue of \cite[Part III, Prop. 3.21]{BH} for Gromov hyperbolic spaces.
 
\subsection*{Weak ultrametric inequality}
A well-known construction \cite[Section 7.3]{GdlH} shows the existence of a metric in Theorem \ref{prop.met}, provided there exists $C>0$ such that for all $\xi_1,\xi_2,\xi_3\in\La$,  we have
\begin{enumerate}
\item
(weak symmetry) $d_p(\xi_1,\xi_2)\leq e^{C} d_p(\xi_2,\xi_1)$;
\item
(weak ultrametric inequality) $d_p(\xi_1,\xi_3)\leq e^{C }\op{max}(d_p(\xi_1,\xi_2),d_p(\xi_2,\xi_3))$.
\end{enumerate}
Hence Theorem \ref{prop.met} follows from the following proposition:
\begin{prop}\label{lem.WSU}
There exists $C=C(p)>0$ such that for all $\xi_1,\xi_2,\xi_3\in\La$, we have
\begin{align*}
 [\xi_1,\xi_2]_p &\geq  [\xi_2,\xi_1]_p-C ;\\
 [\xi_1,\xi_3]_p &\geq\op{min}( [\xi_1,\xi_2]_p , [\xi_2,\xi_3]_p )-C.
\end{align*}
\end{prop}

In the case of $X=\bb H^2$, the classical Gromov product satisfies that there exists a uniform constant $C>0$ such that for any $x,y\in\partial\bb H^2$,
$$
| \cal G(x,y)-2 d(o,z)|\leq C
$$
where $z$ is the unique projection of $o$ to the geodesic connecting $x$ and $y$.
In the following lemma \ref{lem.OP}, we establish the analogous property for $\fa$-valued Gromov products 
on $\La\times \La$ using the Morse property of Anosov groups.

For $\ga\in\Ga$ and any geodesic segment $\alpha$ in $\Ga$, we define the set of projections of $\ga$ to $\alpha$ by
$$
\pi_\alpha(\ga):=\{\ga'\in\alpha : d_{\mathsf w}(\ga,\ga')=d_{\mathsf w}(\ga,\alpha) \}.
$$
Since $\Ga$ is hyperbolic, the diameter of $\pi_\alpha(\ga)$ is less than $4\delta$.

\begin{lemma}\label{lem.OP}
There exists  $C_1>0$  such that for any $x\neq y$ in $\partial\Ga$ and $\ga\in\pi_{[x,y]}(e)$, we have
$$
\norm{\cal G(\zeta(x),\zeta(y))- (\mu( \ga )+\op i\mu(\ga ))}\leq C_1.
$$
In particular, $\cal G$ is almost symmetric on $\La$: for any $\xi_1\ne \xi_2\in \La$,
$$
\norm{\cal G(\xi_1,\xi_2)-\cal G(\xi_2, \xi_1)}\leq 2C_1.
$$
\end{lemma}
\begin{proof}
Let $\alpha:=(u_0=e,u_1,u_2,\cdots)$ and $\alpha':=(v_0 =e,v_1 ,v_2 ,\cdots)$ be geodesic representatives of $x$ and $y$, respectively.
Let $\ga\in \pi_{[x,y]}(e)$ be arbitrary, and  $f,g,h\in G$ be elements satisfying the following:
\begin{itemize}
\item
$f(o)=o$\text{ and } $f^+=\zeta(x)$;
\item
$g(o)=o$\text{ and } $g^+=\zeta(y)$;
\item
$h^+=\zeta(x)$\text{ and } $h^-=\zeta(y)$.
\end{itemize}

Applying Proposition \ref{prop.KLP}(1) to the geodesic line $[x,y]$, we have 
$$d(ho, \ga o)<R_1$$  after replacing $h$ with some element of $hA$.
Hence by Lemma \ref{com}, there exists $C'=C'(R_1)>0$ such that 
\begin{equation}\label{eq.LgL}
\norm{\mu(h)-\mu(\ga)}\leq C'.
\end{equation}

Noting
$$
\cal G(\zeta(x),\zeta(y))=\cal G(h^+,h^-)= \beta_{h^+}(o,ho)+\op i\beta_{h^-}(o,ho) ,
$$
it is now sufficient to show that for some uniform constant $C_1>0$,
$$
\norm{\beta_{h^+}(o,ho)-\mu(h)}\leq C_1 \;\; \text{ and }\;\; \norm{\beta_{h^-}(o,ho)- \mu(h)}\leq C_1.
$$
By Lemma \ref{lem.shadow1}, this claim follows if we show
 \be\label{oho} h^+, h^-\in O_{R}(o,ho)\ee
 for some uniform constant $R>0$.

Since $\Ga$ is hyperbolic, the diameter of the set
$
\pi_{\alpha'}(x)\cup \pi_\alpha(y)\cup\pi_{[x,y]}(e)
$
is at most $C\delta$ for some uniform constant $C>1$.
In particular, we can find $k,\ell\in\bb N$ such that the set $\{u_k,v_\ell ,\ga\}$ has diameter less than $C\delta$.
Applying 
Proposition \ref{prop.KLP}(2) to the geodesic ray $\alpha$, we find $a_1\in A^+$ such that 
$$d(fa_1o,u_k o)<R_1.$$
Since $d_{\mathsf w}(u_k,\ga)=|u_k^{-1}\ga|\le  C\delta$, we have 
$$
d(u_k o,\ga o)=\|\mu(u_k^{-1} \ga ) \|\le 
\sup\{\norm{\mu(\ga')}:|\ga'|\leq C \delta\}.
$$
Therefore
\begin{align*}
d(fa_1o,ho)&\leq d(fa_1o,u_k o)+d(u_k o,\ga o)+d(\ga o,ho)\\
& \le 2R_1+\sup\{\norm{\mu(\ga')}:|\ga'|\leq C \delta\}.
\end{align*}
Setting $R:=2R_1+\sup\{\norm{\mu(\ga')}:|\ga'|\leq C \delta\}$, 
it follows that $h^+=f^+\in O_{R}(o,ho)$. 
Similar argument shows that $h^-=g^+\in O_{R}(o,ho)$. This proves \eqref{oho}.
\end{proof}

\begin{lemma}\label{lem.FW}
For any compact subset $C\subset X$, the set $\{ \beta_{\xi}(p,o) : \xi\in\cal F, p\in C\}$ is bounded.
\end{lemma}
\begin{proof} This follows from Lemma \ref{lem.SP} by setting $\psi=\sum_{\alpha\in \Pi}\varpi_\alpha$.
\end{proof}

\medskip
\noindent
\textbf{Proof of Proposition \ref{lem.WSU}.}
Observe that the identity \eqref{eq.id3} gives that for any $\xi_1\ne\xi_2\in \La$,
$$
[\xi_1,\xi_2]_p-[\xi_1,\xi_2]_o=\psi(\beta_{\xi_1}(p, o)+\op i\beta_{\xi_2}(p, o)).
$$
Now Lemma \ref{lem.FW} shows the existence of $C=C(p,\psi)>0$ such that $|[\xi_1,\xi_2]_p-[\xi_1,\xi_2]_o|\leq C$. Therefore it suffices to show the claim for $p=o$.
 The first inequality is an immediate consequence of Lemma \ref{lem.OP}  with  $C>2C_1\norm{\psi}$.

To show the second inequality, let $C_1>0$ be a constant from Lemma \ref{lem.OP} so that we have
\begin{equation}\label{eq.CT}
[\xi_1,\xi_3]_o \geq  \psi(\mu( \ga_2)+\op i\mu( \ga_2 ))-C_1\norm{\psi}.
\end{equation}
Set $x_i:=\zeta^{-1}(\xi_i)\in\partial\Ga$ for $i=1,2,3$.
For each $i$, we fix a geodesic line $[x_i,x_{i+1}]$ joining $x_i$ and $x_{i+1}$, and choose $\ga_{i+2}\in \pi_{[x_i,x_{i+1}]}(e)$, where all the indices are to be interpreted mod 3.
By the hyperbolicity of $\Ga$, for some uniform constant $C>0$,
there exists $1\le i\le 3$ such that
$d_{\mathsf w}(\ga_i,\ga_{i+1})< C\delta$ and for some $\ga'\in [e, \ga_{i+2}]$,
the diameter of $\{\ga',\ga_i,\ga_{i+1}\}$ is at most $C\delta$.

We first consider the case when $i=1$.
 Since $$d(\ga_1o,\ga_2o)\leq d_{\mathsf w}(\ga_1,\ga_2)\op{max}_{s\in S} d(o,s)<C \delta \op{max}_{s\in S} d(o,s),$$ it follows from Lemma \ref{com} that for some uniform $C_2>0$,
 $$
 \norm{\mu(\ga_1)-\mu(\ga_2)}\leq C_2.
 $$
 
 In view of \eqref{eq.CT}, we now obtain
\begin{align*}
 [\xi_1,\xi_3]_o &\geq   \psi(\mu( \ga_1)+\op i\mu( \ga_1 ))-C_1\norm{\psi}-2C_2\\
&\geq   [\xi_2,\xi_3]_o -2C_1\norm{\psi}-2C_2\text{ by Lemma }\ref{lem.OP}\\
&\geq  \op{min}( [\xi_1,\xi_2]_o , [\xi_2,\xi_3]_o )-2C_1\norm{\psi}-2C_2.
\end{align*}
The case $i=2$ can be handled similarly by interchanging the roles of $\ga_2$ and $\ga_3$.
Finally in the case when $i=3$, let $R_2$ be as in Corollary \ref{cor.R2}.
Since $(e,\cdots,\ga',\cdots,\ga_2)$ is a geodesic, we have by Corollary \ref{cor.R2} that
$$
\norm{\mu(\ga_2)-\mu(\ga')-\mu(\ga'^{-1}\ga_2)}\leq R_2.
$$
By \eqref{eq.CT} and the fact $\psi(\mu( (\ga'^{-1}\ga_2)^{\pm1}))\geq 0$, we deduce
\begin{align*}
 [\xi_1,\xi_3]_o &\geq  \psi(\mu(\ga')+\op i\mu(\ga') )-C_1\norm{\psi}-2R_2\norm{\psi}\\
 &\geq \psi(\mu(\ga_1)+\op i\mu(\ga_1) )-(C_1+2C_2+2R_2)\norm{\psi},
\end{align*}
as the diameter of $\{\ga',\ga_1,\ga_3\}$ is less than $\delta$.
The rest of the proof is similar to the case $i=1$.
$\qed$

\subsection*{Covering lemma}

Using Theorem \ref{prop.met}, we obtain:
\begin{lemma}[Triangle inequality] \label{lem.triangle}
There exists $N=N(\psi, p)\ge 1$ such that for any  $\xi_1,\xi_2,\xi_3\in\La$,
$$
d_p(\xi_1,\xi_3)\leq N \big(d_p(\xi_1,\xi_2) +d_p(\xi_2,\xi_3)\big).$$
In particular, $d_p(\xi_1, \xi_2)\le N d_p(\xi_2, \xi_1)$.
Moreover, $N(\psi, p)$ can be taken uniformly for all $p$ in a fixed compact subset of $X$.
\end{lemma}
\begin{proof}
Choose $\e>0$ sufficiently small so that Theorem \ref{prop.met} holds, and set $d:=d_\e$, $C:=C_\e$.
We then have
$$
d_p(\xi_1,\xi_3)^\e\leq C  d(\xi_1,\xi_3)\leq C (d(\xi_1,\xi_2)+d(\xi_2,\xi_3)) \leq C^2 (d_p(\xi_1,\xi_2)^{\e}+ d_p(\xi_2,\xi_3)^{\e}).
$$
Since $(a^\e+b^\e)^{1/\e}\leq \alpha (a+b)$ for all $a,b\geq 0$ for some uniform constant $\alpha=\alpha(\e)>0$, it suffices to take the $1/\e$ power in each side of the above.
Now the second part follows from \eqref{eq.id3} and Lemma \ref{lem.FW}. \end{proof}

For $\xi\in \La$ and $r>0$, set 
$$
\bb B_p(\xi,r):=\{\eta \in \La : d_{\psi,p}(\xi,\eta)<r\}.
$$

\begin{lemma}[Covering lemma] \label{inc} 
There exists $N_0(\psi, p)\ge 1$ satisfying the following:
for any finite collection  $\bb B_p (\xi_1, r_1), \cdots, \bb B_p (\xi_n, r_n)$ with $\xi_i\in \La$ and $r_i>0$, there exists a disjoint subcollection $\bb B_p (\xi_{i_1}, r_{i_1}),
\cdots , \bb B_p (\xi_{i_\ell}, r_{i_\ell})$ such that
$$\bb B_p (\xi_1, r_1)\cup \cdots \cup \bb B_p (\xi_n, r_n)\subset
\bb B_p (\xi_{i_1}, 3N_0 r_{i_1})\cup
\cdots \cup \bb B_p (\xi_{i_\ell}, 3N_0 r_{i_\ell}).$$
Moreover, $N_0(\psi, p)$ can be taken uniformly for all $p$ in a fixed compact subset of $X$.
\end{lemma}
\begin{proof}  Let $N=N(\psi,p)$ be as given by Lemma \ref{lem.triangle}.
For simplicity, set $B_i:=\bb B_p(\xi_i,r_i) $.
We may asume $r_1\geq\cdots\geq r_n$ without loss of generality and define inductively
$$i_1=1,\,i_{j+1}=\min\{i>i_j: B_i \cap(B_{i_1}\cup\cdots\cup B_{i_j})=\emptyset\},$$
as long as possible, to obtain a maximal disjoint subcollection $\{B_{i_1},\cdots,B_{i_\ell}\}$.
Let $\xi\in B_{j}$ for some $1\le j\le n$.
Then there exists $1\le  k\le \ell$ such that $B_j\cap B_{i_k}\ne \emptyset$ and $r_{i_k}\ge r_j$.
Choose $\eta \in B_j\cap B_{i_k}$.
Then  by Lemma \ref{lem.triangle},
we have $d_p(\eta,\xi)\le N (d_p(\eta, \xi_j) + d_p(\xi_j, \xi)  )
< 2 N^2 r_j \le 2N^2 r_{i_k}$ and $d_p(\xi_{i_k},\eta)<r_{i_k}$.
Hence
\begin{align*}
d_p(\xi_{i_k},\xi )\leq N(d_p(\xi_{i_k},\eta )+d_p(\eta,  \xi))< 3 N^3r_{i_k}.
\end{align*}
Hence it suffices to set $N_0:=N^3$. \end{proof}

\subsection*{Comparing Gromov products}
Although we will not be using it in the rest of the paper,
we record the following theorem which is of independent interest:
\begin{thm}\label{GG}  For any $\psi\in \dg$, there exist $c_1=c_1(\psi)\ge 1, c_2=c_2(\psi)>0$ such that
for all $x\ne y\in \partial \Gamma$,
$$ c_1^{-1} (x|y) - c_2 \le  \psi (\mathcal G (\zeta(x), \zeta(y)))\le c_1 (x|y) +c_2.$$
\end{thm}

Note that if  $\ga\in\pi_{[x,y]}(e)$ for $x\neq y$ in $\partial\Ga$, then
$ |  (x|y)-|\ga| | \le C$ for some uniform constant $C>0$ (cf. \cite{BH}).
Given this fact, Theorem \ref{GG} follows immediately from  Lemma \ref{lem.OP} and the following lemma:

\begin{lem} For any $\psi\in \dg$,  there exist  constants $C_\psi, c_\psi>0$ such that  for all $\ga\in\Ga$,
$$C_\psi^{-1}|\ga| -c_\psi \le \psi(\mu(\ga))\le C_\psi |\ga|.$$
\end{lem}
\begin{proof}
Since $\psi>0$ on $\L_\Ga$, we have 
$$0<d:=\min_{\|u\|=1, u\in \L_\Ga} \psi(u)\le D:=\max_{\|u\|=1, u\in \L_\Ga} \psi(u)<\infty.$$
Hence $d\|\mu(\ga)\| \le \psi(\mu(\ga)) \le D \|\mu(\ga)\|$ for all $\ga \in \Ga$.
So the upper bound follows from Lemma \ref{lem.WL}, and the lower bound follows from
the well-known property of Anosov groups
that for some uniform $C>0$, $C^{-1}|\ga|-C\leq \norm{\mu(\ga)}$ for all $\ga\in \Ga$ \cite{GW}.
\end{proof}

\section{Conical points, divergence type and classification of $\PS$ measures}\label{class-s}
In this section, we show that for Anosov groups, the space of all PS-measures
on $\La$ is homeomorphic to $\dg$.

\subsection*{Conical limit points} 
 For a discrete subgroup $\Gamma<G$ and $x\in \Ga\ba G$, we mean by $\limsup xA^+M$ the set of all limit points $\lim\limits_{i\to\infty} xa_im_i$ where $a_i\to\infty$ in $A^+$ and $m_i\in M$.

\begin{Def}[Conical limit points] \label{def.cone}
\rm
We call $\xi\in\cal F$  a conical limit point of $\Gamma$ if $\limsup \Ga gA^+M\neq\emptyset$ for some $g\in G$ with $g^+=\xi$.
Equivalently, $\xi\in \cal F$ is conical if there exists $R>0$ such that
$\xi \in O_R(o, \ga_i o)$ for some sequence $\gamma_i\to \infty$ in $\Ga$.
We denote by $\La_c$ the set of all conical limit points of $\Gamma$.
\end{Def}

\begin{lem}\label{st}
 Let 
$\Ga$ be a Zariski dense discrete subgroup and
 $\mathfrak c\subset \inte\mathfrak a^+\cup\{0\}$ be a closed convex cone whose interior contains $\L_\Gamma-\{0\}$.
If $\gamma_i g_ia_i$ is a bounded  sequence
 where $g_i\in G$ is bounded, $\ga_i\in \Ga$ and  $a_i\to \infty$ in $A^+$,
then  
\be\label{ur} \log a_i\in \mathfrak c\quad \text{for all sufficiently large $i$}.\ee
In particular, for any $x\in \Gamma\ba G$,
$\limsup xA^+M$ coincides with the set
$$ \{\lim\limits_{i\to\infty}xa_im_i: m_i\in M, \;\log a_i\to \infty \text{ in } \mathfrak c \}.$$
\end{lem}
\begin{proof} 
 As $g_i$ and $\gamma_i g_ia_i$ are bounded sequences, the sequence
  $\mu(\gamma_i^{-1}) -\log a_i$ is also bounded by Lemma \ref{com}.
Hence $\lim\limits_{i\to\infty}\frac{\log a_i}{\norm{\log a_i}}$ belongs to the asymptotic cone of $\mu(\Ga)$, which is equal to $\cal L_\Ga$ \cite[Thm. 1.2]{Ben}.
Since $\cal L_\Ga-\{0\}\subset\op{int}\mathfrak c$,
it follows that $\log a_i\in \mathfrak c$ for all large $i$.
\end{proof}

We note that for $\mathfrak c\subset \inte\mathfrak a^+\cup\{0\}$ as above, there exists a constant $s=s(\mathfrak c)>0$ such that
for any sequence $v_i\to \infty$ in $\mathfrak c$,  
$$ \min_{\alpha\in \Pi} \liminf_{i\to \infty} \alpha( \tfrac{ v_i}{\|v_i\|}) \ge s.$$

We deduce from  Proposition \ref{prop.KLP}: recall $\cal E=\{[g]\in \Ga\ba G: g^+\in \La\}$.
\begin{Prop}\label{prop.con} For $\Ga$ Anosov, there exist a compact subset $\cal Q$ of $\cal E$
and a closed convex cone $\mathfrak c\subset \inte \mathfrak a^+\cup\{0\}$
such that for any $x\in \cal E$, there exists $\log a_i\to\infty$ in $\mathfrak c$ such that 
 $$xa_i\in \cal Q\quad \text{ for all $i\ge 1$}.$$
In particular, $\La=\La_c.$
\end{Prop}
\begin{proof} For $\Gamma$ Anosov, we have $\cal L_\Ga-\{0\}\subset\op{int}\mathfrak a^+ $ by Theorem \ref{pop}.
Hence we can find a closed convex cone $\mathfrak c\subset \inte\mathfrak \fa^+\cup\{0\}$ such that $\L_\Ga\subset \inte \mathfrak c \cup \{0\}$.
We first check that $\La_c\subset\La$.
Let $g^+\in\La_c$ for some $g\in G$.
Then there exists $\ga_i\in\Ga$ and $a_im_i\to\infty$ in $A^+M$ such that $\ga_i ga_im_i$ is bounded.
By Lemma \ref{st}, it follows that $\log a_i\to\infty$ in $\mathfrak c$. In particular, $a_i\to \infty$ regularly in $A^+$.
Hence by Lemma \ref{st2}, $ga_io\to g^+$ as $i\to\infty$.
Since $d(ga_io,\ga_i^{-1}o)$ is bounded, $\ga_i^{-1}o\to g^+$ as $i\to\infty$.
By Lemma \ref{lem.lim}, $g^+\in\La$.

Let $g^+=\xi\in\La$ and $z\in\partial\Ga$ be such that $\xi=\zeta(z)$.
Choose a geodesic ray $r=(\ga_0=e,\ga_1,\ga_2,\cdots)$ representing $z$.
Note that if $g^+=h^+$, then for any sequence $a_i\to\infty$ in $A^+$, there exists $b_i\in A^+$ such that $d(ga_io,hb_io)\le 1$ for all sufficiently large $i$.
Hence we may assume that $g\in K$ by replacing $g$ by an element of $gP$.
By Proposition \ref{prop.KLP}, $\ga_i o$ is contained in the $R_1$-neighborhood of $gA^+o$, with $R_1$ given therein.
Hence for some $a_i\to \infty$ in $A^+$, $\Gamma\ba \Gamma g a_i\in \cal Q$ where
 $\cal Q=\Gamma\ba \Gamma \{h\in G: d(o, ho)\le R_1\} \, \cap \cal E$. Hence $g^+\in \La_c$. Moreover, by Lemma \ref{st}, $\log a_i\in \mathfrak c$ for all sufficiently large $i$.
 This finishes the proof.
\end{proof}

\subsection*{Classification of $\PS$ measures on $\La$}
\begin{lemma}\label{lem.inj}
Let $\psi_i\in\fa^*$ and $\nu_{\psi_i}$ be a $(\Ga,\psi_i)$-$\PS$ measure for $i=1,2$.
If $\nu_{\psi_1}=\nu_{\psi_2}$, then $\psi_1=\psi_2$.
\end{lemma}
\begin{proof}
Suppose that $\nu_{\psi_1}=\nu_{\psi_2}$.
Then for all $\ga\in\Ga$ and $\xi\in\La$, we have
$$
{\psi_1(\beta_\xi(e,\ga))}={\psi_2(\beta_\xi(e,\ga))}.
$$ 
By setting $\xi=y_{\ga}$, we obtain $\la(\ga)\in\op{ker}(\psi_1-\psi_2)$ for all $\ga\in\Ga$, by Lemma \ref{lem.el}.
Hence $\cal L_\Ga\subset\op{ker}(\psi_1-\psi_2)$.
Since $\cal L_\Ga$ has nonempty interior \cite[Thm. 1.2]{Ben}, this implies that $\psi_1=\psi_2$. 
\end{proof}
\begin{remark}
\rm{When $\Ga$ is an Anosov subgroup, $\nu_{\psi_1}$ and $\nu_{\psi_2}$ are even mutually singular to each other whenever $\psi_1\neq\psi_2$ (See Theorem \ref{thm.MS} below).}
\end{remark}
We denote by $\cal S_\Ga$ the space of all $\PS$ measures on $\La$.
Recall that for $\psi\in\dg$, Quint constructed a $(\Ga,\psi)$-$\PS$ measure on $\La$ \cite{Quint2}.
In the Anosov case, such a measure is unique, which we denote
by $\nu_\psi$.
By Lemma \ref{lem.inj}, the map $\psi\mapsto\nu_\psi$ from $\dg$ to $\cal S_\Ga$ is injective.

\begin{theorem}\label{thm.bij}
For $\Gamma<G$ Anosov,
the map $\psi\mapsto \nu_\psi$ is a homeomorphism between $\dg$ and $\cal S_\Ga$.
\end{theorem}
In the rank one case, there exists a unique
Patterson-Sullivan measure on $\La$ and its dimension is given by
the critical exponent of $\Ga$. The above theorem generalizes such phenomenon.

To prove that the map $\psi\mapsto\nu_{\psi}$ is surjective, 
we need the following shadow lemma.
 It was first presented in \cite[Thm. 3.3]{Alb} and then in \cite[Thm. 8.2]{Quint2} in slightly different forms.
\begin{lemma}[Size of shadow]  \label{lem.SH2}
Let $\Ga<G$ be a Zariski dense discrete subgroup. For $\psi\in\fa^*$, let $\nu_\psi$ be a 
 $(\Ga,\psi)$-conformal measure on $\cal F$. Then
 \begin{enumerate}
 \item for some $R=R(\nu_\psi)>0$,  we have $ c:=\inf_{\ga\in\Ga}\nu_\psi (O_R(\ga o,o))>0$;
 \item  for all $r\ge R$ and for all $\ga\in\Ga$,
$$
c \cdot e^{-\|\psi\| \kappa r}  e^{-\psi(\mu(\ga))}\leq \nu_\psi(O_r(o,\ga o))\leq  e^{\|\psi\| \kappa r}  e^{-\psi(\mu(\ga))}
$$
where $\kappa>0$ is a constant given in Lemma \ref{lem.shadow1}.
In particular, if $\La=\La_c$ in addition, then $\nu_\psi$ is atom-free on $\La$.
\end{enumerate}
\end{lemma}
\begin{proof}
Suppose that there exist sequences $R_i\to\infty$ and $\ga_i\in\Ga$ such that for all $i\ge 1$, $\nu_\psi(O_{R_i}(\ga_i^{-1}o,o))<1/i$.
Write $\ga_i=k_ia_i\ell_i\in KA^+K$ with $k_i, \ell_i\in K$
and $a_i\in A^+$.
Passing to a subsequence, we may assume that $\ell_i\to\ell_0$ as $i\to\infty$.

We claim that 
\be\label{limsup} \limsup O_{R_i}(a_i^{-1}o,o)\supset N^+e^+.\ee
Fix an arbitrary $h\in N^+$ and $a_ih=k_ib_in_i\in KAN$ be the Iwasawa decomposition of $a_ih$. Then
the Iwasawa decomposition of $a_iha_i^{-1}$ is given by
$$
a_iha_i^{-1}=k_i(b_ia_i^{-1})(a_in_ia_i^{-1})\in KAN.
$$
Since $a_iha_i^{-1}$ is uniformly bounded, both $b_ia_i^{-1}$ and $a_in_ia_i^{-1}$ are also uniformly bounded for all $i$.
It follows that the sequence $n_i\in N$ is uniformly bounded as well.
To prove the claim, we observe that for all $i\ge 1$,
\begin{align*}
& h^+\in O_{R_i}(a_i^{-1}o,o)\\
\Leftrightarrow &\,a_ih^+\in O_{R_i}(o,a_io)\\
\Leftrightarrow &\,k_i^+\in O_{R_i}(o,a_io)\\
\Leftrightarrow &\,k_iA^+o \cap B(a_io,R_i)\neq\emptyset\\
\Leftrightarrow &\,a_i^{-1}k_iA^+ o\cap  B(o,R_i)\neq\emptyset\\
\Leftrightarrow &\,hn_i^{-1}b_i^{-1}A^+o\cap B(o,R_i)\neq\emptyset.
\end{align*}
On the other hand, 
by the uniform boundedness of $n_i$ and $b_i^{-1}a_i$ and since $R_i\to \infty$, we have $hn_i^{-1}(b_i^{-1}a_i)o\in B(o,R_i)$ for all sufficiently large $i$.
Hence the claim \eqref{limsup} follows.

Since $O_{R_i}(\ga_i^{-1}o,o)=\ell_i^{-1}O_{R_i}(a_i^{-1}o,o)$, the hypothesis 
$\nu_\psi(O_{R_i}(\ga_i^{-1}o,o))<1/i$ now implies that $\nu_\psi( \ell_0^{-1}N^+e^+)=0$ by Claim \eqref{limsup}.
Since $N^+e^+$ is Zariski open in $\cal F$, this contradicts the fact that $\La\subset \op{supp}\nu_\psi$ is Zariski dense in $\cal F$.
This proves the claim (1).

Now let $\ga\in\Ga$ and $r>R$ be arbitrary.
By Lemma \ref{lem.shadow1}, for all $\xi\in O_r(\ga^{-1} o,o)$, we have
$$
\norm{\beta_\xi(\ga^{-1}o, o)-\mu(\ga)}\leq\kappa r.
$$
Since
\begin{align*}
\nu_\psi(O_r(o,\ga o))=\int_{O_r(\ga^{-1}o,o)}e^{-\psi(\beta_\xi(\ga^{-1} o, o))}\,d\nu_\psi (\xi),
\end{align*}
(2) then follows from (1).

Suppose that $\La=\La_c$.
Then for any $\xi\in \La$, there exist $r>0$ and a sequence $\ga_i\to \infty$ in $\Gamma$
such that $\xi\in \bigcap_{i} O_r(o, \ga_i o)$. 
Since  $\nu_\psi(\xi)\le \nu_\psi(O_r(o,\ga_i o))
\le C e^{-\psi (\mu(\ga_i))}$ and $\psi(\mu(\ga_i))\to +\infty$ as $i\to \infty$,
$\nu_{\psi}(\xi)=0$. Hence the second claim follows. \end{proof}

\begin{lemma}\label{lem.dom}\cite[Lem. III.1.3]{Quint1}
Let $\theta : \fa\to\bb R$ be a continuous function satisfying $\theta(tu)=t\theta(u)$ for all $t\geq 0$ and $u\in\fa$.
If $\theta(u)>\psi_\Ga(u)$ for all $u\in\fa-\{0\}$, then
$$
\sum_{\ga\in\Ga}e^{-\theta(\mu(\ga))}<\infty.
$$

If there exists $u\in \fa$ such that $\theta(u) <\psi_\Ga(u)$, then $
\sum_{\ga\in\Ga}e^{-\theta(\mu(\ga))}=\infty.
$
\end{lemma}

\begin{lemma} \label{lem.DT} Suppose that $\Gamma$ is 
Zariski dense. Let $\psi\in \fa^*$.
 If there exists a $(\Gamma, \psi)$-conformal measure $\nu_\psi$ such that $\nu_\psi(\Lambda_c)>0$, then 
$$\sum_{\ga\in\Ga}e^{-\psi(\mu(\ga))}=\infty.$$
Moreover, $\psi(v)=\psi_\Ga(v)$ for some non-zero $v\in \L_\Ga$.
\end{lemma}

\begin{proof} Note that $\La_c$ is an increasing union $\bigcup_{N=1}^\infty \La_N$, where
$$
\La_N:=
\{ \xi\in\La :\text{ there exists }\ga_i\to\infty\text{ in }\Ga\text{ such that }\xi\in O_N(o,\ga_i o)\}.
$$
Hence $\nu_\psi(\La_{N_0})>0$ for some $N_0\ge 1$. Fix $N \ge \max \{R(\nu), N_0\}$,
 and set $C':=  e^{\|\psi\| \kappa N} $ where $R(\nu)$ is as in Lemma \ref{lem.SH2}.
 Observe that for any $m\ge 1$,
$$
\La_N\subset\bigcup_{\ga\in\Ga, d(o,\ga o)>m} O_N(o,\ga o).
$$
Hence
\begin{align*}
0<\nu_\psi(\La_N)\leq \sum_{d(o,\ga o)>m}\nu_\psi(O_N(o,\ga o))\leq C'\sum_{d(o,\ga o)>m}e^{-\psi(\mu(\ga))}.
\end{align*}
Since $m>1$ is arbitrary, the first claim follows.

We note that $\psi\geq\psi_\Ga$ by \cite[Thm. 8.1]{Quint2}. 
If $\psi(u)>\psi_\Ga(u)$ for all $u\in\L_\Ga-\{0\}$, and hence for all $u\in \fa-\{0\}$, 
then Lemma \ref{lem.dom} implies $
\sum_{\ga\in\Ga}e^{-\psi(\mu(\ga))}<\infty.
$
This is a contradiction by the first claim.
\end{proof}

When $\Ga$ is Anosov, $\La=\La_c$ and hence by Theorem \ref{pop}(5),
\begin{cor} 
If $\Ga$ is Anosov, then  $\sum_{\ga\in\Ga}e^{-\psi(\mu(\ga))}=\infty$
for any $\psi\in\dg$.
\end{cor}

\noindent{\bf Proof of Theorem \ref{thm.bij}:} In order to prove surjectivity,
suppose that there exists a $(\Ga,\psi)$-$\PS$ measure, say $\nu_\psi$, for $\psi\in\fa^*$.
By Lemma \ref{lem.DT},
$\psi(v)=\psi_\Ga(v)$ for some non-zero $v\in \L_\Ga$. By Theorem \ref{pop}(1),
it follows that $\psi\in D_\Ga^\star$, proving surjectivity.

If $\psi_i\to \psi$ in $\dg$, then any weak-limit of $\nu_{\psi_i}$ is a $(\Gamma, \psi)$-PS measure.
By the uniqueness of $(\Ga,\psi)$-conformal measure, $\nu_{\psi_i}$ converges to $\nu_\psi$ as $i\to \infty$. Hence the map $\psi\mapsto \nu_\psi$ is continuous.
Now suppose $\nu_{\psi_i}\to \nu_{\psi}$ where $\psi_i, \psi\in \dg$.
Since the closed cone generated by $\mu(\Ga)$ is equal to $\L_\Ga$ that has non-empty interior, we can find $\ga_1, \cdots,
\ga_k\in \Ga$ such that $\mu(\ga_i)$'s form a basis of $\fa$. For each $\ga_\ell$ and $r>0$,
we have $\nu_{\psi_i}(O_r(o, \ga_\ell o)) \to \nu_{\psi}(O_r(o, \ga_\ell o))$. Hence
$\{ (\psi_i-\psi)(\mu(\ga_\ell)):i=1,2,\cdots \} $ is bounded by Lemma \ref{lem.SH2}.
It follows that $\{\psi_i:i=1,2, \cdots\}$ is a relatively compact subset of $\fa^*$.
Suppose that $\phi\in \fa^*$ is a limit of $\{\psi_i\}$. By passing to a subsequence, assume that $\psi_i\to \phi\in \fa^*$. 
Since $\nu_{\psi_i}\to \nu_{\psi}$, it follows that $\nu_\psi$ is a $(\Gamma, \phi)$-PS measure.
Since $\psi\mapsto \nu_\psi$ is a bijection between $\dg$ and $\cal S_\Ga$, 
we have  $\phi\in \dg$ and $\nu_{\phi}=\nu_{\psi}$.
 By Lemma \ref{lem.inj}, we have $\phi=\psi$. Since every limit of the sequence $\psi_i$ is
$\psi$, it follows that $\psi_i$ converges to $\psi$ as $i\to \infty$. This finishes the proof.

\medskip

\noindent{\bf Critical exponents.}
Recall the definition of $\mathsf O_\Ga$ from \eqref{ooo}. For each unit vector $w\in \mathsf O_\Ga$,
consider the Poincare series
 $$\mathcal P_w(s, p):=\sum_{\ga\in \Ga} 
 e^{-s \langle w, \underline{a}(p, \ga p) \rangle} $$
Define the critical exponent $\delta_w$ to be the abscissa of convergence of $\mathcal  P_w(s, p)$, which is independent of $p\in G/K$:
 \be\label{critical} \delta_w:=\inf\{s\in \br :  \mathcal P_w(s, p) <\infty\}.\ee

\begin{cor} Let $w\in \mathsf O_\Ga$ be a unit vector.
\begin{enumerate}
\item   For $w=\tfrac{\nabla \psi_\Gamma (u)}{\| \nabla \psi_\Gamma (u)\|} \in \mathsf O_\Ga$ for
$u\in \inte\L_\G$ with $\|u\|=1$, we have
$$\delta_w=\|\nabla \psi_\Gamma (u)\| .$$
 In particular, $w\mapsto \delta_w$ is analytic on $\{w\in \mathsf O_\Ga: \|w\|=1\}$.

\item For any $p\in G/K$, $\mathcal P_w(\delta_w, p)=\infty$.

\end{enumerate}
\end{cor}

\begin{proof} (1) follows from Lemmas \ref{lem.dom} and \ref{lem.DT}  together with  Theorem \ref{pop}.
(2) is a direct consequence of  Lemma \ref{lem.DT}.
\end{proof}

\section{Myrberg limit points of Anosov groups}
In this section, we discuss the notion of Myrberg limit points.
We show that for Anosov groups, the set of Myrberg limit points has full measure for any $\PS$ measure on $\La$. In the rank one case, this was proved by Tukia \cite[Thm. 4A]{Tuk}.
Let $\Ga<G$ be a Zariski dense discrete subgroup.
\begin{definition}[Myrberg points]\label{def.MYR} \rm Let $p\in X$. We call a point $\xi_0\in \La$ a Myrberg limit point for $\Gamma$ if, for any  $\xi\ne \eta$ in $\La$, there exists a sequence $ \gamma_i \in \Gamma$ such that
$\gamma_i p\to \xi$ and $\gamma_i \xi_0\to \eta$ as $i\to\infty.$
\end{definition}
Note that this definition is independent of the choice of $p\in X$ by Lemma \ref{ssame}.
We denote by $\La_M\subset \La$ the set of all Myrberg limit points for $\Gamma$.

When $G$ is of rank one, a Myrberg limit point $\xi\in\La$ is characterized by the property that any geodesic ray toward $\xi$ is dense in the space of all geodesics connecting limit points.
The following proposition generalizes this to a general Anosov subgroup.

\begin{proposition}\label{prop.Myr1}
Let $\Ga$ be Anosov.
We have $\xi_0\in \La_M$ if and only if for any $g\in G$ with $g^+=\xi_0$,
$$
\limsup \Gamma\ba\Ga g A^+M =\Om.
$$
\end{proposition}
Let $\Ga<G$ be an Anosov subgroup for the rest of this section.
\begin{lem}\label{lem.target}  
Let $b_i\in A$ be a sequence tending to $\infty$ such that $w^{-1}b_i^{-1}w\in A^+$ for some $w\in \cal W$.
If $\gamma_i gb_i \to h$  for some $h,g\in G$ and $\ga_i\in\Ga$, then $\lim_{i\to \infty}
\gamma_i go= hw^+\in\La$.
In particular, if $b_i\in A^+$, then $\lim_{i\to \infty} \gamma_igo= h^-$.
\end{lem}
\begin{proof} 
Let $c_i:=h^{-1}\ga_i g b_i$ and  $a_i:=w^{-1}b_i^{-1}w\in A^+$.
Then $gw=\gamma_i^{-1} hc_i wa_i$. Hence
by Lemma \ref{st}, $a_i\to \infty$ regularly in $A^+$.
 Lemma \ref{st2} implies that $hc_iwa_i (o)\to hw^+$. 
Since $\gamma_i gw=hc_i wa_i$,  we have $\gamma_i gw(o)=\gamma_i go \to hw^+$.
This proves the first claim by Lemma \ref{lem.lim}.
 If $b_i\in A^+$, then $w_0^{-1} b_i^{-1} w_0\in A^+$. 
Since $w_0^+=e^-$, the last claim follows.
\end{proof}

The following is proved in \cite[Coro. 5.8]{KLP}:
\begin{thm}[The limit map as a continuous extension of the orbit map] \label{lem.ext}
For any $p\in X$, the map $\Gamma\cup \partial \Gamma \to X\cup \cal F$
given by  $\ga\mapsto\ga p$ for $\gamma\in \G$ and $x\mapsto \zeta(x)$  for $x\in \partial \Ga$
is continuous.
\end{thm}
We need the following basic fact about word hyperbolic groups.
\begin{lemma}\label{lem.C}
Let $x\neq y$ in $\partial\Ga$.
If $\ga_i\in \Ga$ is an infinite sequence such that $(\ga_i x, \ga_i y)\to (x',y')\in \partial \Ga \times \partial \Ga$, then $\ga_i$ converges to either $x'$ or $y'$.
\end{lemma}
 
\begin{proof}
Choose a geodesic line $[x,y]$, and  its representative
 $(\cdots,u_2,u_1,u_0=v_0,v_1,v_2,\cdots)$.
Note that $x=[u_0,u_1,u_2,\cdots]$ and $y=[v_0,v_1,v_2,\cdots]$.
It suffices to show that $\ga_i u_0$ converges to either $x'$ or $y'$.
Suppose not.
Then by passing to a subsequence we have $\ga_iu_0\to z'$ where $z'\not\in\{x',y'\}$.
Since $(z'|x'), (z'|y')<\infty$,  there exists a subsequence $n_k$ such that $\sup_{k}(\ga_k u_{0}|\ga_k u_{n_k})+(\ga_k u_{0}|\ga_k v_{n_k})<\infty$.
Let $\cal L_k^-:=[\ga_ku_0,\ga_ku_{n_k}]$ and $\cal L_k^+:=[\ga_k u_0,\ga_k v_{n_k}]$,
so that $\sup_{k}d_{\mathsf w}(e,\cal L_k^\pm)< \infty$.
The thin triangle property of the hyperbolic group $\Ga$ implies that if the projection of $e$ to the geodesic segment $\cal L_k^-\cup\cal L_k^+$ lies in $\cal L_k^\pm$, then $d_{\mathsf w}(e,\ga_ku_0)$ is equal to $d_{\mathsf w}(e,\cal L_k^\mp)$ up to a uniform additive constant.
And hence $d_{\mathsf w}(e,\ga_ku_0)$ is uniformly bounded, which is a contradiction as $\ga_k\to\infty$ as $k\to\infty$. \end{proof}

The following is immediate from  Theorem \ref{lem.ext} and Lemma \ref{lem.C}:
\begin{cor}\label{lem.pm}
Let $\gamma_i\in \Gamma$ be an infinite sequence such that $(\ga_i\xi,\ga_i \eta )\to(\xi',\eta')$ in $\Lambda^{(2)}$ as $i\to\infty$.
Then for any $p\in X$,  $\gamma_i p$ converges  to either $\xi'$ or $\eta'$.
\end{cor}

\begin{lemma}
Let $g\in G$ be such that $g^\pm\in\La$.
If $\lim_{i\to \infty} \ga_ig^\pm=\xi$ for some infinite sequence $\ga_i\in\Ga$, then $\lim_{i\to \infty}\ga_ig o=\xi$.
\end{lemma}
\begin{proof}
Set $x^{\pm}:=\zeta^{-1}(g^\pm)$ and $y=\zeta^{-1}(\xi)$.
Since $\zeta : \partial\Ga\to\La$ is a homeomorphism, we have $\ga_i x^\pm\to y$ as $i\to\infty$.
By Lemma \ref{lem.C}, we have $\ga_i\to y$ as $i\to\infty$. By Theorem \ref{lem.ext},
we get $\lim_{i\to \infty}\ga_io=\xi$.
By Lemma \ref{ssame}, $\lim_{i\to \infty} \ga_i go= \xi$ as desired. \end{proof}

Since the fibers of the visual map $g\mapsto g^+$ are $P$-orbits, the following lemma
is an easy consequence of  the regularity lemma \ref{st}.
\begin{lemma}\label{lem.GH}
If $g, h\in G$ satisfy $g^+=h^+$, then 
$$\limsup \Ga gA^+M=\limsup\Ga hA^+M.$$
\end{lemma}
\noindent
\textbf{Proof of Proposition \ref{prop.Myr1}.} 
Set $\tilde \Om:=\{g\in G: g^{\pm}\in \La\}$.
Suppose $\xi_0\in \La_M$ and $g^+=\xi_0$. We claim that $\Gamma gA^+M=\tilde \Om$.
By Lemma \ref{lem.GH}, we may assume that $g^-\in\La$. Let $h\in \tilde \Om$.
 As $\xi_0\in \La_M$,
there exists $\gamma_i\in \Gamma$ such that $\gamma_i g^+ \to h^+$ and $\gamma_i g o \to h^-$.
By Lemma \ref{lem.C}, by passing to a subsequence, $\gamma_i g^-$ converges to $h^-$.
Therefore $\gamma_i gAM\to hAM$ in  $G/AM$; there exists $b_i m_i \in AM$ such that $\gamma_i g b_i m_i \to h$. We claim that $b_i\in A^+$ for all large $i$. If not,
by passing to a subsequence, we have  $m_i^{-1}$ converges to some $m_0\in M$ and there exists $w\in \cal W-\{e\}$ such that
 $a_i:=w^{-1} b_iw\in A^+$.  Then $\gamma_i g w a_i \to h m_0 w$ . By Lemma \ref{lem.target},
$\gamma_i g o\to hm_0w^-$, and hence $hm_0w^-=h^-$. It follows that
 $w=e$, yielding a contradiction.  Therefore  $h\in \limsup \Gamma gA^+M$, proving the claim.
 
 Now suppose that  $\limsup \Gamma g A^+M =\tilde\Omega$. We claim that $g^+\in \La_M$.
Let $\xi\ne \xi'$ in $\La$, and let $h\in G$ be such that $h^+=\xi$ and $h^-=\xi'$.
By the hypothesis and Lemma \ref{st}, there exist $\gamma_i\in \G$, $m_i\in M$
 and $a_i\to \infty$ regularly in $A^+$ such that $\gamma_i g a_im_i\to h$ 
in $G$. Then $\gamma_i g^+\to h^+=\xi$.
By Lemma \ref{lem.target}, $\gamma_i go\to h^-=\xi'$.
Hence $g^+\in \La_M$. 
This finishes the proof of Proposition \ref{prop.Myr1}.
$\qed$

\begin{lem}\label{lem.ZZ}
Let $\psi\in \dg$, and $(\xi, \xi'), (\eta_1,\eta_2)\in \La^{(2)}$.
If $\ga_i\in\Ga$ and $t_i\to+\infty$ are such that 
$$
\lim_{i\to \infty} (\ga_i\xi,\ga_i\xi',t_i+\psi(\beta_{\ga_i\xi}(o,\ga_io)))= (\eta_1, \eta_2, 0),$$
then $\lim_{i\to \infty}  \ga_i(o)=\eta_2$.
\end{lem}
\begin{proof}
Write $x=\zeta^{-1}(\xi)$, $x'=\zeta^{-1}(\xi')$, $y_1=\zeta^{-1}(\eta_1)$, $y_2=\zeta^{-1}(\eta_2)$, and choose $u\in[x ,x']$.
Since the triangle $[\ga_ix,\ga_ix']\cup[\ga_iu,\ga_ix]\cup [\ga_iu,\ga_ix']$ is $\delta$-thin, it follows that for all $i$, either $\ga_ix\in O_\delta(u,\ga_i u)$ or $\ga_ix'\in O_\delta(u,\ga_i u)$.
We claim the latter holds for all large $i$.

Suppose not.
Then by passing to a subsequence, we may assume that $\ga_ix\in O_\delta(u,\ga_iu)$ for all $i$.
Then  by Proposition \ref{prop.SS} and Lemma \ref{lem.shadow1},
there exists a uniform constant $c>0$ such that $\ga_i\xi\in O_{c(\delta+1)}(uo,\ga_iuo)$ and 
$$
|\psi(\beta_{\ga_i\xi}(uo,\ga_iuo)))-\psi(\mu(\ga_i))|<\norm{\psi}\kappa c(\delta+1).
$$
Since $\psi(\mu(\ga_i))\to+\infty$ as $i\to\infty$ by Lemma \ref{concave0}, and  $\psi(\beta_{\ga_i\xi}(uo,\ga_i uo)))$ and $\psi(\beta_{\ga_i\xi}( o,\ga_i o)))$ are uniformly close to each other, $\psi(\beta_{\ga_i\xi}( o,\ga_i o)))\to+\infty$.
This contradicts the hypothesis that the sequence $t_i+\psi(\beta_{\ga_i\xi}( o,\ga_i o))$ converges to a finite number as $i\to\infty$.
It follows that for all sufficiently large $i$, 
\begin{equation}\label{eq.BB}
\ga_ix'\in O_\delta(u,\ga_iu).
\end{equation}
On the other hand, $\ga_iu\to y_\ell$ for some $\ell\in\{1,2\}$ by Lemma \ref{lem.C}.
Since $\ga_ix'\to y_2$ and $O_\delta(u,\ga_i u)$ converges to $y_\ell$, \eqref{eq.BB} implies that $\ga_iu\to y_2$.
Therefore $\ga_io\to\eta_2$ by Lemma \ref{lem.ext}.
\end{proof}

\begin{theorem}\label{fullm} For any $\PS$-measure $\nu$ on $\Lambda$, $\nu(\Lambda_M )=1$.
\end{theorem}
\begin{proof}
By Theorem \ref{thm.bij}, $\nu=\nu_\psi$ for some $\psi\in\dg$.
Let $\mathsf m_\psi$ be the $\br:=\{\tau_s:s\in \br\}$-ergodic finite measure on $\Ga\ba\La^{(2)}\times\bb R$ in Theorem \ref{thm.reparam}.
Let $Z_\psi\subset\Ga\ba\La^{(2)}\times\bb R$ denote the set of elements with dense $\br_+$-orbits, and $\tilde Z_\psi$ be its lift in $\La^{(2)}\times\bb R$.
By the Birkhoff ergodic theorem, $Z_\psi$ has full $\mathsf m_\psi$-measure, and hence
$\nu(\pi(\tilde Z_\psi))=\nu(\La)$ where $\pi :\La^{(2)}\times\bb R\to\La$ denotes the projection map $\pi(\xi,\eta,t)=\xi$. 
It is now sufficient to prove that $\pi(\tilde Z_\psi)\subset\La_{\op M}$.

Let $\xi\in \pi(\tilde Z_\psi)$ and $(\eta_1,\eta_2)\in\La^{(2)}$ be arbitrary.
We need to show that there exists $\ga_i\in\Ga$ such that $\ga_i\xi\to\eta_1$ and $\ga_io\to\eta_2$ as $i\to\infty$.
Choose $(\xi,\xi',0)\in \tilde Z_\psi$.
By definition, we can find $\ga_i\in\Ga$ and $t_i\to+\infty$ such that the sequence
$$
\ga_i (\tau_{t_i}.(\xi,\xi',0))=\ga_i(\xi,\xi',t_i)=(\ga_i\xi,\ga_i\xi',t_i+\psi(\beta_{\ga_i\xi}(o,\ga_io)))
$$
converges to $(\eta_1, \eta_2, 0)$. 
Since $\ga_io\to\eta_2$ by Lemma \ref{lem.ZZ}, this finishes the proof.

\end{proof}

In the rank one case, the $\BMS$ measure is finite, and
$A=\{a_t\}$ is the union of $A^+=\{a_t:t\geq 0\}$ and $A^-=\{a_t:t\leq 0\}$. The $AM$-ergodicity of the $\BMS$ measure implies that for almost all $x\in\Ga\ba G$, $xA^\pm M$ is dense in $\Omega=\{x\in \Ga\ba G :x^\pm\in\La\}$.
In general, $A=\cup_{w\in\cal W} wA^+w^{-1}$, and we have the following corollary of Theorem \ref{fullm}:
\begin{cor}\label{cor.AM}
Let $\psi\in\dg$.
For $m_\psi^{\BMS}$-almost all $x\in\Omega$, 
each $ xA^+ M$ and $xw_0A^+M$  is dense in $\Om$.
\end{cor}
\begin{proof} 
Note that for $x=\Ga g \in \Om$, $xwA^+M $ is dense in $\Om$ if and only if $gw^+\in \La_M$ by Proposition \ref{prop.Myr1}.
For $w=e$ (resp. $w=w_0$),  the claim follows as $\nu_\psi(\La_M)=1$ (resp. $\nu_{\psi\circ \op i} (\La_M)=1$)
by Theorem \ref{fullm}.
\end{proof}

We also observe:
\begin{lem} For any  $x\in \cal E$ and $w\in \cal W-\{e, w_0\}$,
the map $A^+M\to xwA^+M$ is proper.
\end{lem}

\begin{proof}
Note that if $(g^+,gw^+)\in\cal F^{(2)}$ for $g\in G$ and $w\in\cal W$, then $w=w_0$.
Choose $g\in G$ so that $\Ga g=x\in \cal E$.
Since $g^+\in\La$ and $\La\times \La -\{(\xi, \xi)\}\subset \cal F^{(2)}$ by the antipodality, $gw^+\in\La$ can happen only for $w\in\{ e,w_0\}$.
Suppose for some $\ga_i\in\Ga$ and $a_i\to\infty$ in $A^+$, $\ga_i gwa_i$ converges to some $h\in G$ as $i\to\infty$.
This means that $d(gwa_i,\ga_i^{-1}h)\to 0$ as $i\to\infty$, and hence $gw^+\in\La$.
Hence, for each $w\in \cal W-\{ e,w_0\}$,
$\limsup xwA^+M=\emptyset$, proving the claim.
\end{proof}

\section{Criterion for ergodicity via essential values}\label{sec.gen}
In this section, let $\Gamma<G$ be a Zariski dense discrete subgroup, and let
$\nu_\psi$ be a $(\Gamma, \psi)$-conformal measure on $\cal F$ for $\psi\in \fa^*$.
Consider the action of  $G$ on $\F\times \fa$ by 
$$g (\xi, v)= (g\xi, v+\beta_\xi( g^{-1},e)).$$
Then the map $g\mapsto (g^+, b:= \beta_{g^+} (e,g))$ induces a $G$-equivariant homeomorphism $G/NM\simeq \cal F\times \mathfrak a$. Using this homeomorphism, we  define a $\Ga$-invariant Radon measure $\widehat\nu_\psi$ on $G/NM\simeq \cal F\times\mathfrak a$ by
$$
d\widehat\nu_\psi (gNM)=d\nu_\psi(g^+)e^{ \psi(b)}\,db.
$$
Since $dm^{\BR}_\psi =d\widehat\nu_\psi\,dm\,dn$, the $NM$-ergodicity of $m_\psi^{\BR}$ is equivalent to the $\Ga$-ergodicity of $\widehat\nu_\psi$.
For simplicity, we set $\nu:=\nu_\psi$ and $\widehat \nu:=\widehat\nu_\psi$ for the rest of the section.
Schmidt gave a characterization of $\Ga$-ergodicity of $\widehat \nu$ using the notion of $\nu$-essential values in the rank one case (\cite{Sch2}, see also \cite[Prop. 2.1]{Ro}).
\begin{definition}\rm
An element $v\in\mathfrak a$ is called a $(\nu,\Ga)$-\textit{essential value}, if for any Borel set $B\subset \cal F$ with $\nu(B)>0$ and any $\e>0$, there exists $\ga\in\Ga$ such that 
$$
\nu\left(B\cap\ga^{-1} B\cap\{\xi\in \cal F : \norm{\beta_\xi(\ga^{-1} o,o)-v}<\e\}\right)>0.
$$
\end{definition}

Let ${\mathsf E}_{\nu}={\mathsf E}_{\nu}(\Gamma)$ denote the set of all $(\nu,\Ga)$-essential values in $\mathfrak a$. It is easy to see that
$\mathsf E_{\nu}$ is a closed subgroup of $\fa$.
The main goal of this section is to prove the following criterion of $\Ga$-ergodicity of $\widehat\nu$, which can be considered as a higher rank version of \cite[Prop. 2.1]{Ro} .

\begin{proposition}\label{prop.erg1}
$(G/NM,\Ga,\widehat\nu)$ is ergodic if and only if $(G/P,\Ga,\nu)$ is ergodic and ${\mathsf E}_{\nu}(\Ga)=\mathfrak a$.
\end{proposition}

Fixing $\nu$, we set ${\mathsf E}:={\mathsf E}_\nu(\Ga)$ in the rest of this section.
Our proof of Proposition \ref{prop.erg1} is an easy adaptation of the proof of \cite[Prop. 2.1]{Ro}
to a higher rank case.  We begin with the following lemma .
\begin{lem}\label{contt}
Let $h:G/NM=\cal F\times \fa \to [0,1]$ be a  $\Gamma$-invariant Borel function such that for each $\xi\in \cal F$,
$h(\xi, \cdot)$ is a $C$-Lipschitz function on $\fa$ for some $C>0$ independent of $\xi$.
 Then for each $\log a\in {\mathsf E}$, $h(x a)=h(x)$ for $\widehat\nu$-a.e. $x\in G/NM$.
\end{lem}
\begin{proof} 
Suppose that $\widehat \nu\{x \in G/NM: h(x)\neq h(xa)\}>0$ for some $\log a\in \mathsf E$.
We will then find a subset $\cal A^*=\cal A^*(a) \subset G/NM$ with $\widehat \nu(\cal A^*)>0$ and 
$\gamma\in \Gamma$ such that $h (\gamma^{-1}x) \ne h(x) $ for all $x\in \cal A^*$;
this contradicts the $\Gamma$-invariance of $h$.

By replacing $h$ with $-h$ if necessary, we may assume that $\widehat \nu
\{x  \in G/NM: h(x)< h(xa)\}>0$.
Hence there exist $r,\e>0$ such that
$$
Q_a :=\{x\in G/NM: h(x)<r-C \e<r+C \e<h(xa)\}
$$
has a positive $\widehat\nu$-measure.
Now we can choose a ball $\cal O=B_{\mathfrak a}(v_0,\e/2)\subset \fa$  such that
$$
\widehat\nu((\cal F\times \cal O)\cap Q_a )>0.
$$
Set $
F_a:=\{\xi\in \cal F : (\{\xi\}\times \cal O)\cap Q_a\neq\emptyset\}.
$
We claim that
\begin{equation}\label{eq.cl1}
\text{if }(\xi,w)\in F_a\times \cal O,\text{ then }h(\xi,w+\log a)>r>h(\xi,w).
\end{equation} 
Note that there exists $v\in\mathfrak a$ with $\norm{v}<\e$ such that $(\xi,w+v)\in Q_a$ and hence
\begin{align*}
|h(\xi,w)|&\leq |h(\xi,w)-h(\xi,w+v)|+|h(\xi,w+v)|<C \norm{v}+(r-C \e)\leq r.
\end{align*}
Similarly,
\begin{align*}
|h(\xi,w+\log a)| &\geq |h(\xi,w+v+\log a)|
-|h(\xi,w+\log a)-h(\xi,w+v+\log a)|
\\
&> (r+C \e)-C\norm{v}>r,
\end{align*}
which verifies the claim \eqref{eq.cl1}.

Since $-\log a\in {\mathsf E}$ and $\nu(F_a)>0$, there exists $\ga\in\Ga$ such that
$$
\cal A:=F_a\cap \ga F_a\cap\{\xi\in G/P : \norm{\beta_\xi(o,\ga o )+\log a}<\e/2\}
$$
has a positive $\nu$-measure.
For $\xi\in\cal A$, set
$$
\cal O_\xi:=\{w\in \cal O: w-(\beta_\xi(o,\ga o)+\log a)  \in \cal O\}.
$$
Since $\norm{\beta_\xi(o,\ga o)+\log a}<\e/2$, and $\cal O$ is a Euclidean ball of diameter $\e$, there is a uniform positive lower bound for 
the volume of $\cal O_\xi$.
It follows that
$$
\cal A^*:=\bigcup_{\xi\in\cal A}\{\xi\}\times \cal O_\xi
$$
has positive $\widehat \nu$-measure.
We now claim that $h\circ\ga^{-1}>h$ on $\cal A^*$. 

Let $(\xi, w)\in \cal A^*$.
Since
$(\xi,w)\in F_a\times \cal O$,
\eqref{eq.cl1} implies that $h(\xi, w)<r$.

Write 
$\ga^{-1}(\xi,w)=(\ga^{-1}\xi,w-(\beta_\xi(o,\ga o)+\log a)+\log a).$ Since $(\ga^{-1}\xi,w-(\beta_\xi(o,\ga o)+\log a))\in F_a\times \cal O$,
 \eqref{eq.cl1} says that $$
h(\ga^{-1}(\xi,w))>r ;$$
this proves the claim. 
\end{proof}

\noindent
\textbf{Proof of Proposition \ref{prop.erg1}.}
 Assume that $(G/NM,\Ga,\widehat\nu)$ is ergodic.
Let $\pi : G/NM\to G/P$ denote the projection map.
Since $\pi_*\widehat\nu$ is absolutely continuous with respect to $\nu$, it follows that $(G/P,\Ga,\nu)$ is ergodic.

To show ${\mathsf E}=\mathfrak a$, fix an arbitrary Borel set $B\subset G/P$ of positive $\nu$-measure.
For any $w\in\mathfrak a$ and $\e>0$, we define
$$
\cal B_{w,\e}:=\{(\xi,v)\in G/P\times\mathfrak a : \xi\in B,\, \norm{v-w}<\e\}\subset G/NM.
$$
Observe that 
\begin{align*}
\widehat\nu(\cal B_{0,\e})&=\int_{G/P}\int_{\mathfrak a} \mathbf{1}_{\cal B_{0,\e}}(\xi,b)e^{ \psi(b)}\,db\,d\nu(\xi)\geq \op{Vol}(B_{\mathfrak a}(0,\e))\,e^{-\norm{\psi}\e}\nu(B)>0.
\end{align*}
Hence it follows from the ergodicity of $(G/NM,\Ga,\widehat\nu)$ that $\widehat\nu(G/NM-\Ga\cal B_{0,\e})=0$.
In particular, there exists $\ga\in\Ga$ such that $\widehat\nu(\cal B_{w,\e}\cap \ga\cal B_{0,\e})>0$.
Finally, note that if $(\xi,v)\in \cal B_{w,\e}\cap \ga\cal B_{0,\e}$, then $\xi\in B\cap\ga B$, and 
$$
\norm{\beta_\xi(e,\ga)-w}\leq\norm{\beta_\xi(e,\ga)-v}+\norm{v-w}\leq\e+\e=2\e.
$$
This, together with the fact $\pi_*\widehat\nu\ll\nu$, implies that
$$
\nu(B\cap\ga B\cap\{\xi\in G/P : \norm{\beta_\xi(e,\ga)-w}\leq2\e\})>0,
$$
which finishes the proof of $(\Rightarrow)$.

We now assume that $(G/P,\Ga,\nu)$ is ergodic and ${\mathsf E}=\mathfrak a$.
Let $h : G/NM\to [0,1]$ be a $\Ga$-invariant Borel function. We need to show that $h$ is constant $\widehat\nu$-a.e.
Identifying $\mathfrak a\simeq\bb R^r$ with $r=\text{rank }G$,  for each $\tau=(\tau_1,\cdots,\tau_r)\in\mathfrak a$,
we define a $\Gamma$-invariant Borel
 function $h_\tau: G/NM\to\bb R$ as follows:
$$
h_{\tau}(x)=\int_0^{\tau_1}\cdots\int_0^{\tau_r}h(x\op{exp}(t_1,\cdots,t_r))\,dt_r\,\cdots dt_1.
$$

Note that $h_\tau$  satisfies the hypothesis of Lemma \ref{contt}. Hence by
 the hypothesis ${\mathsf E}_\nu=\fa$,
 for each $a\in A$, $h_{\tau}(x)=h_{\tau}(xa)$ for $\widehat\nu$-a.e. $x\in G/NM$.

Let $\{a_n:n\in\bb N\}$ be a countable dense subset of $A$.
Then there exists $\Omega_n$ of full $\widehat\nu$-measure such that for all $x\in\Omega_n$, $h_\tau(x)=h_\tau(xa_n)$.
Set $\Omega:=\cap_{n=1}^\infty \Omega_n$.
Then for all $x\in\Omega$, we have $h_\tau(x)=h_\tau(xa)$ for all $a\in A$, as $h_\tau(\xi,\cdot)$ is continuous on $\mathfrak a$. 
Now $h_\tau$ is a $\Ga$-invariant function on $G/NM$, which is also $A$-invariant $\widehat\nu$-a.e.

Since $(G/P,\Ga,\nu)$ is ergodic, there exists $c(\tau)\in\bb R$ such that $h_\tau=c(\tau)$ $\widehat\nu$-a.e. on $G/NM$.

Next, fix $1\leq i\leq r$ and $\tau_1,\cdots,\tau_{i-1},\tau_{i+1},\cdots,\tau_r \geq 0$, and define
$$f(t):=(\tau_1,\cdots,\tau_{i-1},t,\tau_{i+1},\cdots,\tau_r)\in\mathfrak a.$$
Then $t\mapsto c(f(t))$ is linear; indeed, by definition, we have 
$$h_{f(t+s)}=h_{f(t)}+h_{f(s)}\circ\exp(te_i)$$ for all $t,s\geq 0$ and hence $c(f(t+s))=c(f(t))+c(f(s))$.
We conclude $c(\tau)=\kappa\tau_1\cdots\tau_r$, for some $\kappa\in\bb R$.

Hence for each $\tau\in\mathfrak a$, $h_\tau=\kappa\tau_1\cdots\tau_r$ $\widehat\nu$-a.e.
Since $|h_{\tau+\sigma}-h_\tau|\leq2^r\norm{\sigma}\norm{\tau}^{r-1}$ 
and hence $\tau\to h_\tau$ is continuous,  using a countable dense subset of $\fa$, 
we conclude there exists a subset $\Omega$ of full $\widehat\nu$-measure such that 
$$h_\tau(x)=\kappa\tau_1\cdots\tau_r\quad\text{ for all $x\in\Omega$ and $\tau\in\mathfrak a$.}$$
By restricting $h_\tau$ to each fiber of $\pi : G/NM\to G/P$, and applying the Lebesque differentiation theorem, we conclude that $\frac{1}{\tau_1\cdots\tau_r}h_\tau(x)\to h(x)$ as $\tau\to0$ for $\widehat\nu$-a.e. $x$.
Consequently, $h=\kappa$ $\widehat\nu$-a.e., finishing the proof.

\section{Ergodicity of $m^{\BR}_\psi$ and classification}
Let $\Gamma<G$ be an Anosov subgroup. 
Recall the $NM$-invariant $\BR$ measure $m_\psi^{\BR}$ defined in \eqref{def.BR}.
We prove the following theorem in this section:

\begin{thm}\label{thm.BR}
For each $\psi\in\dg$, $m_\psi^{\BR}$ is $NM$-ergodic.
\end{thm}
Recall the definition of $\widehat\nu_\psi$ and $\nu_\psi$ from section \ref{sec.gen}.
Since $(\cal F , \Gamma, \nu_{\psi})$ is ergodic by Theorem \ref{pop}, the following proposition
 implies that $( G/NM, \Ga, \widehat\nu_{\psi})$, and hence
$(\Ga\ba G, NM, m_\psi^{\BR})$,  is ergodic by Proposition \ref{prop.erg1}. 

\begin{proposition}\label{prop.EV}\label{posm} Let $\Ga_0$ be a Zariski dense normal subgroup of $\Ga$.
For any $\psi\in\dg$, we have ${\mathsf E}_{\nu_\psi}(\Ga_0)=\fa$.  In particular,
${\mathsf E}_{\nu_\psi}(\Ga)=\fa$.
\end{proposition}
Most of the section is devoted to the proof of Proposition \ref{prop.EV}.
We fix a Zariski dense normal subgroup $\Ga_0$ of $\Ga$.

\begin{lem}\label{fii}
For any finite subset $S_0\subset \la(\Gamma_0)$, the subgroup generated by $\la(\Gamma_0)-S_0$ is dense in $\fa$.
\end{lem}
\begin{proof}
Let $F$ denote the closure of the subgroup generated by $\lambda(\Gamma_0)-S_0$.
Suppose that $F\ne \fa$. Identifying $\fa=\br^r$,
since $F$ is infinite, there exist $1\le k<r$ and $0\le m \le r$ such that
$F=\sum_{i=1}^k \br v_i +\sum_{i=1}^m \z w_i$
where $v_i, w_i$ are linearly independent vectors.
For each $s=\la(\ga)\in S_0$,  $\la(\ga^n)=n \la (\ga)\to \infty$ as $\ga$ is loxodromic. Hence
there exists $n_s\in \N$ so that $n_s\la (\ga)\in F$. Setting $N:=\prod_{s\in S_0} n_s$, we have
$S_0\subset \sum_{i=1}^k \br v_i +N^{-1} \sum_{i=1}^m \z w_i$.

Therefore, the closure of the subgroup generated by $F\cup S_0$ is 
contained in $\sum_{i=1}^k \br v_i + N^{-1}
\sum_{i=1}^m \z w_i$. Since $\la (\Gamma_0)\subset  \sum_{i=1}^k \br v_i + N^{-1}
\sum_{i=1}^m \z w_i$ and $\la(\Gamma_0)$ generates a dense subgroup of $\fa$ \cite{Ben2}, it follows that $k=\text{dim }\fa$, yielding a contradiction.
\end{proof}

\begin{prop}\label{dense1} 
For any $\psi\in \dg$ and $C>0$,
the set $\{\lambda(\ga)\in \fa^+:\ga\in\Ga_0, \, \psi(\lambda(\ga))\geq C\}$ generates a dense subgroup of $\fa$.
\end{prop}
\begin{proof}
 Theorem 3.2 in \cite{Samb3} extends to general Anosov subgroups (see also \cite[Thm. A.2-(2)]{CAR}), and hence the cocycle $c=\psi\circ\sigma$ has a finite exponential growth rate.
In particular, \be\label{fin}
\# \{\lambda(\ga):\ga\in\Ga, \, \psi(\lambda(\ga))< C\}\le
\#\{[\ga]\in[\Ga]: \psi(\lambda(\ga))<C\}<\infty
\ee
where $[\Ga]$ denotes the set of conjugacy classes in $\Ga$.
Hence $\# \{\lambda(\ga):\ga\in\Ga_0, \, \psi(\lambda(\ga))< C\}<\infty$ and the claim follows from Lemma \ref{fii}.
\end{proof}

\begin{lem}\label{comp3} There exists a compact subset $\cal C\subset G$ such that
for any $\xi\in \La$, there exists $g\in \cal C$ such that $g^+=\xi$ and $g^-\in \La$.
\end{lem}
\begin{proof} 
In the Gromov hyperbolic space $\Ga$,
there exists a finite subset $F \subset\Ga$ such that 
for any $x\in \partial \Ga$, there exists $y\in \partial \Ga$ such that $[x,y]\cap F\ne \emptyset$.
It suffices to choose a compact subset
 $\cal C\subset G$ such that
 $\cal C (o)$ contains the $R_1$-neighborhood of $F(o)$ with
 $R_1$ given in Proposition \ref{Morse}.
\end{proof}

We set $$N_0:=\max_{p\in \cal C (o)} N_0(\psi, p)<\infty$$ with $N_0(\psi, p)$ and $\cal C$
given by Lemmas \ref{inc} and \ref{comp3} respectively.

In view of Proposition \ref{dense1},   Proposition \ref{prop.EV} is an immediate consequence of the following:
\begin{prop} \label{ess} 
For any $\ga_0\in \G_0$ with $\psi(\la (\ga_0))\ge  1+ \log 3N_0$,
$$\la(\ga_0)\in {\mathsf E}_{\nu_\psi}(\Ga_0) .$$
\end{prop}
\subsection*{Essential values of $\nu_\psi$} Most of this section is devoted to the proof of this proposition.
We fix  $\ga_0\in\Ga_0$  with
$$\psi(\lambda(\ga_0))\ge \log 3N_0+1.$$

Since $\psi>0$ on $\la(\Gamma)-\{0\}$ by Theorem \ref{pop}(4),
we have
\be\label{ggg00}
\psi (\op{i} \la(\ga_0)) +  \psi(\lambda(\ga_0))>\log 3N_0+1.\ee
\subsection*{Definition of $\cal B_R(\gamma_0,\e)$}
 Let $0<\e<\|\psi\|^{-1}$ be an arbitrary number. We fix $g\in \cal C$  such that $g^+=y_{\ga_0}$ and $g^-\in \La$, given by Lemma \ref{comp3}. 
 Set $p:=go\in \cal C (o) $, $\xi_0:=y_{\ga_0}$ and $\eta:=g^-$.

For $\xi\in \La$ and $r>0$, set
$$
\bb B_p(\xi,r):=\{\eta \in \La : d_{\psi,p}(\xi,\eta)<r\}
$$
where $d_{\psi, p}$  is the virtual visual metric defined in section \ref{sec.vis}.

For each $\ga\in \Ga$, define
$r_p(\ga)>0$ to be the supremum $r\ge 0$ such that
\begin{equation}\label{eq.nbd1}
\max_{\xi\in \bb B_p(\ga \xi_0,3N_0r) } \norm{\beta_\xi(p,\ga\ga_0^{\pm 1} \ga^{-1}p) \mp \la(\gamma_0)  }<\e.
\end{equation}

For each $R>0$,
 we define the family of virtual-balls  as follows:
$$
\cal B_R(\gamma_0, \e) =\{\bb B_p(\ga \xi_0,r) : \ga\in\Ga,  0<r<\min (R, r_p(\ga))\}.
$$
Equivalently, $\bb B_p(\ga\xi_0,r)\in\cal B_R(\ga_0,\e)$ if and only if $r< R$ and one has $\norm{\beta_\eta(p,\ga\ga_0\ga^{-1}p)-\la(\ga_0)}\leq\e$ for all $\eta\in\bb B_p(\ga\xi_0,r)$.

Let $C=C( \psi, p)>0$ be as in Theorem \ref{concave}. Since $\xi_0\in  O_{\e/(8\kappa)}  (\eta,p)$
where $\kappa>0$ is as in Lemma \ref{lem.shadow1}, we can choose $0<s=s(\gamma_0)< R$
small enough such that 
 \be\label{cho}
\bb B_p( \xi_0,e^{\psi(\la(\ga_0)+\op i\,\la(\ga_0))+\frac{1}{2}\norm{\psi}\e+2C}s)\subset 
 O_{\e/(8\kappa)} (\eta,p);
\ee
\be\label{eq.cond2}
\sup_{x\in\bb B_p(\xi_0,e^{2C}s)}  \norm{\beta_x(p,\ga_0^{\pm 1} p)\mp\la(\ga_0)}<\e/4.\ee

For each $\ga\in \Ga$ and $r>0$, set
\begin{align*}
D(\ga \xi_0, r) &:=\bb B_p(\ga \xi_0,\frac{1}{3N_0} e^{-\psi(\underline a(\ga^{-1}p,p)+\op i\underline a(\ga^{-1}p,p))}r).
\end{align*}

\begin{lemma}\label{lem.WD1}\label{spn} Fix $R>0$.
If $\xi\in \La$ and $\ga_i\in\Ga$ is a sequence such that $\gamma_i^{-1} p\to \eta$ and $\gamma_i^{-1}\xi\to \xi_0$ as $i\to\infty$,
then for any $0<r\le s(\gamma_0)$, there exists $i_0=i_0(r)>0$ such  that for all $i\ge i_0$,
 $$D(\ga_i\xi_0,r)\in \cal B_R(\ga_0, \e)\quad \text{and}\quad
\xi \in D({\gamma_i}\xi_0,r).$$ 

In particular, for any $R>0$,
$$\La_M\subset \bigcup_{D\in \cal B_R(\ga_0,\e)} D.$$
\end{lemma}
\begin{proof}
Set $\Gamma_p:=\{\gamma\in \Gamma: \psi(\underline a(\ga^{-1}p,p)+\op i\underline a(\ga^{-1}p,p))>0\}$;
note that $\Gamma -\Gamma_p$ is a finite subset by Lemma \ref{concave0}.
Hence we may assume that for all $i$, $\ga_i\in \Ga_p$.
Since $\gamma_i^{-1}p\to \eta$ as $i\to\infty$, we may assume by Lemma \ref{lem.SC} that for all $i$,
\be\label{oe}
   O_{\e/(8\kappa)} (\eta,p)\subset  O_{\e/(4\kappa)} (\ga_i^{-1}p,p).
\ee

To prove that $D(\ga_i\xi_0,r)\in\cal B_R(\ga_0,\e)$, we need to check that
$$
\max_{\xi'\in \bb B_p(\ga_i\xi_0,3N_0s_i) } \norm{\beta_{\xi'}(p,\ga_i\ga_0^{\pm 1} \ga_i^{-1}p)\mp \la(\gamma_0 )  }<\e,
$$
where $s_i=\frac{1}{3N_0}e^{-\psi(\underline a(\ga_i^{-1}p,p)+\op i\underline a(\ga_i^{-1}p,p))}r$.
Let $\xi'\in\bb B_p(\ga_i\xi_0,3N_0s_i)$.
We only prove that $\norm{\beta_{\xi'}(p,\ga_i\ga_0 \ga_i^{-1}p) - \la(\gamma_0 )  }<\e$, as the other case can be treated similarly.
First, observe that
\begin{align}\label{eq.VV}
d_p( \xi_0,\gamma_i^{-1}\xi')&= d_p( \gamma_i \xi_0,\xi') e^{\psi ( \beta_{\xi_0}(\ga_i^{-1}p, p)+\op i \beta_{\ga_i^{-1}\xi'} (\ga_i^{-1}p, p))
} \notag \\
&\le  e^{-\psi(\underline{a}(\ga_i^{-1}p, p) +\op i \underline{a}(\ga_i^{-1}p, p) )+
 \psi (\beta_{\xi_0}(\ga_i^{-1}p, p)+\op i \beta_{\ga_i^{-1}\xi'} (\ga_i^{-1}p, p)
)}  r\notag \\
 &\le e^{2C}r \text{ by Theorem \ref{concave}}.
\end{align}
Since $r\leq s(\ga_0)$, this implies that
$$
\norm{\beta_{\ga_i^{-1}\xi'}(p,\ga_0 p)-\la(\ga_0)}< \e/4.
$$
Hence, by \eqref{eq.id3}, we have
\begin{align}\label{eq.YY}
d_p(\xi_0,\ga_0^{-1}\ga_i^{-1}\xi' )
&=e^{-\psi(\beta_{\xi_0}(\ga_0 p,p)+\op i\beta_{\ga_i^{-1}\xi'}(\ga_0 p, p))}d_p(\xi_0 ,\ga_i^{-1}\xi')\notag\\
&\le  e^{\psi(\la(\ga_0)+\op i\,\la(\ga_0))+\frac{1}{2}\norm{\psi}\e+2C}r.
\end{align}
Since $r\leq s(\ga_0)$, it follows from \eqref{eq.VV}, \eqref{eq.YY}, and \eqref{cho} that both $\ga_i^{-1}\xi'$ and  $\ga_0^{-1}\ga_i^{-1}\xi'$ belong to $\cal O_{\e/(8\kappa)}(\eta,p)$.
Since, $\ga_i^{-1}\xi', \ga_0^{-1}\ga_i^{-1}\xi'\in O_{\e/(4\kappa)}(\ga_i^{-1}p,p)$ by \eqref{oe}, it follows from Lemma \ref{lem.shadow1} that
\begin{align*}
\norm{\beta_{\ga_i^{-1}\xi'}(\ga_i^{-1}p, p)-\beta_{\ga_0^{-1}\ga_i^{-1}\xi'}(\ga_i^{-1}p,p)}<2\kappa(\e/4\kappa)=\e/2.
\end{align*}

Now we have
\begin{align*}
&\norm{\beta_{\xi'}(p,\ga_i\ga_0\ga_i^{-1}p)-\la(\ga_0)}\\
&\leq \norm{\beta_{\xi'}(\ga_i p,\ga_i\ga_0p)-\la(\ga_0)}+\norm{\beta_{\xi'}(p,\ga_i p)-\beta_{\xi'}(\ga_i\ga_0\ga_i^{-1}p,\ga_i\ga_0p)}\\
&=\norm{\beta_{\ga_i^{-1}\xi'}(p,\ga_0p)-\la(\ga_0)}+\norm{\beta_{\ga_i^{-1}\xi'}(\ga_i^{-1}p, p)-\beta_{\ga_0^{-1}\ga_i^{-1}\xi'}( \ga_i^{-1}p, p)}\\
&\leq \e/4+\e/2<\e,
\end{align*}
which verifies that $D(\ga_i\xi_0,r)$ belongs to the family $\cal B_R(\ga_0,\e)$.

We now check that $\xi\in D({\ga_i}\xi_0, r)$. 
Since $\gamma_i^{-1}\xi\to \xi_0$, we may  assume that for all $i$,
\be \label{fin1} d_p(\xi_0,\ga_i^{-1}\xi)<\frac{1}{3N_0}e^{-\norm{\psi}\e}r. \ee
Since $r\leq s(\ga_0)$, \eqref{cho}, \eqref{oe}, and \eqref{fin1} imply that $\ga_i^{-1}\xi\in \cal O_{\e/(4\kappa)}(\ga_i^{-1}p,p)$.
Since $\xi_0\in \cal O_{\e/(4\kappa)}(\ga_i^{-1}p,p)$ as well, we have
$$
\norm{\beta_{\ga_i^{-1}\xi}(\ga_i^{-1}p,p)-\underline a(\ga_i^{-1}p,p)}\leq \e/4 \text{ and }\norm{\beta_{\xi_0}(\ga_i^{-1}p,p)-\underline a(\ga_i^{-1}p,p)}\leq \e/4,
$$
by Lemma \ref{lem.shadow1}.
Note that
\begin{align*}
&d_p(\ga_i \xi_0,\xi)=d_{\ga_i^{-1} p} ( \xi_0,\ga_i^{-1}\xi)\\
&= e^{-\psi (\beta_{\xi_0}(\ga_i^{-1}p,p)+\op i\beta_{\ga_i^{-1}\xi}(\ga_i^{-1}p, p))} d_p(\xi_0,\ga_i^{-1}\xi)\\
&\le e^{-\psi(\underline a(\ga_i^{-1}p,p)+\op i\underline a(\ga_i^{-1}p,p))+\frac{1}{2}\norm{\psi}\e}d_p(\xi_0,\ga_i^{-1}\xi)\\
&\le \frac{1}{3N_0} e^{-\psi(\underline a(\ga_i^{-1}p,p)+\op i\underline a(\ga_i^{-1}p,p))} r\text{ by \eqref{fin1}}.
\end{align*}
This proves that $\xi\in D(\ga_i\xi_0,r)$.
\end{proof}

Consider the following measure $\nu_p=\nu_{\psi, p}$
 on $\La$: $$d\nu_p (\xi)=e^{\psi(\beta_{\xi}(o, p))} d\nu_\psi(\xi).$$

\begin{prop}\label{lem.sh} Let $B\subset \mathcal F$ be a Borel subset with $\nu_p(B)>0$.
Then for $\nu_p$-a.e. $\xi\in B$,
$$
\lim\limits_{R\to 0} \sup\limits_{\xi\in D, D\in\cal B_R(\ga_0, \e)}\frac{\nu_p(B\cap D)}{\nu_p(D)}  =1.
$$
\end{prop}
\begin{proof}
For a given Borel function $h: \cal F \to\bb R$, we define $h^* : \cal F\to\bb R$ by
$$
h^*(\xi)=\lim\limits_{R\to0}\sup\limits_{\xi\in D,D\in\cal B_R(\ga_0, \e)}\frac{1}{\nu_p(D)}\int_D h\,d\nu_p.
$$
By Lemma \ref{lem.WD1}, $h^*$ is well defined on $\Lambda_M$. Since $\Lambda_M$ has a full $\nu_p$ measure  by Theorem \ref{fullm}, $h^*$ is defined $\nu_p$-a.e. on $\cal F$.
We will prove that $h=h^*$, $\nu_p$-a.e.; by taking $h=\mathbf{1}_{B}$, the conclusion of the lemma will follow.
Note that $h=h^*$ when $h$ is continuous.
To deal with the general case, we proceed as follows.

\noindent
\textbf{Step 1:} For all $\alpha>0$,
$$
\nu_p(\{h^*>\alpha\})\leq \frac{e^{\psi(\lambda(\ga_0))+\norm{\psi}\e}}{\alpha}\int_{\cal F}|h|\,d\nu_p.
$$

Letting $Q$ be an arbitrary compact subset of $\{\xi: h^*(\xi)>\alpha\}$, it suffices to show that
$$
\nu_p(Q)\leq \frac{e^{\psi(\lambda(\ga_0))+\norm{\psi}\e}}{\alpha}\int_{\cal F}|h|\,d\nu_p.
$$
Fix $R>0$. By definition, for each $x\in Q$, there exists $D_x\in\cal B_{R}(\ga_0, \e)$ containing $x$ such that
$$
\frac{1}{\nu_p(D_x)}\int_{D_x}h\,d\nu_p>\alpha.
$$
Since $K$ is compact, there exists a finite subcover of $\{D_x:x\in Q\}$, say $D_i=\bb B_p(\ga_i\xi_0,
s_i)(i=1,\cdots,n)$ where $\ga_i\in\Ga$ and $s_i=\frac{1}{3N_0} e^{-\psi(\underline a(\ga_i^{-1}p,p)+\op i\underline a(\ga_i^{-1}p,p))}r_i$ for some $0<r_i<R$.

For brevity, we will write $3N_0D_i:=\bb B_p(\ga_{i}\xi_0,3N_0s_{i})$.
By Lemma \ref{inc}, there exists a disjoint subcollection $\{D_{i_1},\cdots,D_{i_\ell}\}$
such that  $$
\bigcup_{k=1}^n  D_k \subset \bigcup_{j=1}^\ell 3N_0D_{i_j}.
$$

Now we claim that $3N_0D_{i_j} \subset \ga_{i_j}{\ga_0}^{-1}\ga_{i_j}^{-1}D_{i_j}$: note that for $\xi\in 3N_0D_{i_j}$,
\begin{align*}
d_p(\ga_{i_j} \xi_0,\ga_{i_j}{\ga_0}\ga_{i_j}^{-1}\xi)&=d_{\ga_{i_j}{\ga_0}^{-1}\ga_{i_j}^{-1}p}(\ga_{i_j} \xi_0,\xi)\\
&=e^{- \psi(\beta_{\ga_{i_j} \xi_0}(\ga_{i_j}{\ga_0}^{-1}\ga_{i_j}^{-1}p,p)+\op i\beta_\xi(\ga_{i_j}{\ga_0}^{-1}\ga_{i_j}^{-1}p,p))}d_{p}(\ga_{i_j} \xi_0,\xi)\\
&\leq  3N_0e^{- \psi(\lambda(\ga_0)+\op i\lambda(\ga_0))+\norm{\psi}\e}s_{i_j}<s_{i_j},
\end{align*}
by \eqref{eq.id3}, \eqref{ggg}, \eqref{eq.nbd1} and \eqref{ggg00}.
Hence
\begin{align*}
\nu_p(3N_0D_{i_j})&\leq \nu_p(\ga_{i_j}{\ga_0}^{-1}\ga_{i_j}^{-1}D_{i_j})\\
&=\int_{D_{i_j}} e^{\psi(\beta_\xi(e,\ga_{i_j}\ga_0\ga_{i_j}^{-1}))}\,d\nu_p(\xi)\\
&\leq e^{\psi(\lambda(\ga_0))+\norm{\psi}\e}\nu_p(D_{i_j}),
\end{align*}
where the last inequality follows from \eqref{eq.nbd1}.
Therefore,
\begin{align*}
\nu_p(Q) &\leq\sum_{j=1}^\ell\nu_p(3N_0D_{i_j})\leq \sum_{j=1}^\ell e^{\psi(\lambda(\ga_0))+\norm{\psi}\e}\nu_p(D_{i_j})\\
&\leq \frac{ e^{\psi(\lambda(\ga_0))+\norm{\psi}\e}}{\alpha}\sum_{j=1}^\ell \int_{D_{i_j}}h\,d\nu_p\leq \frac{ e^{\psi(\lambda(\ga_0))+\norm{\psi}\e}}{\alpha}\int_{\cal F}|h|\,d\nu_p,
\end{align*}
which was to be proved. 

\noindent
\textbf{Step 2:} $h(\xi)=h^*(\xi)$ for $\nu_p$-a.e $\xi$.

We first prove that $h(\xi)\leq h^*(\xi)$ for $\nu_p$-a.e $\xi$.
Let $\alpha>0$ be arbitrary.
It suffices to show that $\nu_p(\{\xi:h(\xi)-h^*(\xi)>\alpha\})=0$.
Let $h_n$ be a continuous function converging to $ h$ in $L^1(\nu_p)$.
Note that $h_n^*=h_n$ and
\begin{align*}
&\nu_p(\{\xi: h(\xi)-h^*(\xi) >\alpha\})\\
&\leq \nu_p(\{\xi: h(\xi)-h_n(\xi) >\alpha/2\})+\nu_p(\{\xi: h_n^*(\xi)-h^*(\xi) >\alpha/2\})\\
&\leq \tfrac{2}{\alpha}\norm{h-h_n}_1+\tfrac{2}{\alpha}e^{\psi(\lambda(\ga_0))+\norm{\psi}\e}\norm{h-h_n}_1.
\end{align*}
Taking $n\to\infty$, we get 
$$\nu_p(\{\xi:h(\xi)-h^*(\xi)>\alpha\})=0.$$
As $\alpha>0$ is arbitrary, it follows that $h\leq h^*$, $\nu_p$-a.e.
A similar argument shows that $h^*\leq h$, $\nu_p$ -a.e. 
\end{proof}

\noindent{\bf Proof of Proposition \ref{ess}:}
It is easy to check that $ {\mathsf E}_{\nu}(\Ga_0)={\mathsf E}_{\nu_p}(\Ga_0)$.
Hence it suffices to show $\la(\ga_0)\in {\mathsf E}_{\nu_p}(\Ga_0)$.
 Let $B\subset \cal F$ be a Borel subset with $\nu_p(B)>0$ and $\e>0$.
By Proposition \ref{lem.sh}, there exists $D=\bb B_p(\ga \xi_0,r)\in \cal B_R(\ga_0, \e)$ for $\ga\in \Ga$ and $r>0$ such that
\be\label{nu1}
\nu_p(D\cap B)> (1+e^{-\psi(\lambda(\ga_0))-\norm{\psi}\e})^{-1}\nu_p(D).
\ee

Since $r<r_p(\gamma)$,  we have
\begin{align*}
D&\subset \{\xi:\norm{\beta_\xi(p,\ga{\ga_0}^\pm\ga^{-1}p)\mp\lambda(\ga_0)}\le \e\}\\
&\subset
 \{\xi: |{\psi (\beta_\xi(p,\ga{\ga_0}^\pm\ga^{-1}p))\mp\psi(\lambda(\ga_0)})|  \le \|\psi\| \e\}
.
\end{align*}
We note that $\ga{\ga_0}\ga^{-1} D\subset D$: if $\xi\in D$, by \eqref{eq.id3},
\begin{align*}
d_p(\ga \xi_0,\ga{\ga_0}\ga^{-1}\xi)&=d_{\ga{\ga_0}^{-1}\ga^{-1}p}(\ga \xi_0,\xi)\\
&=e^{\psi(\beta_{\ga \xi_0}(p,\ga\ga_0^{-1}\ga^{-1}p)+\op i\beta_\xi(p,\ga\ga_0^{-1}\ga^{-1}p))}d_p(\ga \xi_0,\xi)\\
&\leq e^{-\psi(\lambda(\ga_0)+\op i\lambda(\ga_0))+\norm{\psi}\e}r<r.
\end{align*}

Since
$$
B\cap\ga{\ga_0}\ga^{-1}B\cap\{\xi:\norm{\beta_\xi(p,\ga{\ga_0}\ga^{-1}p)-\lambda(\ga_0)}<\e\}\supset (D\cap B)\cap \ga{\ga_0}\ga^{-1}(D\cap B),
$$
it suffices to prove that $(D\cap B)\cap \ga{\ga_0}\ga^{-1}(D\cap B)$ has a positive $\nu_p$-measure.
Note that
\begin{align*}
\nu_p(\ga{\ga_0}\ga^{-1}(D\cap B))&=\int_{D\cap B}e^{\psi(\beta_{\xi}(p,\ga{\ga_0}^{-1}\ga^{-1}p))}\,d\nu_p(\xi)\\
&\geq e^{-\psi(\lambda(\ga_0))-\norm{\psi}\e}\nu_p(D\cap B) .\end{align*}
Hence by \eqref{nu1}, 
\begin{align*}
&\nu_p(D\cap B)+\nu_p(\ga{\ga_0}\ga^{-1}(D\cap B) )> (1+ e^{-\psi(\lambda(\ga_0))-\norm{\psi}\e}) \nu_p(D\cap B)
>\nu_p(D).
\end{align*}
Since both $D\cap B$ and $\ga{\ga_0}\ga^{-1}(D\cap B)$ are contained in $D$, this implies
that  their intersection has a positive $\nu_p$-measure.
Since $\ga \ga_0\ga^{-1}\in \Ga_0$, it follows that $\la(\ga_0)\in \mathsf E_{\nu_p}(\Ga_0)$.

In view of Proposition \ref{prop.erg1}, we obtain the following corollary:
\begin{cor}\label{normal}
Let $\Ga_0$ be a Zariski dense normal subgroup of an Anosov subgroup $\Ga<G$.
Let $\psi\in \dg$. If $\nu_\psi$ is $\Ga_0$-ergodic, then
$m_{\psi}^{\BR}$, considered as a measure on $\Ga_0\ba G$, is $NM$-ergodic.
\end{cor}

\subsection*{Patterson Sullivan measures are mutually singular}

\begin{thm}\label{thm.MS}
Let $\Ga<G$ be an Anosov subgroup.
Then 
$\{\nu_\psi:\psi\in\dg\}$ are pairwise mutually singular.
\end{thm}
\begin{proof}
Since $\Ga<G$ is Anosov, the family $\{\nu_\psi:\psi\in\dg\}$ consists of $\Gamma$-ergodic measures
(see the remark following Theorem \ref{pop}).
Hence any $\nu_{\psi_1}$ and $\nu_{\psi_2}$ in this family are either mutually singular or absolutely continuous  to each other.
Now the claim follows from Lemma \ref{lem.sing} below, in view of Proposition \ref{posm}.
\end{proof}

\begin{lemma}\label{lem.sing}
For $i=1,2$, let $\nu_{\psi_i}$ be a $(\Ga,\psi_i)$-$\PS$ measure for some
$\psi_i\in \fa^*$. If ${\mathsf E}_{\nu_{\psi_2}}=\mathfrak a$ and $\nu_{\psi_1}\ll\nu_{\psi_2}$, then $\psi_1=\psi_2$.
\end{lemma}
\begin{proof} Suppose that $\nu_{\psi_1}\ll\nu_{\psi_2}$ and that $\psi_1\ne \psi_2$.
Consider the Radon-Nikodym derivative $f:=\frac{d\nu_{\psi_1}}{d\nu_{\psi_2}}\in L^1(\Lambda,\nu_{\psi_2})$.
Note that there exists a $\nu_{\psi_2}$-conull set $E\subset\La$ such that for all $\xi\in E$ and $\ga\in\Ga$, we have
\begin{equation}\label{eq.FE1}
f(\ga^{-1}\xi)=e^{(\psi_1-\psi_2)(\beta_\xi(e,\ga))}f(\xi).
\end{equation}

If $f$ were continuous, then $f\ne 0$ and by applying $\xi=y_\ga$ in the above,
we get $\psi_1(\la(\ga))=\psi_2(\la (\ga))$ for all $\gamma\in \Ga$. Since $\la(\Ga)$ generates a
dense subgroup of $\fa$, it follows that $\psi_1=\psi_2$.

In general, we use the hypothesis $\mathsf E_{\nu_{\psi_2}}=\mathfrak a$.
Choose $0<r_1<r_2$ such that
$$
B:=\{\xi\in\La:r_1<f(\xi)<r_2\}
$$
has a positive $\nu_{\psi_2}$-measure.
Since $\psi_1\neq\psi_2$, we can choose $w\in\mathfrak a$ such that 
\begin{equation}\label{eq.ratio}
e^{(\psi_1-\psi_2)(w)}>\tfrac{2r_2}{r_1}.
\end{equation}
Choose $\e>0$ such that $e^{\norm{\psi_1-\psi_2}\e}<2$.
Since $\nu_{\psi_2}(B)>0$ and ${\mathsf E}_{\nu_{\psi_2}}=\mathfrak a$, there exists $\ga\in\Ga$ such that
$$
B':=B\cap\ga B\cap\{\xi\in\La : \norm{\beta_\xi(e,\ga)-w}<\e\}
$$
has a positive $\nu_{\psi_2}$-measure.
Now note that
\begin{align*}
\int_{B'}f(\ga^{-1}\xi)\,d\nu_{\psi_2}(\xi)&>e^{(\psi_1-\psi_2)(w)-\norm{\psi_1-\psi_2}\e}\int_{B'}f(\xi)\,d\nu_{\psi_2}(\xi)\\
&>\tfrac{r_2}{r_1}\int_{B'}f(\xi)\,d\nu_{\psi_2}(\xi)
\end{align*}
by \eqref{eq.FE1}, \eqref{eq.ratio}, and the choice of $\e$.
In particular,
$$
\nu_{\psi_2} \left\{\xi\in B': f(\ga^{-1}\xi)>\tfrac{r_2}{r_1}f(\xi)\right\}>0.
$$
It follows that there exists $\xi\in B'\cap E$ such that
\be\label{bpr}
f(\ga^{-1}\xi)>\tfrac{r_2}{r_1}f(\xi).
\ee
On the other hand, for $\xi\in B'$, both $\xi$ and $\ga^{-1}\xi$ belong to $ B$.
Hence, by definition of $B$, for all $\xi\in B'$, we have
$$
f(\ga^{-1}\xi)<\tfrac{r_2}{r_1}f(\xi).
$$
This is a contradiction to \eqref{bpr}.
\end{proof}

\subsection*{$P$-semi-invariant measures.}
In this section, we establish that 
$P$-semi-invariant Radon measures supported in $\cal E=\{x\in\Ga\ba G : x^+\in\La\}$, up
to constant multiples, are parametrized by $\dg$.

If $\mu$ is $P$-semi-invariant, then there exists a linear form $\chi_\mu\in \mathfrak a^*$ such that for all $a\in A$,
$$
a_*\mu=e^{-\chi_\mu 
(\log a) }\mu.
$$

We set $\psi_\mu:=\chi_\mu+2\rho\in\fa^*$.
The first part of the following proposition is known in the rank one case (see e.g. \cite{Bab}, \cite{Bu}, and \cite{LL}) and the proof can be easily adapted to the higher rank case.
\begin{proposition}\label{prop.NMA}  For any Zariski dense discrete subgroup $\Ga<G$,
any $P$-semi-invariant Radon measure $\mu$  on $\Ga\ba G$ is proportional to $m_{\nu_{\psi_\mu},m_o}$ where $\nu_{\psi_\mu}$ is a $(\Ga,\psi_\mu)$-conformal measure and $\psi_\mu\in D_\Ga$.
Moreover, if 
 $\mu$ is supported on $\cal E$, then  $\mu$ is proportional to $m_{\psi_\mu}^{\BR}$.
 If $\Ga$ is Anosov, we also have $\psi_\mu\in\dg$.
\end{proposition}

\begin{proof} 
For simplicity, set $\chi=\chi_\mu$ and $\psi=\psi_\mu.$
Let $\tilde\mu$ be the $\Ga$-invariant lift of $\mu$ to $G$ and $\pi : G\to G/P $ be the projection.
Choose a section $c : G/P \to K$ so that $\pi\circ c=\op{id}$ and consider the measurable isomorphism
\begin{equation*}
 \begin{array}{ccc}
G/P \times M\times A\times N  & \to &G \\
(\xi,\enspace m,\enspace a,\enspace n)&\to  & c(\xi)man.
     \end{array}
\end{equation*}
Let 
$dm$, $dn$, $da$ be the Haar measures on $M$, $N $, and $A$.
As $\tilde\mu$ is 
a $P$-semi-invariant Radon measure, there exist $\tilde\chi\in \fa^*$ and a Radon measure $\nu$ on $G/P$ such that 
$$
d\tilde\mu(c(\xi)man)=e^{\tilde\chi(\log a)}dn\,da\,dm\,d\nu(\xi).
$$
Without loss of generality, we may assume that $|\nu|=1$.
Because $d\tilde \mu(\cdot \,a)=e^{\chi (\log a)} d\tilde\mu(\cdot)$, we have 
\begin{equation*}
 \chi= \tilde \chi -2\rho,\text{ or equivalently, }\tilde\chi= \psi.
\end{equation*}

Note that $G$ is measurably isomorphic to the product $G/P \times P $ and 
the left $\Ga$-action with respect to these coordinates is given by
$\ga \cdot(\xi,p)=(\ga\cdot\xi,\Phi(\ga,\xi)p)$ for some $P $-valued cocycle $\Phi : \Ga\times G/P \to P $
 where $\ga\in\Ga$ and $(\xi,p)\in G/P \times P $.
One can check that $$\Phi(\ga,\xi)=m(\ga,\xi)\exp(\beta_\xi(\ga^{-1},e))n(\ga,\xi)$$ for some $m(\ga,\xi)\in M$ and $n(\ga,\xi)\in N $.
Hence, for $p=man$, the $MAN $-coordinates for $\Phi(\ga,\xi)p$ are given by
\begin{equation}\label{eq.coord1}
\Phi(\ga,\xi)p=\big(m(\ga,\xi)m\big)\big(\exp(\beta_\xi(\ga^{-1},e))a\big)\big((ma)^{-1}n(\ga,\xi)man\big).
\end{equation}
Since $\tilde\mu$ is left-$\Ga$-invariant, we have for any $f\in C_c(G)$ and any $\ga\in\Ga$,
\begin{align*}
\int_G f(g)\,d\tilde\mu(g)&=\int_G f(g)\,d(\ga_*\tilde\mu)(g)\\
&=\int_{G/P }\int_{P } f((\ga\xi,\Phi(\ga,\xi)p)e^{ \psi(\log a)}\,dn\,da\,dm\,d\nu(\xi)\\
&=\int_{G/P }\int_{P } f(\xi,p)e^{ \psi\big(\log a-\beta_{\ga^{-1}\xi}( \ga^{-1},e)\big)}\,dn\,da\,dm\,d(\ga_*\nu)(\xi),
\end{align*}
where in the last equality, we have used \eqref{eq.coord1} and the change of variables $a'=a\exp(\beta_\xi(e,\ga^{-1}))$.
On the other hand, we have
\begin{equation*}
\int_G f(g)\,d\tilde\mu(g)=\int_{G/P }\int_{P } f(\xi,p)e^{ \psi(\log a)}dn\,da\,dm\,d\nu(\xi).
\end{equation*}
By comparing these two identities, we get that for any $\ga\in\Ga$,
$$d(\ga_*\nu)(\xi)=e^{\psi(\beta_\xi(e,\ga))}d\nu(\xi),$$
that is, $\nu$ is a $(\Ga,\psi)$-conformal measure.
By \cite[Thm. 8.1]{Quint2}, $\psi\in D_\Ga$.

Finally, recall that for all $g\in G$ and $\phi\in C_c(G)$,
\begin{equation*}
\int_{N } \phi(gn)\,dn=\int_{G/P }\phi(gn)e^{2\rho(\beta_{gn^-}(e,gn))}dm_o(gn^-).
\end{equation*}
For $g=c(\xi)man\in KAN$, we have $\beta_{g^+}(e,g)=\log a$ and $g^+=\xi$.
Hence, for any $f\in C_c(G)$,
\begin{align*}
\int_Gf(g)d\tilde\mu(g)&=\int_{G/P }\int_{P } f(c(\xi)man)e^{ \psi(\log a)}dn\,da\,dm\,d\nu(\xi)\\
&=\int_{G/M}\int_M f(g)e^{2\rho(\beta_{g^-}(e,g))}e^{\psi(\beta_{g^+}(e,g))}\,dm\,da\,dm_o(g^-)\,d\nu(g^+)\\
&= \tilde m_{\nu,m_o}(f).
\end{align*}
 Therefore $\tilde\mu (f) = \tilde m_{\nu,m_o}(f)$.

 Now, if $\mu$ is supported on $\cal E$, then $\nu$ is supported on $\La$.
Hence $ \nu$ is a $(\Ga,\psi)$-$\PS$ measure; so $ \mu= m_\psi^{\BR}$. 
When $\Ga$ is Anosov, $\psi\in\dg$ by Theorem \ref{thm.bij}.
\end{proof}
Let $\cal P_\Ga$ be the space of all $P$-semi-invariant Radon measures on $\cal E$ up to proportionality.
Let $ {\cal Q}_\Ga$ be the space of all $NM$-invariant, ergodic and $A$-quasi-invariant Radon measures supported on $\cal E$ up to
proportionality.

\begin{thm} \label{thm.MC}
Let $\Gamma<G$ be an Anosov subgroup. 
We have $\cal P_\Ga=\cal Q_\Ga$ and the map $\dg \to \cal Q_\Ga$ given by
$\psi\mapsto [m_\psi^{\BR}]$ is a homeomorphism between $\dg$ and $\cal Q_\Ga$. In particular,
$\cal Q_\Ga$ is homeomorphic to $\br^{\text{rank }G-1}$.
\end{thm}
\begin{proof} For $\mu\in \cal Q_\Ga$ and $a\in A$, $a_*\mu$ and $\mu$
are equivalent to each other, and  by the $NM$-ergodicity of $\mu$,
the Radon-Nikodym derivative $da_*\mu/d\mu$ is constant, say $\chi(a)$.
 Now the function $a\mapsto \chi(a)$ gives the semi-invariance of $\mu$ by $A$ and hence by $P$.
This implies $\cal Q_\Ga\subset\cal P_\Ga$.
The other direction $\cal P_\Ga\subset\cal Q_\Ga$ follows from Proposition \ref{prop.NMA} and Theorem \ref{thm.BR}.

Let ${\cal Q}_\Ga^\spadesuit$ be the space of all $NM$-ergodic $A$-quasi-invariant Radon measures supported on $\{x\in\Ga\ba G : x^+\in\La\}$, so that $\cal Q_\Ga={\cal Q}_\Ga^\spadesuit/\sim$.
Set $\iota(\psi)=m_\psi^{\BR} $ for $\psi\in \dg$.  Since $m_\psi^{\BR}$ is $NM$-ergodic by Theorem \ref{thm.BR}, the map $\iota:\dg\to  {\cal Q}_\Ga^\spadesuit$ is well defined and injective  by Lemma \ref{lem.sing}.
By Proposition \ref{prop.NMA},
$\iota(\dg)$ contains  precisely one representative of each class in $\cal Q_\Ga$.
Hence it suffices to show that the map $\iota$ gives a homeomorphism between $\dg$ and its image
$\iota(\dg)$.
Continuity of $\iota$ follows from Theorem \ref{thm.bij}.
Now, suppose that $m_{\psi_i}^{\BR}\to  m_\psi^{\BR}$ for some sequence $\psi_i, \psi\in \dg$.
Then the $A$-semi-invariance of the BR-measures given by \eqref{eq.SI} and the convergence $a_* m_{\psi_i}^{\BR}\to a_* m_\psi^{\BR}$ implies that $\lim_{i\to \infty} e^{(2\rho -\psi_i)(\log a)}m_{\psi_i}^{\BR}(f)
= e^{(2\rho -\psi)(\log a)} m_{\psi}^{\BR}(f)$ for all $f\in C_c(\Gamma\ba G)$.
Since  $\lim_{i\to \infty}m_{\psi_i}^{\BR}(f)
=  m_{\psi}^{\BR}(f)$, we get
$\lim_{i\to \infty} e^{(2\rho -\psi_i)(\log a)}= e^{(2\rho -\psi)(\log a)}$ for all $a\in A$.
Hence $\psi_i\to \psi$. This proves that $\dg$ and $\cal Q_\Ga$ are homeomorphic to each other. The last claim follows from Proposition \ref{homeo}.
\end{proof}

\end{document}